\documentclass[11pt]{elsarticle}
\usepackage[toc,page]{appendix}

    \makeatletter
\def\ps@pprintTitle{%
	\let\@oddhead\@empty
	\let\@evenhead\@empty
	\let\@oddfoot\@empty
	\let\@evenfoot\@oddfoot
}
\makeatother

\usepackage{amsmath,amssymb,amsthm,bm}
\usepackage{caption,subcaption,cprotect}

\newtheorem{theorem}{Theorem}[section]

\theoremstyle{definition}

\usepackage{hyperref}
\usepackage{url,soul,booktabs,cleveref}
\Crefname{equation}{Eq.}{Eqns.}

\usepackage[font=small,labelfont=bf]{caption}
\usepackage{appendix}

\newcommand{\revise}[1]{{#1}}

\usepackage{makecell}

\usepackage{etoolbox}
\providetoggle{comments_on}
\settoggle{comments_on}{true}
\iftoggle{comments_on}{
 \usepackage[left=1.75in,right=2.5in,marginparwidth=2.2in,paperwidth=9.2in]{geometry}
    \usepackage[commentmarkup=margin]{changes}
    \definechangesauthor[color=red]{ed}      
    
    \newcommand{\ed}[1]{\textcolor{red}{\scriptsize #1}}
    
}{
    \newcommand{\ed}[1]{}    
}
\usepackage{multirow}

\newtheorem{remark}{Remark}
\newtheorem{example}{Example}

\usepackage{pgfplots}
\pgfplotsset{compat=newest}

\newcommand{\bx}[0]{\mathbf{x}}

\newcommand{\bB}[0]{\mathbf{B}}
\newcommand{\RR}[0]{\mathbb{R}}

\usepackage{changes}
\begin{document}

\begin{abstract}
Deep neural networks (DNN) have been used to model nonlinear relations between physical quantities. Those DNNs are embedded in physical systems described by partial differential equations (PDE) and trained by minimizing a loss function that measures the discrepancy between predictions and observations in some chosen norm. This loss function often includes the PDE constraints as a penalty term when only sparse observations are available. As a result, the PDE is only satisfied approximately by the solution. However, the penalty term typically slows down the convergence of the optimizer for stiff problems. We present a new approach that trains the embedded DNNs while numerically satisfying the PDE constraints. We develop an algorithm that enables differentiating both explicit and implicit numerical solvers in reverse-mode automatic differentiation. This allows the gradients of the DNNs and the PDE solvers to be computed in a unified framework. We demonstrate that our approach enjoys faster convergence and better stability in relatively stiff problems compared to the penalty method. Our approach allows for the potential to solve and accelerate a wide range of data-driven inverse modeling, where the physical constraints are described by PDEs and need to be satisfied accurately.   
\end{abstract}


\begin{frontmatter}

\title{Physics Constrained Learning for Data-driven Inverse Modeling from Sparse Observations}

\author[rvt1]{Kailai~Xu}
\ead{kailaix@stanford.edu}

\author[rvt1,rvt2]{Eric~Darve}
\ead{darve@stanford.edu}

\address[rvt1]{Institute for Computational and Mathematical Engineering, Stanford University, Stanford, CA, 94305}
\address[rvt2]{Mechanical Engineering, Stanford University, Stanford, CA, 94305}

\end{frontmatter}

\section{Introduction}
Models involving partial differential equations (PDE) are usually used for describing physical phenomena in science and engineering. Inverse problems \cite{isakov2006inverse} aim at calibrating unknown parameters in the models based on observations associated with the output of the models. In real-world inverse problems, the observed data are usually ``sparse'' because only part of the model outputs are observable. Examples of inverse modeling from sparse observations include learning subsurface properties from seismic data on the earth's surface \cite{virieux2017introduction} and learning constitutive relations from surface deformations of 3D solid materials \cite{grediac2012full}. The mapping from partial observations to unknown parameters is usually indirect and hard to compute, and in some cases, the inverse problems are ill-conditioned \cite{akccelik2006parallel}. Recently, neural networks have been applied to learn unknown relations in inverse modeling \cite{huang2019learning,tartakovsky2018learning,meng2019composite}, which is a more challenging task since training the neural network requires differentiating both neural network and PDEs for extracting gradients.

Mathematically, we can formulate the inverse problem as a PDE-cons\-trained optimization problem
\begin{gather}\label{equ:inverse}
    \min_{\theta}\ L(u) = \sum_{i\in \mathcal{I}_{\mathrm{obs}}} \left(u(\bx_i) - u_i\right)^2 \\
    \mathrm{s.t.}\;\; F(\theta, u) = 0 \notag
\end{gather}
where $L$ is the \textit{loss function} that measures the discrepancy between estimated outputs $u$ and observed outputs $u_i$ at locations $\{\bx_i\}$. $\mathcal{I}_{\mathrm{obs}}$ is the set of indices of locations where observations are available. $F$ is the physical constraint, usually described by a system of PDEs, and $\theta$ is the finite or infinite dimensional unknown parameter.  In the case that $\theta$ is infinite dimensional, e.g., the unknown is a function, we can approximately represent the function with a functional form such as linear combination of basis functions or a neural network, in which case $\theta$ can be interpreted as coefficients of the functional form because there is a one-to-one correspondence between the function and its coefficients. For example, the constraint may have the form
\begin{equation*}
    F(\theta(u), u) = 0
\end{equation*}
where $\theta(u)$ is a function that maps $u$ to a scalar or vector value. We use a neural network to approximate unknown functions, and we interpret $\theta$ as weights and biases in the neural network. 

The physical constraints $F(\theta, u)=0$ is usually solved with numerical schemes. Let the discretization scheme be $F_h(\theta, u_h) = 0$ where $u_h$ is the numerical solution. We assume that the gradients $\frac{\partial F_h}{\partial \theta}$ and $\frac{\partial F_h}{\partial u_h}$ exist and are continuous, which is usually the case for PDE systems. For example, in the finite difference method for numerically solving PDE constraints $F(\theta, u)=0$, the gradients can be computed with adjoint state methods \cite{plessix2006review,bradley2010pde,leung2006adjoint,allaire2015review,lauss2018discrete}. In the context of \Cref{equ:inverse}, the sparse observations $\{u_i\}_{i\in \mathcal{I}_{\mathrm{obs}}}$ are usually given as point values; therefore it is most appropriate to work with a strong form partial differential equation where we can read out function values $u(\mathbf{x}_i)$ directly from numerical solutions to $F_h(\theta,u_h)=0$ (in weak form, point-wise values are not defined). Additionally, to compute and store the Jacobian efficiently (used for solving the nonlinear PDEs), we assume that the basis functions in the numerical scheme are local, so that the Jacobian is sparse. Many standard numerical schemes, such as finite difference methods \cite{thomas2013numerical,noye1984finite} and iso-geometric analysis \cite{hughes2005isogeometric,auricchio2010isogeometric}, deal with strong form of PDEs and the discretization scheme is local. 

The notion of sparse observations is not only important in practice but also has implications for computational methods. For example, if given full solution $u_{\mathrm{full}}$, we can estimate $\theta$ by minimizing the residual
\[
    \min_{\theta} \|F_h(\theta, u_{\mathrm{full}})\|_2^2
\]
This is not possible for sparse observations; instead, we must solve for $u_h$ from $F_h(\theta, u_h)=0$ and then compare $u_h$ with observations at locations $\{\bx_i\}_{i\in \mathcal{I}_{\mathrm{obs}}}$. 

One popular approach for solving \Cref{equ:inverse} is by \textit{weakly} enforcing the physical constraints by adding a penalty term to the discretized equations \cite{van2015penalty}
\begin{equation}\label{equ:tl}
    \tilde L_h(\theta, u_h) = L_h(u_h) + \lambda \|F_h(\theta, u_h)\|_2^2
\end{equation}
In this method, both $\theta$ and $u_h$ become independent variables in the unconstrained optimization. The gradients $\frac{\partial \tilde L_h(\theta, u_h)}{\partial \theta}$ and $\frac{\partial \tilde L_h(\theta, u_h)}{\partial u_h}$ can be computed using an adjoint approach, or automatic differentiation \cite{herzog2010algorithms,paszke2017automatic,abadi2016tensorflow_proc,abadi2016tensorflow_art,van2018automatic,bell2008algorithmic}. The gradients are provided to a gradient-based optimization algorithm for minimizing \Cref{equ:tl}. The physics-informed neural network (PINN) \cite{raissi2019physics,raissi2019deep,yang2018physics,meng2019composite} is an example where the penalty method can be applied. The penalty method is attractive since it avoids solving $F_h(\theta, u_h)=0$ but the drawback is that the physical constraints are not satisfied exactly, and the number of optimization variables increases by including $u_h$ (either the discretized values or coefficients in the surrogate models). The number of extra degrees of freedom can be very large for dynamical problems, in which case $u_h$ is a collection of all solution vectors at each time step. 
    
    Another approach is to enforce the physical constraint by numerically solving the equation $F_h(\theta, u_h)=0$ \cite{roth2013discrete,rees2010optimal,funke2013framework}. The solution $u_h(\bx;\theta)$ depends on $\theta$ and the objective function becomes 
    \begin{equation*}
        \tilde L_h(\theta) = L_h(u_h(\bx;\theta))
    \end{equation*}
The optimal $\theta$ can be found using the gradient function 
\begin{equation}\label{equ:gradient}
     \nabla_{\theta} L_h(u_h(\bx;\theta))
\end{equation}
The challenge is solving the equation $F_h(\theta, u_h)=0$ and computing \Cref{equ:gradient}. If $F_h(\theta, u_h)$ is nonlinear in $u_h$, we need the Jacobian $\frac{\partial F_h}{\partial u_h}$ in the Newton-Raphson method \cite{galantai2000theory}. For computing \Cref{equ:gradient}, the finite difference method has been applied to this kind of problem but suffers from the curse of dimensionality and stability issues \cite{baydin2018automatic,thomas2013numerical}. Adjoint methods \cite{plessix2006review} can be used but can be time-consuming and error-prone to derive and implement. 

We propose \textit{physics constrained learning} (PCL) that enforces the physical constraint by solving $F_h(\theta, u_h)=0$ numerically and uses reverse mode automatic differentiation and forward Jacobian propagation to extract the gradients and Jacobian matrices. The key is computing the gradient of the coupled system of neural networks and numerical schemes efficiently. It is worth mentioning that we optimize the storage and computational cost of Jacobian matrices calculation by leveraging their sparsity. However, if the solution representation is nonlocal, such as that in PINNs, the corresponding Jacobian matrices are dense, and thus enforcing physical constraints by solving the PDE system is challenging. 

Our major finding is that compared to the penalty method, the convergence of PCL is potentially more robust and faster due to fewer iterations for stiff problems, despite each iteration being more expensive. Additionally, PCL does not require selecting a penalty parameter and therefore requires less effort to choose hyper-parameters. The PCL also simplifies the implementation by calculating gradients and Jacobians automatically using the reverse mode automatic differentiation and forward Jacobian propagation techniques. By formulating the observations and loss function $L(u)$ in terms of weak solutions (i.e., an integral of the product of the weak solution and a test function), the PCL can be generalized to numerical schemes that deal with weak form solutions such as finite element analysis. Therefore, the PCL promises to benefit a wide variety of inverse modeling problems. 

We consider several applications of PCL for physics-based machine learning. We compare PCL with the penalty method, and we demonstrate that, in our benchmark problems,
\begin{enumerate}
    \item PCL enjoys faster convergence with respect to the number of iterations to converge to a predetermined accuracy. Particularly, we observe a $10^4$ times speed-up compared with the penalty method in the Helmholtz problem. We also prove a convergence result, which shows that for the chosen model problem, the condition number in the penalty method is much worse than that of PCL. 
    \item PCL exhibits mesh independent convergence, while the penalty me\-thod does not scale with respect to the number of iterations as well as PCL when we refine the mesh.
    \item PCL is more robust to noise and neural network architectures. The penalty method includes the solution $u_h$ as independent variables to optimize, and the optimizer may converge to a nonphysical local minimum.
\end{enumerate}

The outline of the paper is as follows: in \Cref{sect:problem}, we formulate the discrete optimization problem and discuss two different approaches: penalty methods and constraint enforcement methods. In \Cref{sect:num}, we present numerical benchmarks and compare the penalty methods and PCL. Finally, in \Cref{sect:conc}, we summarize our results and discuss the limitations of PCL.

\section{Data-driven Inverse Modeling with Sparse Observations}\label{sect:problem}

\subsection{Problem Formulation and Notation}

When the inverse problem \Cref{equ:inverse} is discretized, $u$ and the output of $F$ will be vectors. We use the notation $u_h$, $L_h$, $F_h$ to denote $u$, $L$, $F$ in the discretized problem. $L_h:\RR^n\rightarrow \RR$, $F_h:\RR^d\times\RR^n\rightarrow \RR^n$ are continuously differentiable functions. The discretized problem can be stated as 
\begin{gather}\label{equ:hhh}
    \min_{\theta}\; L_h(u_h) \\
    \mathrm{s.t.}\;\; F_h(\theta, u_h) = 0 \notag
\end{gather}
For example, if we observe values of $u(x)$, $\{u_i\}_{i\in \mathcal{I}_{\mathrm{obs}}}$, at location $\{\bx_i\}$, we can formulate the loss function with least squares
\begin{equation*}
    L_h(u_h) =  \sum_{i\in \mathcal{I}_{\mathrm{obs}}}\left( u_h(\bx_i) - u_i \right)^2
\end{equation*}
$F_h(\theta, u_h)$ is the residual in the discretized numerical scheme. 

\begin{example}
In finite element analysis, if the system is linear, we have
\begin{equation*}
    F_h(\theta, u_h) = A(\theta) u_h - b
\end{equation*}
where $b$ is the external load vector and $A(\theta)$ is the stiffness matrix of finite element analysis, taking into account the boundary conditions. 
\end{example}

\subsection{Penalty Method}

The penalty method (PM) incorporates the constraint $F_h(\theta,u_h)=0$ into the system by penalizing the term with a suitable norm (e.g., 2-norm in this work). We consider the class of differentiable penalized loss function
\begin{equation*}
    \tilde L_{h,\lambda}(\theta, u_h) :=  L_h(u_h) + \lambda\|F_h(\theta, u_h)\|_2^2
\end{equation*}
where $\lambda\in (0,\infty)$. The term $\lambda$ is called the multiplier of the penalty method and is used to control to what extent the constraint is enforced. We will omit the subscript $\lambda$ when it is clear from the context. 

In the penalty method, the constraint may not be satisfied exactly. However, if the value $\lambda$ is made suitably large, the penalty term $\lambda\|F_h(\theta, u_h)\|_2^2$ will impose a large cost for violating the constraint. Hence the minimization of the penalized loss function will yield a solution with small value in the residual term $F_h(\theta, u_h)$. However, a large $\lambda$ places less weight on the objective function $L_h(u_h)$. A proper choice of $\lambda$ for a desirable trade-off is nontrivial in many cases. 

The penalty method can be visualized in \Cref{fig:first}-left, where $\theta$ and $u_h$ both serve as trainable parameters. In the gradient-based optimization algorithm discussed in the next section, the gradients are back-propagated from the loss function to both $\theta$ and $u_h$ separately.  

\begin{figure}[h]
\centering
  \includegraphics[width=0.8\textwidth]{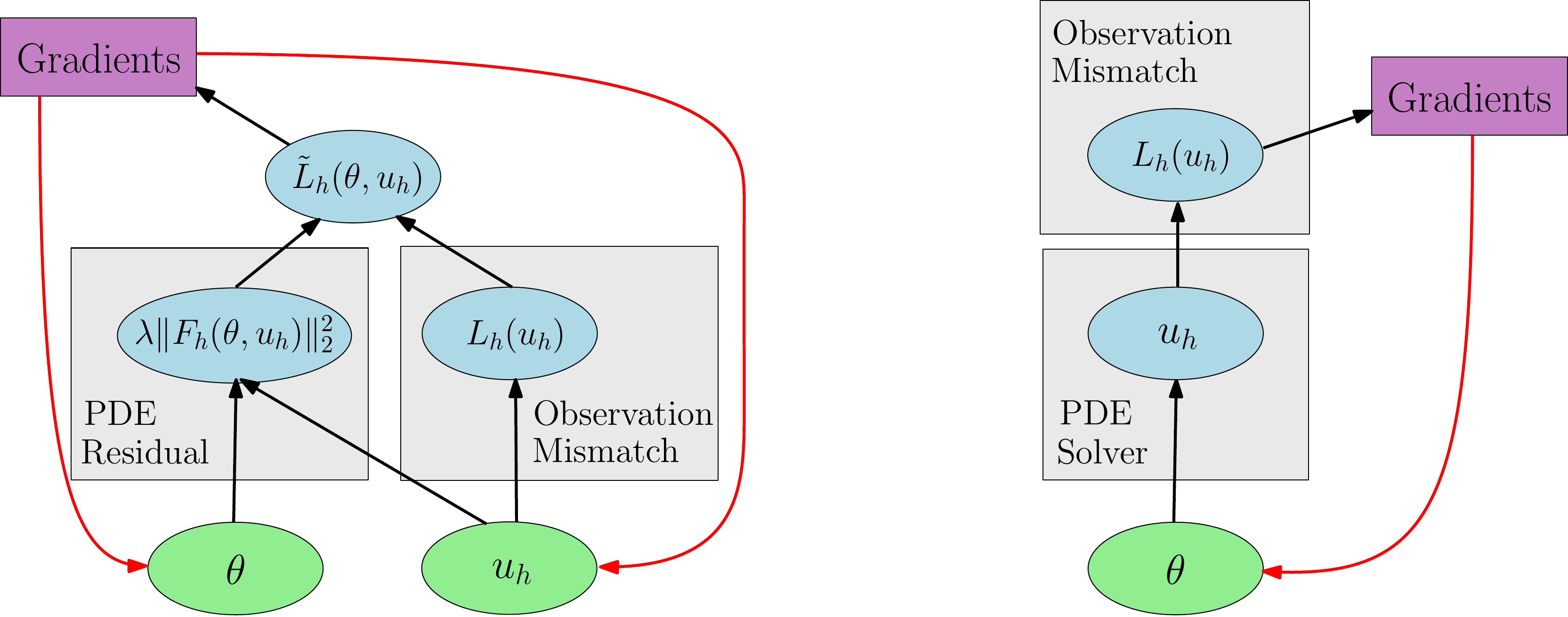}
  \caption{Schematic illustration of penalty method (left) and PCL (right).}
  \label{fig:first}
\end{figure}

The following theorem ensures that the penalty method converges to the optimal point under certain conditions
\begin{theorem}\label{thm:cvg}
    Assume that $L_h$ and $F_h$ are continuously differentiable functions, $\lambda\in (0,\infty)$ and the feasible set
    \begin{equation*}
        S = \{u_h\in \RR^n: F_h(\theta, u_h)=0 \}
    \end{equation*}
    is not empty. Assume $u^k$ is the global minimum of $\tilde L_{h,\lambda_k}(\theta, u_h)$ at step $k$ of the algorithm, and $\lambda_k \to \infty$ as $k\rightarrow \infty$. Then the sequence $\{u^k\}$ converges to the solution of \Cref{equ:hhh} as $k\rightarrow \infty$.
\end{theorem}

\begin{proof}
     See \cite{luenberger1984linear}.
\end{proof}

In practice, \Cref{thm:cvg} may be of limited use since it requires solving a sequence of hard-to-solve unconstrained minimization problems. Still, under certain assumptions, it can provide an approximate solution of \Cref{equ:hhh}.

\subsection{Physics Constrained Learning}\label{sect:acem}

We present an algorithm to compute the gradients of the loss function with respect to $\theta$ while numerically satisfying the physical constraint during optimization (\Cref{fig:first}-right). We denote the numerical solution to $F_h(\theta, u_h)=0$ by $u_h$, i.e., 
\begin{equation}\label{equ:fg0}
    F_h(\theta, G_h(\theta))=0 
\end{equation}
The new loss function becomes
\begin{equation*}
    \tilde L_h(\theta)  = L_h(G_h(\theta))
\end{equation*}
To derive the gradient $\frac{\partial \tilde L_h(\theta)}{\partial \theta}$, we apply the implicit function theorem to \Cref{equ:fg0} and obtain
\begin{align*}
	& \frac{{\partial {F_h(\theta, u_h)}}}{{\partial \theta }} + {\frac{{\partial {F_h(\theta, u_h)}}}{{\partial {u_h}}}}  \frac{\partial G_h(\theta)}{\partial \theta} = 0 \qquad \Rightarrow \\[4pt]
    & \frac{\partial G_h(\theta)}{\partial \theta} =  -\Big( \frac{{\partial {F_h(\theta, u_h)}}}{{\partial {u_h}}} \Big)^{ - 1} \frac{{\partial {F_h(\theta, u_h)}}}{{\partial \theta }}
\end{align*}
Therefore we have
\begin{align}
    \frac{{\partial {{\tilde L}_h}(\theta )}}{{\partial \theta }} 
    &= \frac{\partial {{ L}_h}(u_h )}{\partial u_h}\frac{\partial G_h(\theta)}{\partial \theta}\notag \\
    &= - \frac{{\partial {L_h}({u_h})}}{{\partial {u_h}}} \;
    \Big( {\frac{{\partial {F_h(\theta, u_h)}}}{{\partial {u_h}}}\Big|_{u_h = {G_h}(\theta )}} \Big)^{ - 1} \;
    \frac{{\partial {F_h(\theta, u_h)}}}{{\partial \theta }}\Big|_{u_h = {G_h}(\theta )}\label{equ:s2}
\end{align}
Now we discuss how to compute \Cref{equ:s2} efficiently.  Assume that the number of parameters in $\theta$ is $p$, the degrees of freedom of $u_h$ is $N$, the complexity of solving a linear system with coefficient matrix $\frac{\partial {F_h(\theta, u_h)} }{\partial u_h}$ is $C$ (together with boundary conditions) and the complexity of evaluating $F_h(\theta, u_h)$ is $C'$. Additionally, in the following complexity estimation, we assume that we have already solved $F_h(\theta, u_h)=0$ and obtain the solution vector $u_h$. There are in general two strategies to compute \Cref{equ:s2}:
\begin{enumerate}
    \item We compute \Cref{equ:s2} from right to left, i.e., compute
    \[ 
        z = \underbrace{
            \Big( {\frac{{\partial {F_h}}}{{\partial {u_h}}}\Big|_{u_h = {G_h}(\theta )}} \Big)^{ - 1}}_{N\times N}
            \;\; \underbrace{\frac{{\partial {F_h}}}{{\partial \theta }}\Big|_{u_h = {G_h}(\theta )}}_{N\times p}
    \]
    which has cost $\mathcal{O}(pC)$ and then compute
    \[ \underbrace{\frac{{\partial {L_h}({u_h})}}{{\partial {u_h}}}}_{1\times N}\; z \]
    which has cost $\mathcal{O}(Np)$. The total cost is $\mathcal{O}(\max\{Np, pC\})=\mathcal{O}(pC)$ (typically $C$ is at least $\mathcal{O}(N)$).
    \item We first compute 
    \[ 
        w^T = \underbrace{\frac{{\partial {L_h}({u_h})}}{{\partial {u_h}}\rule[-9pt]{1pt}{0pt}}}_{1\times N} 
        \;\;
        \underbrace{\Big( {\frac{{\partial {F_h}}}{{\partial {u_h}}}\Big|_{u_h = {G_h}(\theta )}} \Big)^{ - 1}}_{N\times N}
    \]
    which is equivalent to solving a linear system 
    \begin{equation}\label{equ:linsys2}
        {\left( {\frac{{\partial {F_h}}}{{\partial {u_h}}}\Bigg|_{u_h = {G_h}(\theta )}} \right)^{T}} w =\left(\frac{{\partial {L_h}({u_h})}}{{\partial {u_h}}}\right)^T
    \end{equation}
    with cost $\mathcal{O}(C)$, and then compute (in the following operation, we assume $w$ is independent of $\theta$)
    $$w^T\;\underbrace{\frac{{\partial {F_h}}}{{\partial \theta }}\Big|_{u_h = {G_h}(\theta )}}_{N\times p} = \frac{\partial (w^T\;  {F_h}(\theta, u_h))}{\partial \theta }\Bigg|_{u_h = {G_h}(\theta )}$$
    If we apply reverse mode automatic differentiation, using fused-multiply adds, denoted by OPS, as a metric of computational complexity, we have \cite{griewank2008evaluating,margossian2018review}
    \begin{equation}
        \mathrm{OPS}\left( \frac{\partial (w^T\;  {F_h}(\theta, u_h))}{\partial \theta } \right) \leq 
        4 \; \mathrm{OPS}({F_h}(\theta, u_h)))
    \end{equation}
    Therefore, the total computational cost will be  $\mathcal{O}(\max\{C', C\})$.
\end{enumerate}
In the case when $p$ is large, e.g., $p$ are the weights and biases of a neural network, the second approach is preferable. Hence, we adopt the second strategy in our implementation. 

In many physical simulations, we need to solve a linear system, which can be expressed as 
\begin{equation}\label{equ:linearcase}
F_h(\theta_1,\theta_2, u_h) = \theta_1 - A(\theta_2)u_h = 0
\end{equation}
Applying \Cref{equ:s2} to \Cref{equ:linearcase}, we have 
	\begin{align*}
p &:=\frac{\partial \tilde L_h(\theta_1, \theta_2)}{\partial \theta_1} = \frac{\partial L_h(u_h)}{\partial u_h}A(\theta_2)^{-1}\\
q &:=\frac{\partial \tilde L_h(\theta_1, \theta_2)}{\partial \theta_2} = -\frac{\partial L_h(u_h)}{\partial u_h}A(\theta_2)^{-1} \frac{\partial A(\theta_2)}{\partial \theta_2}
\end{align*}
which is equivalent to 
$$A^T p^T =\left( \frac{\partial  L_h(u_h)}{\partial u_h} \right)^T\quad q = -p\frac{\partial A(\theta_2)}{\partial \theta_2} $$
In the following, we consider a concrete example of \Cref{equ:linearcase}.

\begin{example}
We consider the following example: solve an inverse modeling problem with the 1D Poisson equation, with $F(\theta, u)=0$ expressed as 
\begin{equation}\label{equ:model equation}
  \begin{aligned}
    \frac{\partial}{\partial x} \left( 
        f(u(x);\theta) \; \frac{\partial u(x)}{\partial x} 
        \right) &= g(x) & x \in (0,1)\\
    u(0)=u(1) &= 0 & 
\end{aligned} 
\end{equation}
where $f(u(x);\theta)$ is a function of $u(x)$ parametrized by an unknown parameter $\theta$. For example, $f(u(x); \theta)$ can be a neural network which maps $u(x)$ to a scalar value and $\theta$ are the weights and biases. The sparse observations are $u_{i,\mathrm{obs}}$, the true values of $u(x)$ at location $i\in \mathcal{I}_\mathrm{obs}$, and $\mathcal{I}_\mathrm{obs}\subset \mathbb{N}$ is a set of location index. 

We apply the finite difference method to \Cref{equ:model equation}. We discretize $u$ on a uniform grid with interval length $h$ and node $x_i = \frac{i-1}{n}$, $i=1,2,\ldots,n+1$. $u_i$ denotes the discretized values at $x_i$, $u_h = [u_1, u_2, \ldots, u_{n+1}]$ and $g_i = g(x_i)$.  The corresponding $F_h(\theta, u_h)$ is
\begin{gather*}
    {F_h(\theta, u_h)}_i = f\left( \frac{u_i+u_{i+1}}{2};\theta \right) \frac{u_{i+1}-u_i}{h^2} - f\left( \frac{u_i+u_{i-1}}{2};\theta \right) \frac{u_{i}-u_{i-1}}{h^2} - g_i \\
    i=2,3,\ldots, n\\
    u_1 =u_{n+1} = 0
\end{gather*}
Assume that the observations are located exactly at some of the grid nodes, the loss function can be formulated as 
\begin{equation*}
    L_h(u_h) =\sum_{i\in \mathcal{I}_{\mathrm{obs}}} \left(u_i - u_{i,\mathrm{obs}}\right)^2 
\end{equation*}
and therefore, we have
\begin{equation}\label{equ:dLdu}
    \frac{{\partial {L_h}({u_{h,i}})}}{{\partial {u_h}}} = 
    \begin{cases}
        2 \left(u_i - u_{i,\mathrm{obs}}\right) & i\in \mathcal{I}_\mathrm{obs}\\
        0 & i\not\in \mathcal{I}_\mathrm{obs}
    \end{cases} 
\end{equation}
The Jacobian of $F_h(\theta, u_h)$ with respect to $u_h$ is 
\begin{align}
    \frac{\partial {F_h(\theta, u_h)}_i }{\partial u_{h,j}} = \ &  \frac{\partial}{\partial u}f\left( \frac{u_i+u_{i+1}}{2};\theta \right) \frac{\delta_{i,j}+\delta_{i+1,j}}{2} \frac{u_{i+1}-u_i}{h^2} \notag \\
    & + f\left( \frac{u_i+u_{i+1}}{2};\theta \right) \frac{\delta_{i+1,j}-\delta_{i,j}}{h^2} \notag \\
    & - \frac{\partial}{\partial u}f\left( \frac{u_i+u_{i-1}}{2};\theta \right) \frac{\delta_{i,j}+\delta_{i-1,j}}{2}\frac{u_{i}-u_{i-1}}{h^2} \notag \\
    & - f\left( \frac{u_i+u_{i-1}}{2};\theta \right) \frac{\delta_{i,j}-\delta_{i-1, j}}{h^2} \label{equ:jacobian}
\end{align}
where 
\begin{equation*}
    \delta_{ij} = \begin{cases}
        1 & i=j\\
        0 & i\neq j
    \end{cases}
\end{equation*}
Note that for fixed $i$, $\frac{\partial {F_h(\theta, u_h)}_i }{\partial u_j}$ can only be nonzero for $j=i-1, i, i+1$ and therefore the Jacobian matrix is sparse.

Additionally, we have
\begin{equation}\label{equ:Ftheta}
    \begin{split}
    \frac{{\partial {F_h}}}{{\partial \theta }}(\theta ,{G_h}(\theta )) & = \nabla_\theta f\left( \frac{u_i+u_{i+1}}{2};\theta \right) \frac{u_{i+1}-u_i}{h^2} \\
    & \hspace*{4em} - \nabla_\theta f\left( \frac{u_i+u_{i-1}}{2};\theta \right) \frac{u_{i}-u_{i-1}}{h^2}             
    \end{split}
\end{equation}
Using the results from \Cref{equ:dLdu,equ:jacobian,equ:Ftheta} and the formula \Cref{equ:s2} we obtain the gradient ${{\partial {{\tilde L}_h}(\theta )}}/{{\partial \theta }}$.
\end{example}

\begin{remark}
    In practice, if the system of PDEs or the relation between $f(u;\theta)$ and $\theta$ is complex (e.g., a neural network), deriving and implementing \Cref{equ:jacobian,equ:Ftheta} are  challenging. Fortunately, we can compute those terms using automatic differentiation techniques. We present the technical details in  \ref{sect:detail} and  \ref{sect:trick}.
\end{remark}

\subsection{Analysis of PCL}

In the penalty method, the constraint $F(\theta,u)=0$ is imposed by including a penalty term $\|F_h(\theta, u_h)\|^2_2$ in the loss function, and the summation of the penalty term and the observation error is minimized simultaneously with gradient descent methods. Intuitively, if the problem $F_h(\theta,u_h)=0$ is stiff, the penalty method suffers from slow convergence due to worse condition number of the least square loss $\|F_h(\theta, u_h)\|^2_2$ compared to the original problem $F_h(\theta, u_h)=0$.

In this section, we analyze the convergence by considering a model problem
\begin{gather*}
    \min_{\theta} \|u-u_0\|^2_2 \\
    \text{s.t.} \;\; Au = \theta y
\end{gather*}
where $\theta\in \RR$ is unknown and $A$ is a nonsingular square coefficient matrix. For simplicity, assume that the true value for $\theta$ is $1$, and $u_0 = A^{-1} y$.

In PCL, we have
\begin{equation*}
    \tilde L_h(\theta) = \|\theta A^{-1} y - u_0\|^2_2 = (\theta-1)^2\|u_0\|_2^2
\end{equation*}
which is a quadratic function in $\theta$ and can be minimized efficiently using gradient-based method. Nevertheless, we have to solve $Au = y$, and the computational cost usually depends on the condition number $\kappa(A)$.

In the penalty method, the new loss function is
\begin{equation}\label{equ:least}
    \min_{\theta, u_h}\tilde L_h(\theta, u_h) = \|u_h-u_0\|^2_2 + \lambda \|Au_h -\theta y\|_2^2
\end{equation}
\Cref{equ:least} is equivalent to a least square problem $\min_{\bm{\theta}}\|\mathbf{A}\bm{\theta}-\mathbf{y}\|_2^2$ with the coefficient matrix and the right hand side
\begin{equation*}
    \mathbf{A}_\lambda = \begin{bmatrix}
        I & 0\\
        \sqrt{\lambda}A & -\sqrt{\lambda}y
    \end{bmatrix}, \qquad 
    \mathbf{y} = \begin{bmatrix}
        u_0\\ 0
    \end{bmatrix}
\end{equation*}
The least-square problem has a condition number that is at least  $\kappa(A)^2$ asymptotically, which is implied by the following theorem:
\begin{theorem}
    The condition number of $\mathbf{A}_\lambda$ is 
    \begin{equation*}
        \liminf_{\lambda\rightarrow \infty}\kappa(\mathbf{A}_\lambda)  \geq  \kappa(A)^2
    \end{equation*}
    and therefore, the condition number of the unconstrained optimization problem from the penalty method is the square of that from PCL asymptotically. 
\end{theorem}

\begin{proof}
    Assume that the singular value decomposition of $A$ is
    \begin{equation*}
        A = U\Sigma V^T
    \end{equation*}
    where 
    \begin{equation*}
        \Sigma = \mathrm{diag}(\sigma_1, \sigma_2, \ldots, \sigma_n)
    \end{equation*}
    Without loss of generality, we assume $\sigma_1\geq\sigma_2\geq\cdots\geq\sigma_n> 0$. We have
    \begin{equation*}
        \mathbf{A}_\lambda = \lambda \begin{bmatrix}
            V & 0\\
            0 & 1
        \end{bmatrix} \begin{bmatrix}
            \frac{1}{\lambda} I + \Sigma^2 & \alpha \\
            \alpha^T & s
        \end{bmatrix} \begin{bmatrix}
            V^T & 0\\
            0 & 1
        \end{bmatrix}
    \end{equation*}
    for 
    \begin{equation*}
        s = y^Ty, \qquad \alpha = -\Sigma U^T y
    \end{equation*}
    Note 
    $$B = \begin{bmatrix}
            \frac{1}{\lambda} I + \Sigma^2 & \alpha \\
            \alpha^T & s
        \end{bmatrix}$$
    is an arrowhead matrix and its eigenvalues are expressed by the zeros of \cite{stor2014accurate}
    \begin{equation*}
        f(x) = s - x - \sum_{i=1}^{n}  \frac{\alpha_i^2}{\sigma_i^2 + \frac{1}{\lambda} - x}
    \end{equation*}
    Note that 
    \begin{equation*}
        \lim_{x\rightarrow \left(\sigma_i^2+ \frac{1}{\lambda}\right)+} f(x) \rightarrow +\infty,\quad \lim_{x\rightarrow \left(\sigma_i^2+ \frac{1}{\lambda}\right)-} f(x) \rightarrow -\infty
    \end{equation*}
    and
    \begin{equation*}
        \lim_{x\rightarrow +\infty} f(x) = -\infty\quad \lim_{x\rightarrow -\infty} f(x) = +\infty
    \end{equation*}
    We infer that the $n+1$ eigenvalues of $B$ are located in $(0, \sigma_1^2 + \frac{1}{\lambda} )$, $(\sigma_1^2 + \frac{1}{\lambda}, \sigma_2^2 + \frac{1}{\lambda})$, $\ldots$, $(\sigma_n^2 + \frac{1}{\lambda}, \infty)$. The smallest eigenvalue of $B$ is positive since $A$ is positive definite and so is $B$. 
    
    Therefore, we deduce that $\kappa(\mathbf{A}_\lambda) \geq \frac{\sigma_n^2+\frac{1}{\lambda}}{\sigma_1^2+\frac{1}{\lambda}}$, thus
    \begin{equation*}
        \liminf_{\lambda\rightarrow \infty} \kappa(\mathbf{A}_\lambda) \geq \frac{\sigma_n^2}{\sigma_1^2} = \kappa(A)^2
    \end{equation*}
\end{proof}

\section{Numerical Benchmarks}\label{sect:num}

In this section, we perform four numerical benchmarks and compare physics constrained learning (PCL) and the penalty method (PM).  Unless specified, we use \texttt{L-BFGS-B} \cite{luenberger1984linear,bonnans2006numerical}  to minimize the loss function for both methods, and use the same tolerance $\varepsilon = 10^{-12}$ for the gradient norm and the relative function change, i.e., we stop the iteration if 
\begin{equation*}
    \|\nabla_\theta l(z)\| <\varepsilon\qquad \left| \frac{l(z)-l(z')}{l(z')} \right| <\varepsilon
\end{equation*}
where $z$ and $z'$ represent the candidate parameter at the current and previous step and 
\begin{align*}
	\mbox{ PCL: }&& z &= \theta, & l(z) &= \tilde L_h(\theta) \\
	\mbox{ PM: }&& z &= (\theta, u_h), & l(z) &= \tilde L_h(\theta, u_h)
\end{align*}
To show that PCL is applicable for various numerical schemes and PDEs, we test various PDEs, numerical schemes and inverse problem types, which are listed in \Cref{tab:example}.

\begin{table}[htpb]
\centering
\begin{tabular}{@{}cccc@{}}
\toprule
Example & \qquad\qquad PDE & Numerical Scheme & Unknown Type \\ \midrule
\qquad1 &\qquad Linear Static & \qquad IGA &\quad Parameter \\
\qquad2 &\quad Nonlinear Static & \qquad FD &\quad Function \\
\qquad3 & Nonlinear Dynamic & \qquad FVM &\quad Function \\ \bottomrule
\end{tabular}
\caption{Numerical Benchmarks. IGA is short for the isogeometric analysis; FD is short for the finite difference method. FVM is short for the finite volume method. When the unknown type is a function, we use a deep neural network to approximate the unknown function and couple it with numerical schemes.}
\label{tab:example}
\end{table}
    
\subsection{Parametric Inverse Problem: Helmholtz Equation}

We consider the inverse problem for a Helmholtz equation \cite{chen1992inverse,bao2005inverse,tadi2011inverse}. Let $\Omega$ be a bounded domain in 2D, and the Helmholtz equation is given by
\begin{equation}\label{equ:helm}
    \Delta u + k^2 g(\bx) u = 0
\end{equation}
with the Dirichlet boundary condition 
\begin{equation}\label{equ:helmbd}
    u(\bx) = u_0(\bx),\ \bx\in \partial\Omega
\end{equation}
where $u(\bx)$ denotes the electric field, the amplitude of a time-harmonic wave, or orbitals for an energy state depending on the models. The parameter $k$ is the frequency and $g(\bx)$ is a physical parameter. In the applications, we want to recover $g(\bx)$ based on boundary measurements, i.e., the values of $u(\bx)$ and $h(\bx) = \frac{\partial u}{\partial n}(\bx)$. $g(\bx)$ are given in a parametric form
\begin{equation*}
    g(\bx) = ax^2 + bxy+cy^2 + dx+ey+f, \ \bx=(x,y)
\end{equation*}
where $\theta = (a, b, c, d, e, f)$ are unknown parameters.  

Assume we have observed $h_i$, i.e., the values of  $\frac{\partial u}{\partial n}(\bx)$ on the boundary points $\{\bx_i\}_{i\in \mathcal{I}_{\mathrm{obs}}}$, the optimization problem can be formulated as
\begin{gather}\label{equ:helmopt}
    \min_{\theta}\  L(u) =\frac{1}{n_{\mathrm{obs}}} \sum_{i\in \mathcal{I}_{\mathrm{obs}}}\left( \frac{\partial u}{\partial n}(\bx_i) - h_i  \right)^2\\
    \mathrm{s.t.}\;\; F_h(\theta, u_h) = 0 \notag
\end{gather}

Here $F_h(\theta, u_h)=0$ is the IGA discretization of \Cref{equ:helm} with the boundary condition \Cref{equ:helmbd}. 

\begin{figure}[hbtp]
    \centering
    \includegraphics[width=0.48\textwidth]{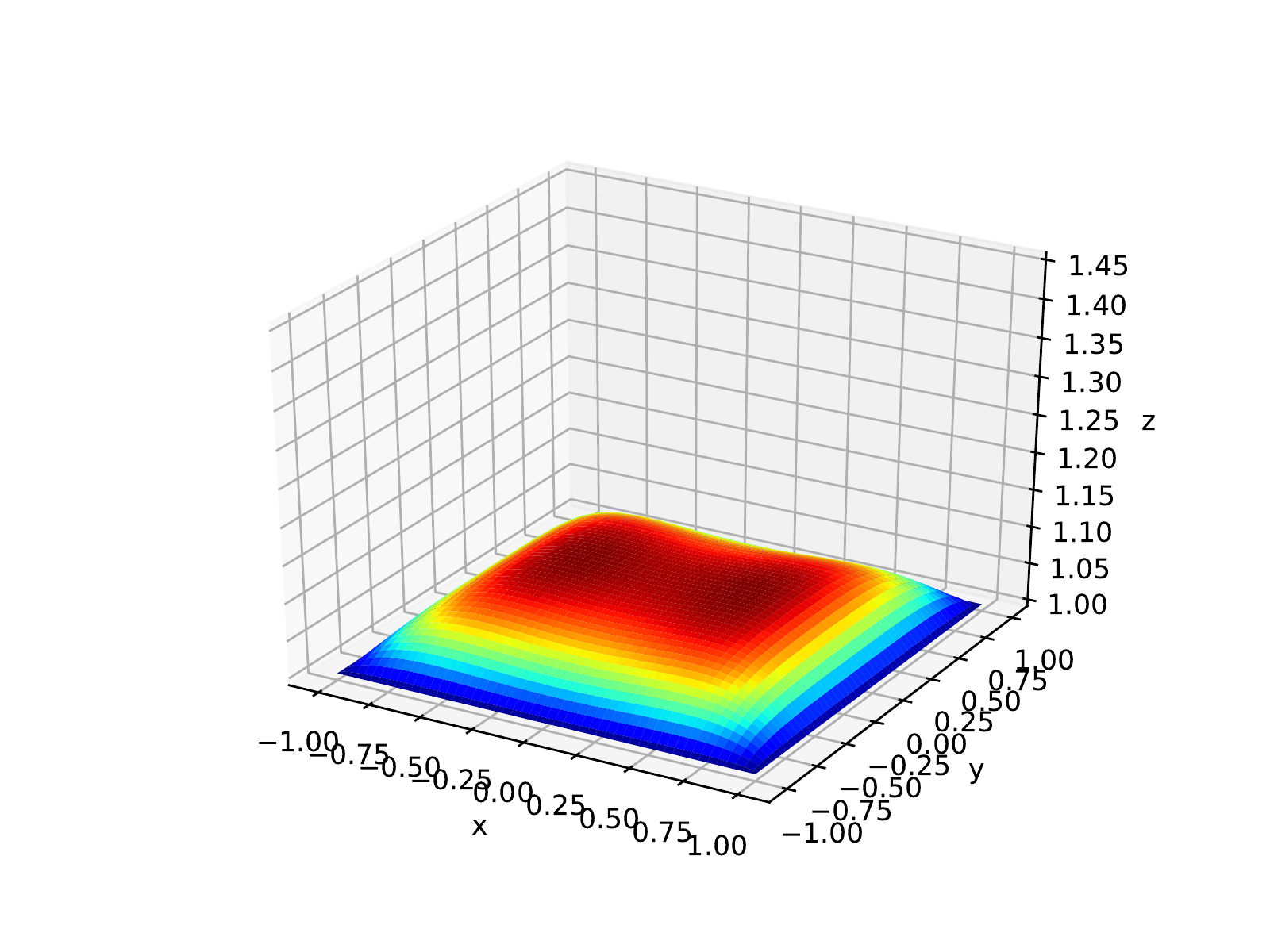}~
    \includegraphics[width=0.48\textwidth]{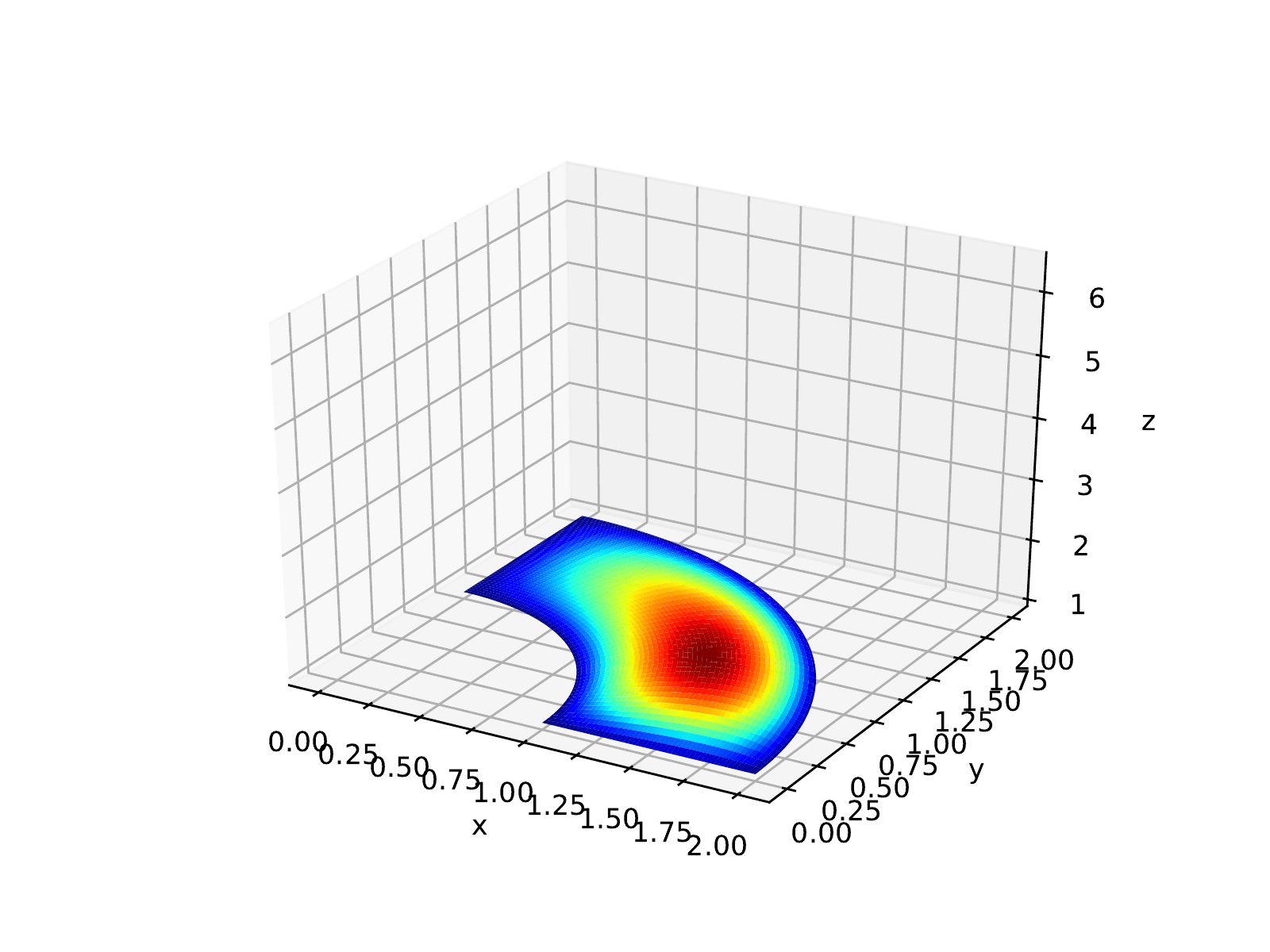}
    \includegraphics[width=0.48\textwidth]{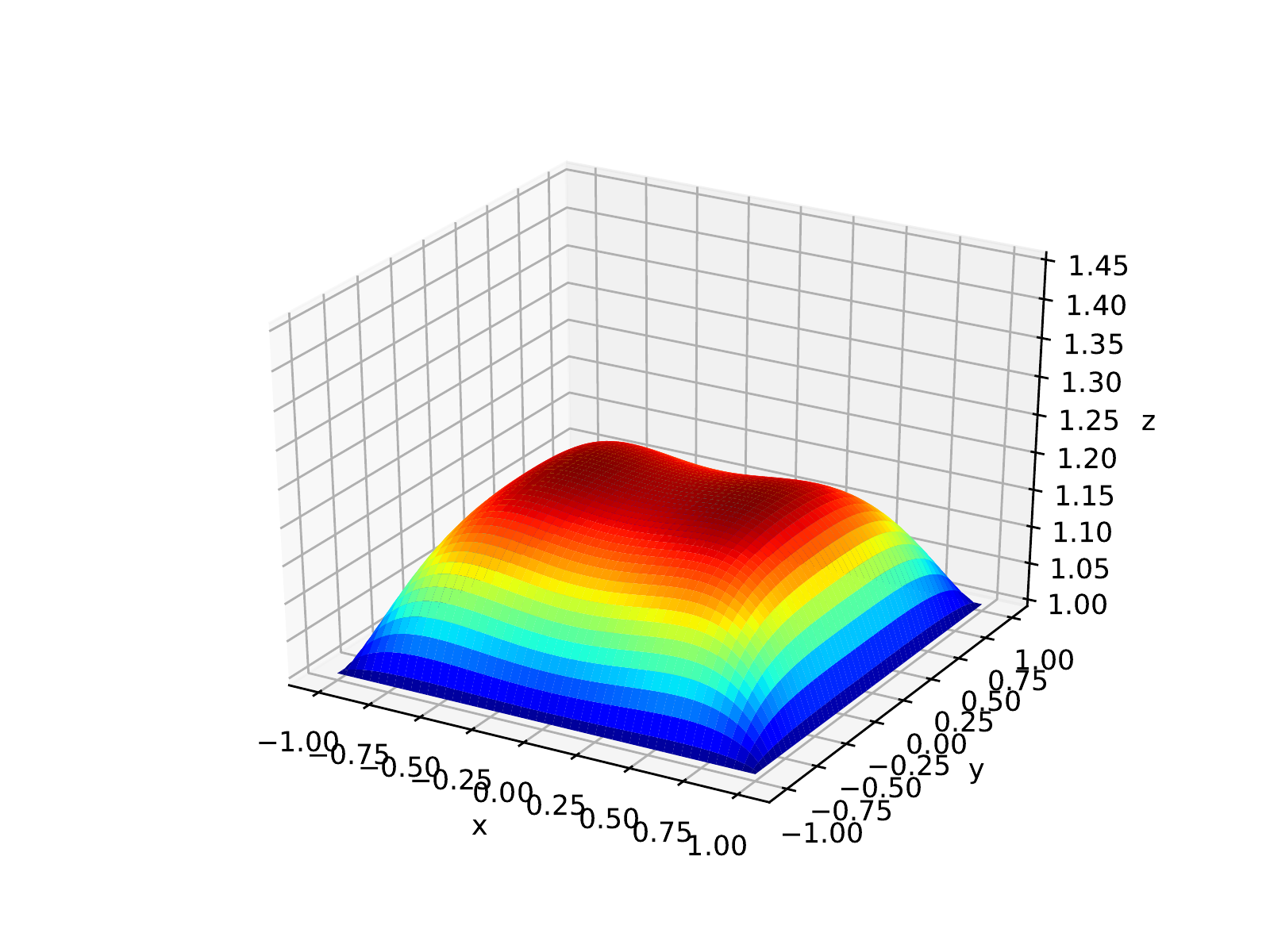}~
    \includegraphics[width=0.48\textwidth]{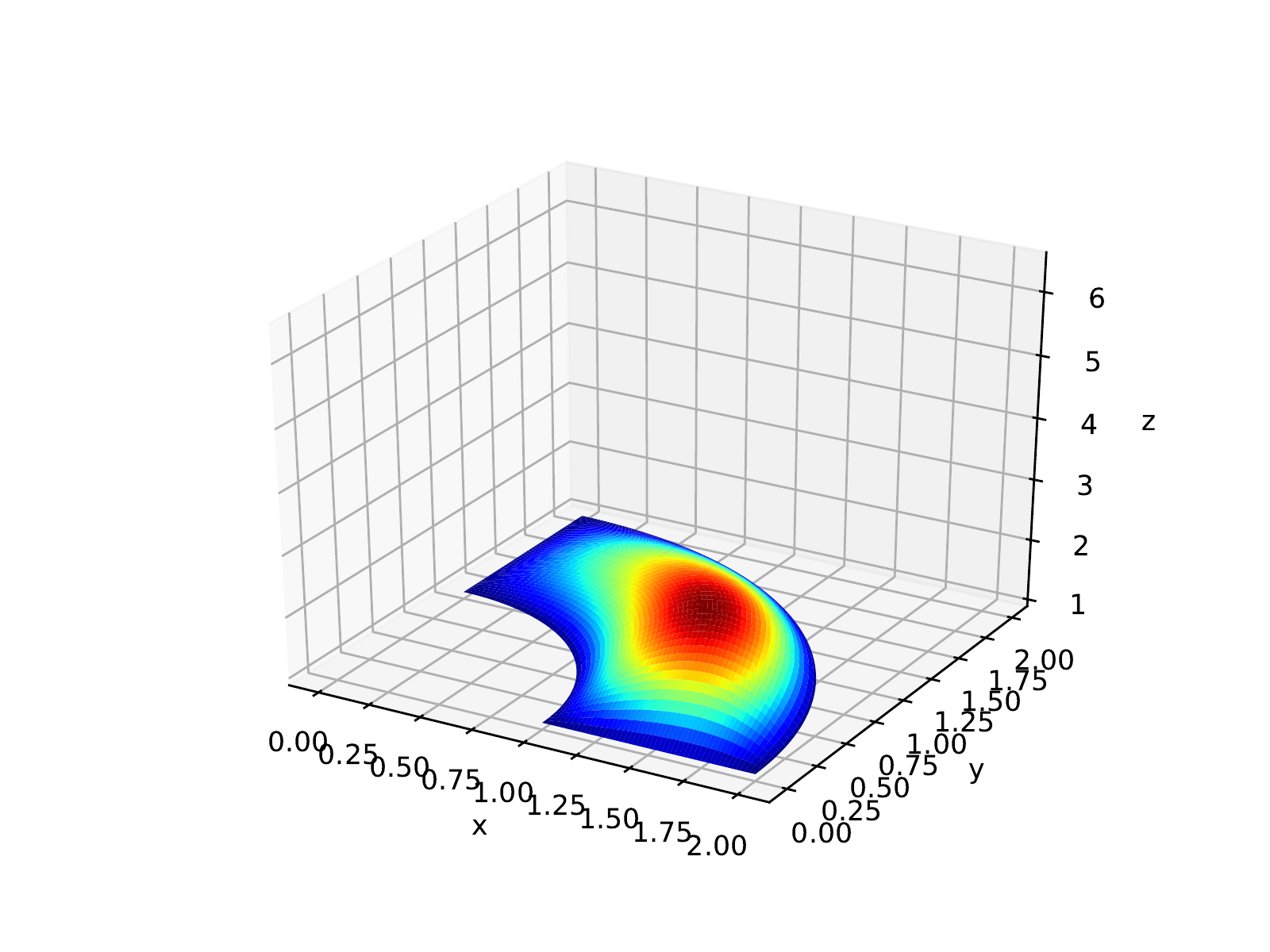}
    \includegraphics[width=0.48\textwidth]{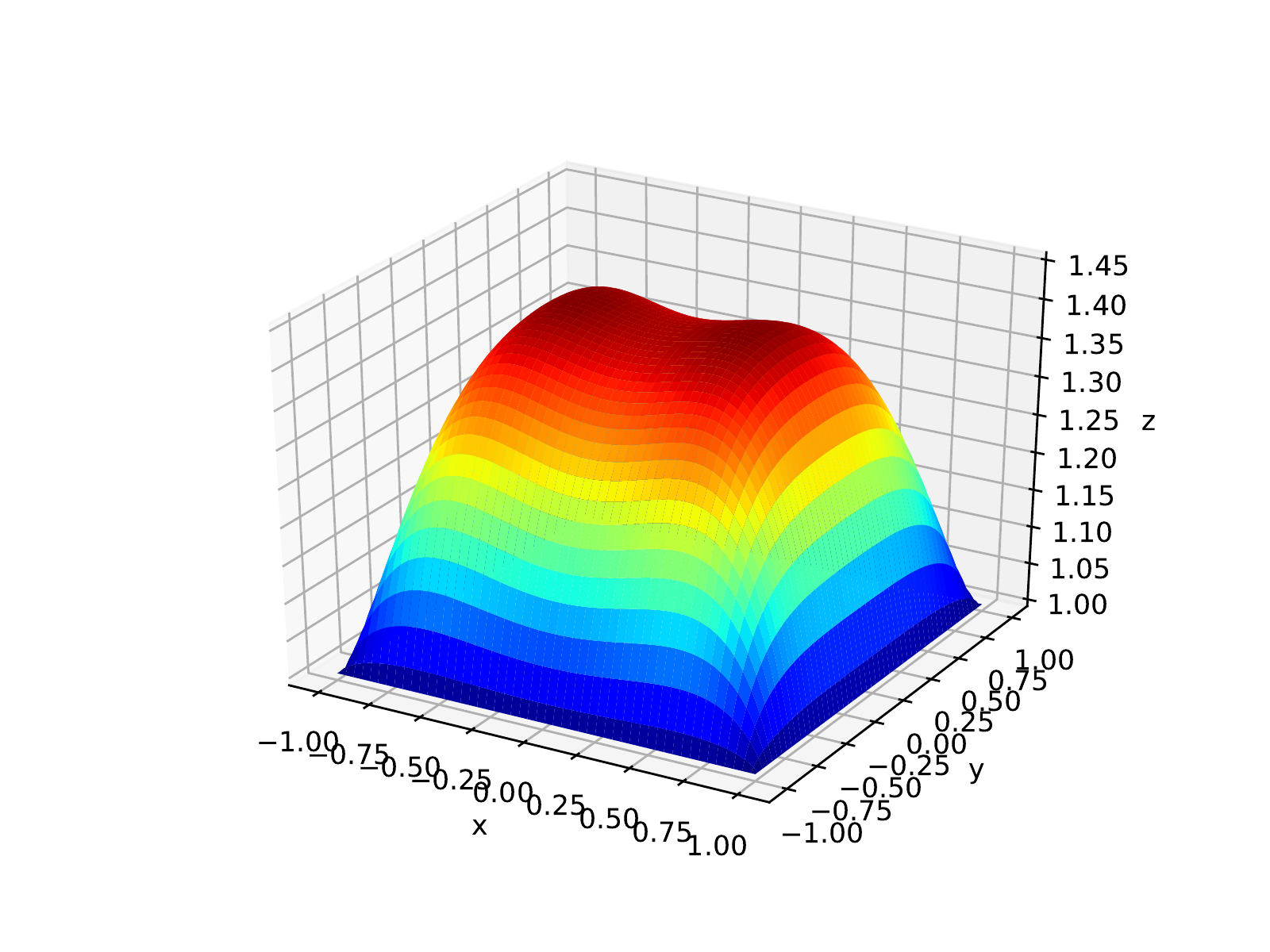}~
    \includegraphics[width=0.48\textwidth]{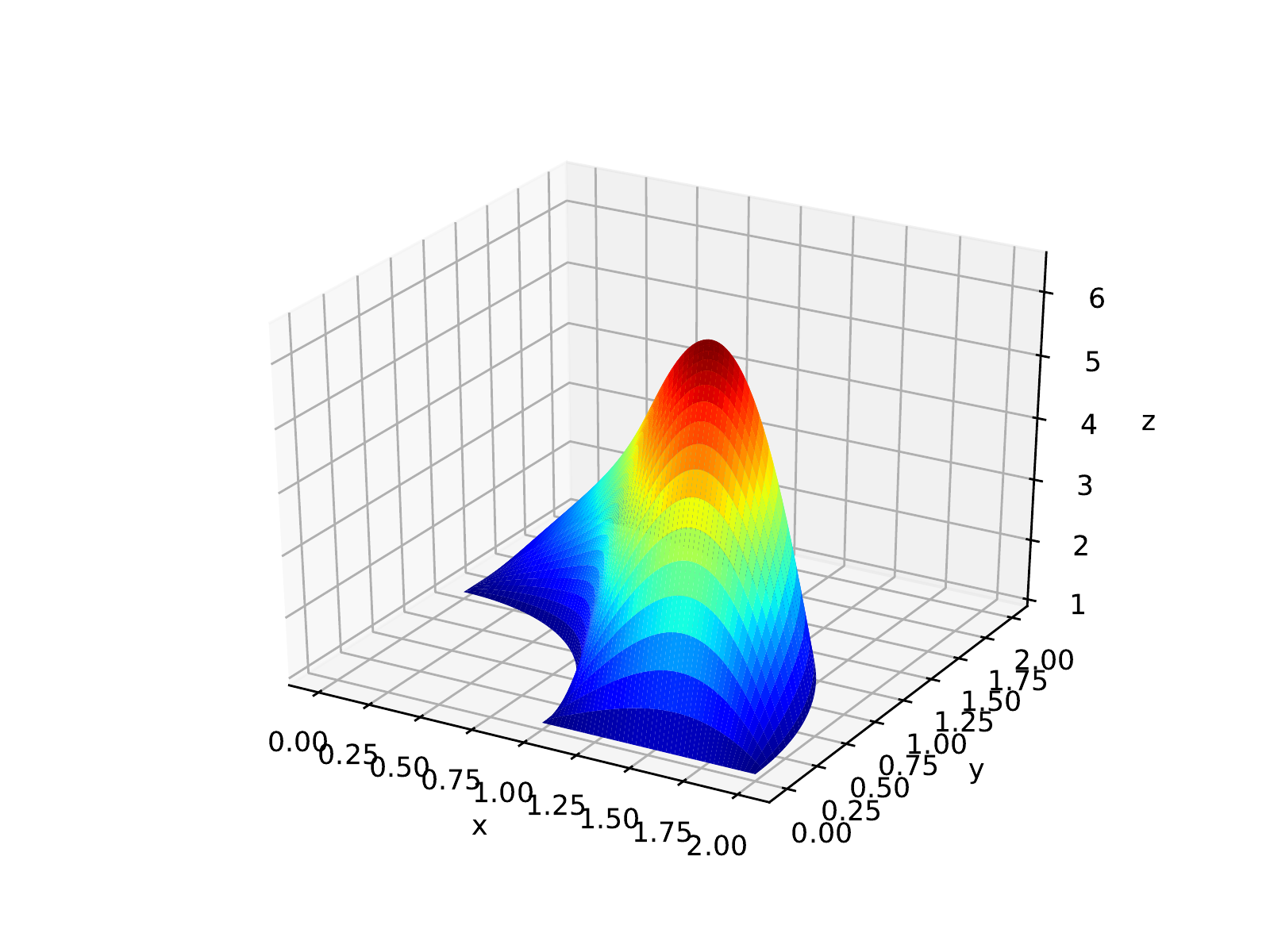}
    \caption{Solution profiles of \Cref{equ:helm} for a square domain and a quarter of annulus domain. The three rows correspond to $k=0.5$, $k=0.75$ and $k=1.0$.}
    \label{equ:helmsol}
\end{figure}

For verification, we use the parameters $\theta^*=(5,0,2,0,0,0)$, i.e., 
\begin{equation*}
    g(\bx) = 5x^2+2y^2
\end{equation*}
The boundary condition is given by $u_0(\bx)\equiv 1$.
We consider a square domain and a quarter of an annulus domain. The solution profiles are shown in \Cref{equ:helmsol} for different $k$. The error of the inverse modeling problem is measured by 
\begin{equation*}
    \mathrm{error} = \|\theta-\theta^*\|_2
\end{equation*}
where $\theta$ is the estimated parameter by solving \Cref{equ:helmopt}.

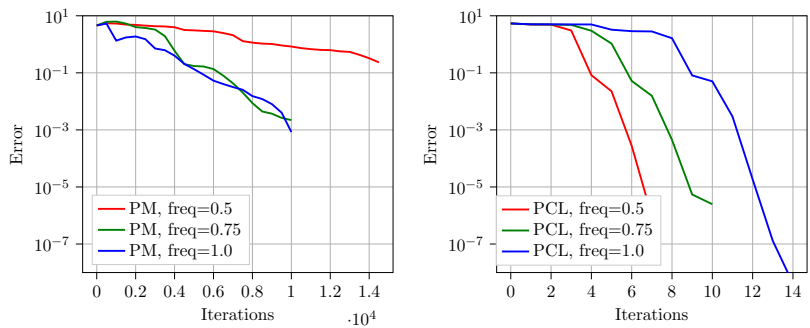
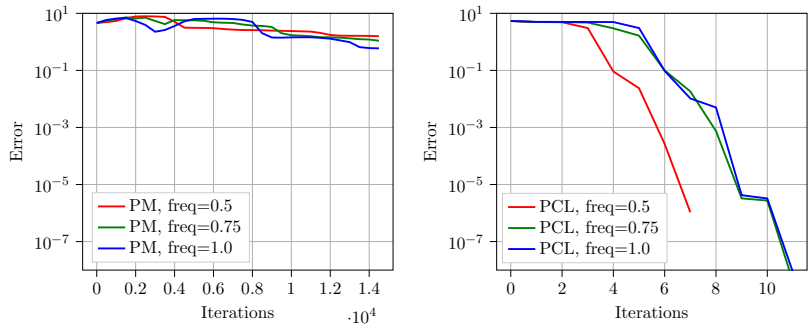
\begin{figure}[hbtp]
\centering
\begin{subfigure}[b]{1.0\textwidth}
    \scalebox{0.6}{
\begin{tikzpicture}

\begin{axis}[
legend cell align={left},
legend style={at={(0.03,0.03)}, anchor=south west, draw=white!80.0!black},
log basis y={10},
tick align=outside,
tick pos=left,
x grid style={white!69.01960784313725!black},
xlabel={Iterations},
xmajorgrids,
xmin=-724, xmax=15226,
xtick style={color=black},
y grid style={white!69.01960784313725!black},
ylabel={Error},
ymajorgrids,
ymin=1e-08, ymax=10,
ymode=log,
ytick style={color=black}
]
\addplot [very thick, red]
table {%
1 4.58257569495584
501 5.45135247300615
1001 5.35329746058426
1501 4.91910502776306
2001 4.78794620704812
2501 4.50705679052027
3001 4.31061519063046
3501 4.21817727209582
4001 3.95972470129332
4501 3.17757500617134
5001 3.08059344254725
5501 2.951415949853
6001 2.84996002788234
6501 2.48147124005749
7001 2.08826099322667
7501 1.29431543517285
8001 1.14441233844365
8501 1.05862336210903
9001 1.02654156717563
9501 0.91174185967978
10001 0.835278329460356
10501 0.735650321696974
11001 0.67678104426737
11501 0.637553757934763
12001 0.624021679900687
12501 0.56541029943367
13001 0.54208191449507
13501 0.426789599571169
14001 0.324273643517257
14501 0.232782450191256
};
\addlegendentry{PM, freq=0.5}
\addplot [very thick, green!50.0!black]
table {%
1 4.58257569495584
501 6.16532876549637
1001 6.26185659757236
1501 5.33958421532958
2001 4.0013571990914
2501 3.73437015875179
3001 3.28417766108257
3501 1.9424599472082
4001 0.588433976941736
4501 0.20217041043927
5001 0.172952786989561
5501 0.166557667043537
6001 0.1352433116366
6501 0.0798435320448481
7001 0.0427960849477299
7501 0.0204414794723116
8001 0.00868100710289147
8501 0.00443207412673499
9001 0.00370480435618101
9501 0.00263488555749241
10001 0.00220643010313388
};
\addlegendentry{PM, freq=0.75}
\addplot [very thick, blue]
table {%
1 4.58257569495584
501 5.34369559589963
1001 1.34211141794814
1501 1.73982316844389
2001 1.88026618202206
2501 1.50734802599141
3001 0.711430783835627
3501 0.621566716668801
4001 0.399216736904136
4501 0.205662781492982
5001 0.134106640876066
5501 0.0853099589701694
6001 0.0530699283707471
6501 0.0405267074888279
7001 0.0316818243490434
7501 0.0255020960401051
8001 0.0152655440774106
8501 0.012099900496458
9001 0.00801252102968886
9501 0.0040117193248313
10001 0.000841123268098174
};
\addlegendentry{PM, freq=1.0}
\end{axis}

\end{tikzpicture}}~
    \scalebox{0.6}{
\begin{tikzpicture}

\begin{axis}[
legend cell align={left},
legend style={at={(0.53,0.3)},draw=white!80.0!black},
log basis y={10},
tick align=outside,
tick pos=left,
x grid style={white!69.01960784313725!black},
xlabel={Iterations},
xmajorgrids,
xmin=-0.7, xmax=14.7,
xtick style={color=black},
y grid style={white!69.01960784313725!black},
ylabel={Error},
ymajorgrids,
ymin=1e-08, ymax=10,
ymode=log,
ytick style={color=black}
]
\addplot [very thick, red]
table {%
0 5.3851648071345
1 4.95174610321776
2 4.89310612842589
3 3.05310441641147
4 0.0829584516389769
5 0.0225890931654804
6 0.000276554332151716
7 9.34439102342043e-07
};
\addlegendentry{PCL, freq=0.5}
\addplot [very thick, green!50.0!black]
table {%
0 5.3851648071345
1 4.97482211748462
2 4.93625660960836
3 4.78884532130295
4 2.98744719166923
5 1.06022383991075
6 0.0520467285377774
7 0.0156698002055435
8 0.000445231142078771
9 5.43205098594362e-06
10 2.49021718720198e-06
};
\addlegendentry{PCL, freq=0.75}
\addplot [very thick, blue]
table {%
0 5.3851648071345
1 5.031150268425
2 5.01459819983586
3 4.9829332449841
4 4.95893444778026
5 3.25205765328274
6 2.88420549556617
7 2.82445475769549
8 1.63175368928338
9 0.0812080400657106
10 0.0503596494429706
11 0.00295024357217641
12 1.85357283909278e-05
13 1.29165108464587e-07
14 4.15290077174589e-09
};
\addlegendentry{PCL, freq=1.0}
\end{axis}

\end{tikzpicture}}
    \caption{Square domain, refinement level = 5.}
\end{subfigure}
\begin{subfigure}[b]{1.0\textwidth}
    \scalebox{0.6}{
\begin{tikzpicture}

\begin{axis}[
legend cell align={left},
legend style={at={(0.03,0.03)}, anchor=south west, draw=white!80.0!black},
log basis y={10},
tick align=outside,
tick pos=left,
x grid style={white!69.01960784313725!black},
xlabel={Iterations},
xmajorgrids,
xmin=-724, xmax=15226,
xtick style={color=black},
y grid style={white!69.01960784313725!black},
ylabel={Error},
ymajorgrids,
ymin=1e-08, ymax=10,
ymode=log,
ytick style={color=black}
]
\addplot [very thick, red]
table {%
1 4.58257569495584
501 4.93037065701811
1001 5.42945653902434
1501 6.7318490752846
2001 7.73125722751208
2501 7.88648256640602
3001 7.77528495916128
3501 7.4936113868809
4001 5.06973034881838
4501 3.14799221784228
5001 3.09506995349464
5501 3.07829199460439
6001 3.02337577719135
6501 2.82244519619423
7001 2.68483965721876
7501 2.59754440521424
8001 2.58217720315911
8501 2.55134853441263
9001 2.48602402152591
9501 2.43884376216644
10001 2.39043348705555
10501 2.34250650305374
11001 2.28859123789353
11501 2.04746409256522
12001 1.74692305922627
12501 1.64620910658359
13001 1.6330838949383
13501 1.62702167896765
14001 1.61604209876749
14501 1.57550688245939
};
\addlegendentry{PM, freq=0.5}
\addplot [very thick, green!50.0!black]
table {%
1 4.58257569495584
501 5.48096813450567
1001 6.2944607992508
1501 7.23127895710294
2001 6.63499594543672
2501 7.03029158507006
3001 5.40308117872972
3501 4.19159821072572
4001 5.79670672809599
4501 5.73566285301473
5001 5.68176227653077
5501 5.47274937422443
6001 4.84742067336098
6501 4.68621935422364
7001 4.59921520887918
7501 4.06224745880846
8001 3.71850963052531
8501 3.61495929121049
9001 3.32248847832404
9501 2.00614408079802
10001 1.70888889794574
10501 1.65298264556363
11001 1.56868177357598
11501 1.45677194307163
12001 1.43546607209863
12501 1.40669350725837
13001 1.32833427791444
13501 1.24231146931677
14001 1.20633025098667
14501 1.0933487051743
};
\addlegendentry{PM, freq=0.75}
\addplot [very thick, blue]
table {%
1 4.58257569495584
501 5.85999235634398
1001 6.56647212705286
1501 6.81069378078375
2001 5.38636594970618
2501 3.9168930318574
3001 2.26181693875539
3501 2.60263902815227
4001 3.60234163908118
4501 5.22723954292218
5001 6.33179763239942
5501 6.41577210492129
6001 6.48434334046829
6501 6.49246505614059
7001 6.3181816935264
7501 5.89130921913141
8001 4.97835101385722
8501 2.00964462325956
9001 1.42262463828938
9501 1.40317465141318
10001 1.43976436396782
10501 1.45054648308404
11001 1.44610731882584
11501 1.38161956201571
12001 1.29122192725389
12501 1.12622768899078
13001 0.976406305057688
13501 0.65508702070087
14001 0.605047988996331
14501 0.591904540994457
};
\addlegendentry{PM, freq=1.0}
\end{axis}

\end{tikzpicture}}~
    \scalebox{0.6}{
\begin{tikzpicture}

\begin{axis}[
legend cell align={left},
legend style={at={(0.55,0.3)},draw=white!80.0!black},
log basis y={10},
tick align=outside,
tick pos=left,
x grid style={white!69.01960784313725!black},
xlabel={Iterations},
xmajorgrids,
xmin=-0.55, xmax=11.55,
xtick style={color=black},
y grid style={white!69.01960784313725!black},
ylabel={Error},
ymajorgrids,
ymin=1e-08, ymax=10,
ymode=log,
ytick style={color=black}
]
\addplot [very thick, red]
table {%
0 5.3851648071345
1 4.95140640170895
2 4.88803773004046
3 3.05024401962223
4 0.0915006172092631
5 0.023681523444515
6 0.000263609151363431
7 1.06487031714277e-06
};
\addlegendentry{PCL, freq=0.5}
\addplot [very thick, green!50.0!black]
table {%
0 5.3851648071345
1 4.97324805090675
2 4.93544362553223
3 4.75976120854492
4 2.97562102022236
5 1.65036523580165
6 0.0990689026367934
7 0.0183159256256242
8 0.000751376614000289
9 3.26652694947055e-06
10 2.74611325926833e-06
11 3.91456716142055e-09
};
\addlegendentry{PCL, freq=0.75}
\addplot [very thick, blue]
table {%
0 5.3851648071345
1 5.02421423036512
2 5.00688080983944
3 4.97553077396333
4 4.9486264287561
5 3.05482302606035
6 0.0956808255079305
7 0.0103860462018544
8 0.00502294763523996
9 4.22698771308632e-06
10 3.24277795360113e-06
11 8.47032099893387e-09
};
\addlegendentry{PCL, freq=1.0}
\end{axis}

\end{tikzpicture}}
    \caption{Square domain, refinement level = 6.}
\end{subfigure}
\caption{Convergence comparison of PM and PCL for the square  domain. Note that the scale of ``Iterations'' are different. Here ``refinement'' denotes the number of mesh refinement processes from the initial domain in the isogeometric analysis discretization. Each refinement process adds an extra knot between two knots in the original knot vector in each dimension.}
\label{fig:helmsquare}
\end{figure}

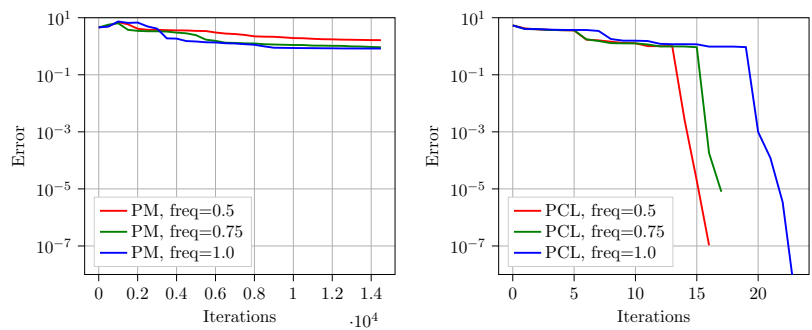
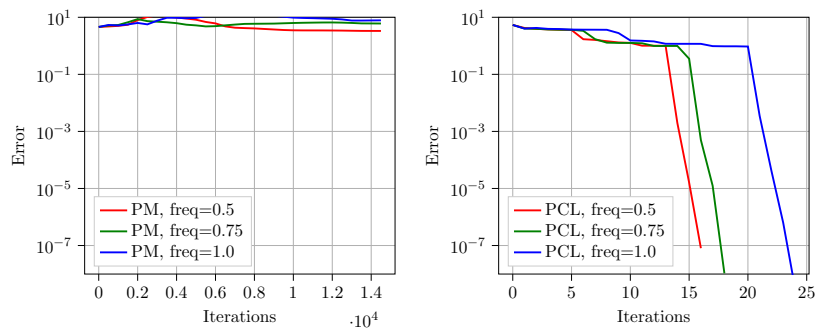
\begin{figure}[hbtp]
\centering
\begin{subfigure}[b]{1.0\textwidth}
    \scalebox{0.6}{
\begin{tikzpicture}

\begin{axis}[
legend cell align={left},
legend style={at={(0.03,0.03)}, anchor=south west, draw=white!80.0!black},
log basis y={10},
tick align=outside,
tick pos=left,
x grid style={white!69.01960784313725!black},
xlabel={Iterations},
xmajorgrids,
xmin=-724, xmax=15226,
xtick style={color=black},
y grid style={white!69.01960784313725!black},
ylabel={Error},
ymajorgrids,
ymin=1e-08, ymax=10,
ymode=log,
ytick style={color=black}
]
\addplot [very thick, red]
table {%
1 4.58257569495584
501 5.06859339732294
1001 6.95590659390533
1501 5.87430726557216
2001 4.11378224406554
2501 3.76986378811492
3001 3.75803964988065
3501 3.66488210636609
4001 3.59578873826808
4501 3.57716053805418
5001 3.45087077499217
5501 3.41481546548514
6001 3.03299669989451
6501 2.7870174791232
7001 2.67850099802229
7501 2.49361563264855
8001 2.22157723532998
8501 2.16941835664327
9001 2.14605562602751
9501 2.04811068746005
10001 1.93310171819497
10501 1.91049369352469
11001 1.82603360236971
11501 1.76623626218144
12001 1.74530742065824
12501 1.70781681686103
13001 1.68974717201499
13501 1.66620103029551
14001 1.64577829340898
14501 1.63523974778935
};
\addlegendentry{PM, freq=0.5}
\addplot [very thick, green!50.0!black]
table {%
1 4.58257569495584
501 5.71441264966296
1001 6.39796182617945
1501 3.76921490526635
2001 3.46355052166365
2501 3.3586595874512
3001 3.34180018467745
3501 3.26763736259737
4001 3.00353730006986
4501 2.84249343829667
5001 2.43859336352222
5501 1.70609451235799
6001 1.54486782593073
6501 1.33560702626854
7001 1.31083801867358
7501 1.26504023960347
8001 1.24952217959523
8501 1.17912888404302
9001 1.15274284162009
9501 1.12772029192249
10001 1.10704603211169
10501 1.10392591704829
11001 1.05492038224803
11501 1.0523937202886
12001 1.03657762065644
12501 1.02993830571398
13001 0.990363004040214
13501 0.978812946393333
14001 0.942489355114732
14501 0.922553101061653
};
\addlegendentry{PM, freq=0.75}
\addplot [very thick, blue]
table {%
1 4.58257569495584
501 4.90212162354602
1001 7.38042236677409
1501 6.57885910981656
2001 6.77108499994818
2501 4.96635595828915
3001 4.13202068060852
3501 1.88772294192585
4001 1.85094120916954
4501 1.52006020648143
5001 1.46815827736752
5501 1.37967256489322
6001 1.35833848003052
6501 1.27461997183427
7001 1.25074460626515
7501 1.18226288934314
8001 1.12792526170701
8501 0.990517581199721
9001 0.886897903000212
9501 0.877926245232686
10001 0.867148580478477
10501 0.862018680058373
11001 0.857516798778188
11501 0.850723339442773
12001 0.846989010094336
12501 0.838539009698406
13001 0.836899788995548
13501 0.833935703278989
14001 0.833131607781453
14501 0.830531519189202
};
\addlegendentry{PM, freq=1.0}
\end{axis}

\end{tikzpicture}}~
    \scalebox{0.6}{
\begin{tikzpicture}

\begin{axis}[
legend cell align={left},
legend style={at={(0.03,0.03)}, anchor=south west, draw=white!80.0!black},
log basis y={10},
tick align=outside,
tick pos=left,
x grid style={white!69.01960784313725!black},
xlabel={Iterations},
xmajorgrids,
xmin=-1.15, xmax=24.15,
xtick style={color=black},
y grid style={white!69.01960784313725!black},
ylabel={Error},
ymajorgrids,
ymin=1e-08, ymax=10,
ymode=log,
ytick style={color=black}
]
\addplot [very thick, red]
table {%
0 5.3851648071345
1 4.22764632230482
2 3.91875809427871
3 3.76332975829825
4 3.65185094078203
5 3.57574300320251
6 1.70266938607677
7 1.60850018220994
8 1.44370006185875
9 1.31118395062504
10 1.27209141407386
11 1.00940088661844
12 1.0091512990393
13 1.00509852915826
14 0.00267288553759536
15 1.9304378158234e-05
16 1.05618557770011e-07
};
\addlegendentry{PCL, freq=0.5}
\addplot [very thick, green!50.0!black]
table {%
0 5.3851648071345
1 4.02722410762472
2 3.93870233640003
3 3.72567997967823
4 3.65875034801529
5 3.64333497001573
6 1.74459304075013
7 1.52838293263339
8 1.2861496754586
9 1.26417427531002
10 1.26377257414361
11 1.15203828513221
12 0.986969115873101
13 0.986808399982718
14 0.98112324794571
15 0.924423861790778
16 0.000180634078480116
17 7.97407147373106e-06
};
\addlegendentry{PCL, freq=0.75}
\addplot [very thick, blue]
table {%
0 5.3851648071345
1 4.00546530021366
2 3.99346392616205
3 3.83070262890467
4 3.74230054258209
5 3.72016126700787
6 3.71204249478843
7 3.43711286855341
8 1.79102682689816
9 1.57096812985229
10 1.56763807208392
11 1.52547692410586
12 1.21600985913525
13 1.17951922795914
14 1.1793732126547
15 1.16876292417318
16 0.972362077584013
17 0.972259446273644
18 0.972264285834266
19 0.930574223221935
20 0.000988106831606222
21 0.000121513367667119
22 3.39784944190817e-06
23 1.73180085942386e-09
};
\addlegendentry{PCL, freq=1.0}
\end{axis}

\end{tikzpicture}}
    \caption{Pipe domain, refinement level = 5.}
\end{subfigure}
\begin{subfigure}[b]{1.0\textwidth}
    \scalebox{0.6}{
\begin{tikzpicture}

\begin{axis}[
legend cell align={left},
legend style={at={(0.03,0.03)}, anchor=south west, draw=white!80.0!black},
log basis y={10},
tick align=outside,
tick pos=left,
x grid style={white!69.01960784313725!black},
xlabel={Iterations},
xmajorgrids,
xmin=-724, xmax=15226,
xtick style={color=black},
y grid style={white!69.01960784313725!black},
ylabel={Error},
ymajorgrids,
ymin=1e-08, ymax=10,
ymode=log,
ytick style={color=black}
]
\addplot [very thick, red]
table {%
1 4.58257569495584
501 4.77335713577189
1001 4.99411561933897
1501 5.48133510157433
2001 7.96192638795548
2501 10.2461005497946
3001 11.5165962930333
3501 10.6260435197886
4001 9.77586529826383
4501 9.04850259753215
5001 8.2049246172012
5501 6.8625358429677
6001 6.09901209882914
6501 4.79679513288098
7001 4.31173492323669
7501 4.1746477948187
8001 4.07114797652649
8501 3.89790515607694
9001 3.73457952311331
9501 3.58024633989553
10001 3.49453440299654
10501 3.47047312257489
11001 3.46554610882642
11501 3.45922380410911
12001 3.44216007069477
12501 3.41721927245665
13001 3.37427508408901
13501 3.33541465389305
14001 3.32768541450057
14501 3.32331482683253
};
\addlegendentry{PM, freq=0.5}
\addplot [very thick, green!50.0!black]
table {%
1 4.58257569495584
501 5.15770842020106
1001 5.36954614739687
1501 6.70204617549202
2001 8.62248792818286
2501 7.32343285586054
3001 7.00627628002178
3501 6.68381952765748
4001 6.21660592208296
4501 5.50201143555606
5001 5.20231443245399
5501 4.78871050078841
6001 4.88582975391575
6501 5.2154751994353
7001 5.55639753869097
7501 5.86739342627966
8001 5.92390769165515
8501 5.9591728317201
9001 6.0078626924431
9501 6.16258954438198
10001 6.30829602553001
10501 6.41747050176922
11001 6.49819986423388
11501 6.56698706942305
12001 6.58746420400245
12501 6.49570317276313
13001 6.33060687629044
13501 6.11056643140074
14001 6.08519017331392
14501 6.05308953220687
};
\addlegendentry{PM, freq=0.75}
\addplot [very thick, blue]
table {%
1 4.58257569495584
501 5.36960647402006
1001 5.1621490408029
1501 5.61685872093901
2001 6.3334112325903
2501 5.64815304457045
3001 7.60070395564615
3501 9.61968277736049
4001 9.59047360654087
4501 9.25263428785579
5001 9.47972691966334
5501 11.2388922508537
6001 14.103380756798
6501 14.6930407268589
7001 14.3385615372305
7501 12.0995456707815
8001 11.5298353543313
8501 11.7122188374969
9001 11.3333000543389
9501 10.4659485748252
10001 9.72308546830607
10501 9.47194906277495
11001 9.29979839562898
11501 9.09127956342412
12001 8.87297335208049
12501 8.42164359667643
13001 7.6934912195695
13501 7.6805751696107
14001 7.74478158410372
14501 7.8097765560055
};
\addlegendentry{PM, freq=1.0}
\end{axis}

\end{tikzpicture}}~
    \scalebox{0.6}{
\begin{tikzpicture}

\begin{axis}[
legend cell align={left},
legend style={at={(0.03,0.03)}, anchor=south west, draw=white!80.0!black},
log basis y={10},
tick align=outside,
tick pos=left,
x grid style={white!69.01960784313725!black},
xlabel={Iterations},
xmajorgrids,
xmin=-1.2, xmax=25.2,
xtick style={color=black},
y grid style={white!69.01960784313725!black},
ylabel={Error},
ymajorgrids,
ymin=1e-08, ymax=10,
ymode=log,
ytick style={color=black}
]
\addplot [very thick, red]
table {%
0 5.3851648071345
1 4.22960983499139
2 3.91902674476567
3 3.76444034448763
4 3.65205081825913
5 3.57700422344256
6 1.70345505553294
7 1.60945598505105
8 1.45169891582011
9 1.31095951652739
10 1.26610463129507
11 1.00928348609512
12 1.00906077588265
13 1.0020308304844
14 0.00195613575347889
15 1.60458412114243e-05
16 8.06558603133158e-08
};
\addlegendentry{PCL, freq=0.5}
\addplot [very thick, green!50.0!black]
table {%
0 5.3851648071345
1 4.03031217545211
2 3.95247835199076
3 3.74462442379136
4 3.67441240316079
5 3.65650246039362
6 3.34363555439219
7 1.71117216651359
8 1.28843708827929
9 1.26326880443103
10 1.26304681625825
11 1.22548926713829
12 0.988378706441165
13 0.987873884781091
14 0.987820451332542
15 0.352693171779884
16 0.000496197383484186
17 1.24622076174151e-05
18 1.05693171260339e-08
};
\addlegendentry{PCL, freq=0.75}
\addplot [very thick, blue]
table {%
0 5.3851648071345
1 4.00594003331974
2 4.23854078044448
3 3.92671710588411
4 3.87075212097486
5 3.71066269713676
6 3.67883240346815
7 3.66683315550819
8 3.64497375829194
9 2.79323898957768
10 1.53897718141758
11 1.49810504090724
12 1.42737678344782
13 1.17901170197867
14 1.1740884595893
15 1.1735833829967
16 1.17191728447493
17 0.97503341055117
18 0.961425978391634
19 0.961417810541457
20 0.944198538205218
21 0.00346235031152209
22 4.13811196863093e-05
23 6.28836049588519e-07
24 3.35861601917909e-09
};
\addlegendentry{PCL, freq=1.0}
\end{axis}

\end{tikzpicture}}
    \caption{Pipe domain, refinement level = 6.}
\end{subfigure}
\caption{Convergence comparison of PM and PCL for the curved pipe domain. }
\label{fig:helmpipe}
\end{figure}

We apply both PM and PCL to this problem. The PCL method exhibits a dramatic acceleration of convergence in terms of the number of iterations. PCL can achieve the same accuracy, such as $10^{-3}$ with only a few iterations, compared with more than $10^4$ iterations for PM. Moreover, PCL can achieve better accuracy than PM, as demonstrated the error in \Cref{fig:helmsquare,fig:helmpipe}.

PCL is also more scalable to mesh sizes. We apply $h$-refinement to the meshes in \Cref{fig:helmsquare,fig:helmpipe}. We see that the PM method requires more iterations to converge to the same accuracy. The number of iterations in the PCL method is less sensitive to the mesh sizes, although a larger sparse linear system is solved per iteration. Another observation is that PCL converges faster for smaller frequencies. However, PM does not exhibit such properties in our test cases.

\subsection{Physics Based Machine Learning: Static Problem}\label{sect:nn2}

In this example, we consider the equation
\begin{equation}\label{equ:poisson4}
    \begin{aligned}
        -\nabla \cdot (\mathbf{f}(u)\nabla u) =&\;  h(\bx) & \mbox{ in } \Omega\\
        u = &\; u_0(\bx) & \mbox{ on } \partial \Omega
    \end{aligned}
\end{equation}
where $\Omega=[0,1]^2$. We assume that $\mathbf{f}(u)\in \RR^{2\times 2}$ and has the form
\begin{equation*}
    \mathbf{f}(u) = \begin{bmatrix}
        f_1(u) & 0 \\
        0 & f_2(u)
    \end{bmatrix}
\end{equation*}

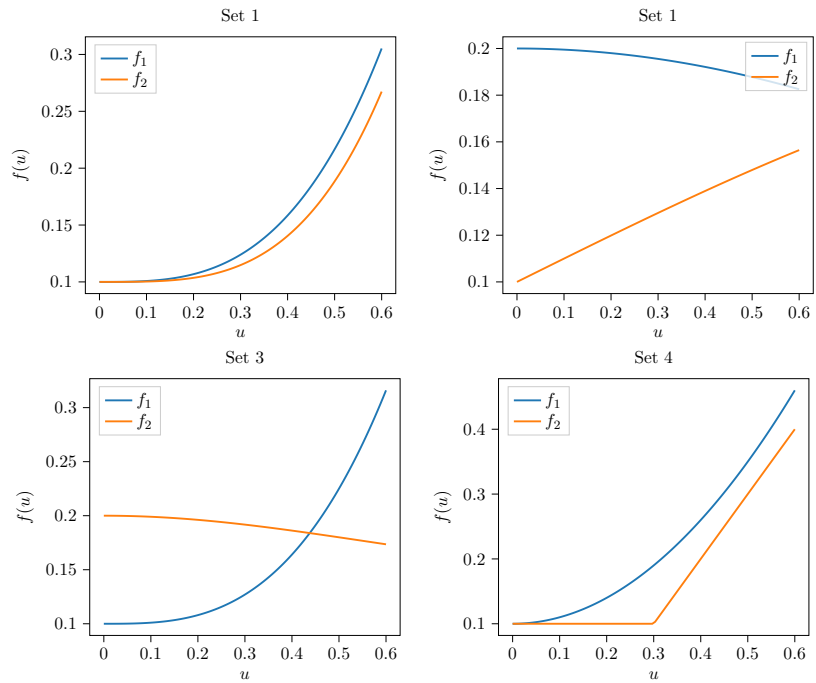
\begin{figure}[htpb]
    \centering
    \scalebox{0.6}{
\begin{tikzpicture}

\definecolor{color0}{rgb}{0.12156862745098,0.466666666666667,0.705882352941177}
\definecolor{color1}{rgb}{1,0.498039215686275,0.0549019607843137}

\begin{axis}[
legend cell align={left},
legend style={fill opacity=0.8, draw opacity=1, text opacity=1, at={(0.03,0.97)}, anchor=north west, draw=white!80.0!black},
tick align=outside,
tick pos=left,
x grid style={white!69.01960784313725!black},
xlabel={$u$},
title={Set 1},
xmin=-0.03, xmax=0.63,
xtick style={color=black},
y grid style={white!69.01960784313725!black},
ylabel={$f(u)$},
ymin=0.0897378376617387, ymax=0.315505409103487,
ytick style={color=black}
]
\addplot [very thick, color0]
table {%
0 0.1
0.00606060606060606 0.100000133597943
0.0121212121212121 0.100001145493844
0.0181818181818182 0.100004026017539
0.0242424242424242 0.100009821679226
0.0303030303030303 0.100019615834396
0.0363636363636364 0.100034519830068
0.0424242424242424 0.100055667753691
0.0484848484848485 0.100084212921215
0.0545454545454545 0.10012132534982
0.0606060606060606 0.100168189846008
0.0666666666666667 0.100226004505417
0.0727272727272727 0.100295979502382
0.0787878787878788 0.10037933609147
0.0848484848484848 0.100477305769001
0.0909090909090909 0.100591129558375
0.096969696969697 0.100722057393283
0.103030303030303 0.100871347579689
0.109090909090909 0.101040266322158
0.115151515151515 0.101230087303482
0.121212121212121 0.101442091308954
0.127272727272727 0.10167756588841
0.133333333333333 0.101937805050557
0.139393939393939 0.102224108985081
0.145454545454545 0.102537783808858
0.151515151515152 0.102880141333229
0.157575757575758 0.103252498849761
0.163636363636364 0.103656178932352
0.16969696969697 0.10409250925385
0.175757575757576 0.104562822415617
0.181818181818182 0.105068455788695
0.187878787878788 0.105610751365409
0.193939393939394 0.106191055620394
0.2 0.106810719380166
0.206060606060606 0.107471097700451
0.212121212121212 0.108173549750598
0.218181818181818 0.108919438704471
0.224242424242424 0.109710131637267
0.23030303030303 0.110546999427801
0.236363636363636 0.111431416665813
0.242424242424242 0.112364761563925
0.248484848484848 0.113348415873891
0.254545454545455 0.11438376480683
0.260606060606061 0.115472196957175
0.266666666666667 0.11661510423005
0.272727272727273 0.117813881771879
0.278787878787879 0.119069927903984
0.284848484848485 0.120384644058998
0.290909090909091 0.121759434719906
0.296969696969697 0.123195707361553
0.303030303030303 0.124694872394474
0.309090909090909 0.126258343110902
0.315151515151515 0.127887535632833
0.321212121212121 0.129583868862026
0.327272727272727 0.131348764431835
0.333333333333333 0.133183646660769
0.339393939393939 0.135089942507685
0.345454545454545 0.137069081528537
0.351515151515152 0.13912249583459
0.357575757575758 0.141251620052029
0.363636363636364 0.143457891282892
0.36969696969697 0.145742749067265
0.375757575757576 0.148107635346675
0.381818181818182 0.150553994428621
0.387878787878788 0.153083272952199
0.393939393939394 0.155696919854757
0.4 0.158396386339551
0.406060606060606 0.161183125844335
0.412121212121212 0.164058594010871
0.418181818181818 0.16702424865529
0.424242424242424 0.170081549739289
0.43030303030303 0.173231959342127
0.436363636363636 0.176476941633368
0.442424242424242 0.179817962846368
0.448484848484848 0.183256491252454
0.454545454545455 0.186793997135777
0.460606060606061 0.190431952768817
0.466666666666667 0.194171832388497
0.472727272727273 0.198015112172908
0.478787878787879 0.201963270218601
0.484848484848485 0.206017786518432
0.490909090909091 0.210180142939948
0.496969696969697 0.214451823204275
0.503030303030303 0.21883431286551
0.509090909090909 0.223329099290594
0.515151515151515 0.227937671639633
0.521212121212121 0.232661520846677
0.527272727272727 0.237502139600923
0.533333333333333 0.242461022328335
0.539393939393939 0.247539665173667
0.545454545454545 0.252739565982868
0.551515151515152 0.258062224285877
0.557575757575758 0.263509141279761
0.563636363636364 0.269081819812227
0.56969696969697 0.274781764365457
0.575757575757576 0.280610481040284
0.581818181818182 0.286569477540686
0.587878787878788 0.292660263158586
0.593939393939394 0.298884348758959
0.6 0.305243246765226
};
\addlegendentry{$f_1$}
\addplot [very thick, color1]
table {%
0 0.1
0.00606060606060606 0.100000017330305
0.0121212121212121 0.100000196070014
0.0181818181818182 0.100000810458137
0.0242424242424242 0.100002218278985
0.0303030303030303 0.100004843967386
0.0363636363636364 0.100009169287109
0.0424242424242424 0.100015727124881
0.0484848484848485 0.100025096961805
0.0545454545454545 0.100037901376094
0.0606060606060606 0.100054803234985
0.0666666666666667 0.10007650337474
0.0727272727272727 0.100103738641497
0.0787878787878788 0.100137280207937
0.0848484848484848 0.100177932106427
0.0909090909090909 0.100226529935821
0.096969696969697 0.100283939710074
0.103030303030303 0.100351056824494
0.109090909090909 0.100428805120843
0.115151515151515 0.100518136036442
0.121212121212121 0.100620027825426
0.127272727272727 0.100735484842459
0.133333333333333 0.100865536880996
0.139393939393939 0.101011238559482
0.145454545454545 0.101173668749977
0.151515151515152 0.101353930044525
0.157575757575758 0.101553148255274
0.163636363636364 0.101772471944952
0.16969696969697 0.102013071984729
0.175757575757576 0.102276141136902
0.181818181818182 0.102562893660173
0.187878787878788 0.102874564935536
0.193939393939394 0.103212411111059
0.2 0.103577708764
0.206060606060606 0.103971754578909
0.212121212121212 0.104395865040483
0.218181818181818 0.104851376140085
0.224242424242424 0.10533964309494
0.23030303030303 0.105862040079122
0.236363636363636 0.106419959965531
0.242424242424242 0.107014814078128
0.248484848484848 0.107648031953763
0.254545454545455 0.108321061113007
0.260606060606061 0.10903536683941
0.266666666666667 0.109792431966707
0.272727272727273 0.110593756673487
0.278787878787879 0.111440858284908
0.284848484848485 0.112335271081058
0.290909090909091 0.113278546111613
0.296969696969697 0.114272251016426
0.303030303030303 0.115317969851773
0.309090909090909 0.116417302921938
0.315151515151515 0.117571866615875
0.321212121212121 0.118783293248707
0.327272727272727 0.120053230907818
0.333333333333333 0.121383343303319
0.339393939393939 0.122775309622698
0.345454545454545 0.124230824389438
0.351515151515152 0.125751597325461
0.357575757575758 0.127339353217187
0.363636363636364 0.128995831785089
0.36969696969697 0.130722787556575
0.375757575757576 0.132521989742055
0.381818181818182 0.134395222114082
0.387878787878788 0.136344282889417
0.393939393939394 0.138370984613921
0.4 0.140477154050155
0.406060606060606 0.142664632067584
0.412121212121212 0.144935273535282
0.418181818181818 0.14729094721706
0.424242424242424 0.149733535668903
0.43030303030303 0.152264935138655
0.436363636363636 0.154887055467853
0.442424242424242 0.157601819995656
0.448484848484848 0.160411165464767
0.454545454545455 0.163317041929305
0.460606060606061 0.166321412664552
0.466666666666667 0.169426254078502
0.472727272727273 0.172633555625173
0.478787878787879 0.175945319719607
0.484848484848485 0.179363561654511
0.490909090909091 0.182890309518486
0.496969696969697 0.186527604115798
0.503030303030303 0.190277498887631
0.509090909090909 0.194142059834796
0.515151515151515 0.198123365441834
0.521212121212121 0.202223506602477
0.527272727272727 0.206444586546436
0.533333333333333 0.210788720767465
0.539393939393939 0.215258036952665
0.545454545454545 0.219854674913009
0.551515151515152 0.224580786515026
0.557575757575758 0.22943853561364
0.563636363636364 0.234430097986108
0.56969696969697 0.239557661267052
0.575757575757576 0.244823424884524
0.581818181818182 0.250229599997113
0.587878787878788 0.25577840943203
0.593939393939394 0.261472087624179
0.6 0.26731288055616
};
\addlegendentry{$f_2$}
\end{axis}

\end{tikzpicture}}~
    \scalebox{0.6}{
\begin{tikzpicture}

\definecolor{color0}{rgb}{0.12156862745098,0.466666666666667,0.705882352941177}
\definecolor{color1}{rgb}{1,0.498039215686275,0.0549019607843137}

\begin{axis}[
legend cell align={left},
legend style={fill opacity=0.8, draw opacity=1, text opacity=1, draw=white!80.0!black},
tick align=outside,
tick pos=left,
x grid style={white!69.01960784313725!black},
xlabel={$u$},
title={Set 1},
xmin=-0.03, xmax=0.63,
xtick style={color=black},
y grid style={white!69.01960784313725!black},
ylabel={$f(u)$},
ymin=0.095, ymax=0.205,
ytick style={color=black}
]
\addplot [very thick, color0]
table {%
0 0.2
0.00606060606060606 0.19999816345833
0.0121212121212121 0.199992653900779
0.0181818181818182 0.199983471529717
0.0242424242424242 0.199970616682421
0.0303030303030303 0.199954089831059
0.0363636363636364 0.199933891582676
0.0424242424242424 0.199910022679172
0.0484848484848485 0.199882483997271
0.0545454545454545 0.19985127654849
0.0606060606060606 0.199816401479108
0.0666666666666667 0.199777860070112
0.0727272727272727 0.199735653737162
0.0787878787878788 0.199689784030532
0.0848484848484848 0.199640252635053
0.0909090909090909 0.199587061370056
0.096969696969697 0.1995302121893
0.103030303030303 0.199469707180902
0.109090909090909 0.199405548567263
0.115151515151515 0.19933773870498
0.121212121212121 0.199266280084768
0.127272727272727 0.19919117533136
0.133333333333333 0.199112427203418
0.139393939393939 0.199030038593425
0.145454545454545 0.198944012527583
0.151515151515152 0.198854352165702
0.157575757575758 0.198761060801082
0.163636363636364 0.198664141860392
0.16969696969697 0.198563598903546
0.175757575757576 0.198459435623569
0.181818181818182 0.198351655846467
0.187878787878788 0.19824026353108
0.193939393939394 0.198125262768941
0.2 0.198006657784124
0.206060606060606 0.197884452933089
0.212121212121212 0.197758652704521
0.218181818181818 0.197629261719168
0.224242424242424 0.197496284729668
0.23030303030303 0.197359726620378
0.236363636363636 0.197219592407191
0.242424242424242 0.197075887237352
0.248484848484848 0.196928616389273
0.254545454545455 0.196777785272334
0.260606060606061 0.196623399426688
0.266666666666667 0.196465464523056
0.272727272727273 0.196303986362519
0.278787878787879 0.196138970876304
0.284848484848485 0.195970424125567
0.290909090909091 0.195798352301171
0.296969696969697 0.195622761723457
0.303030303030303 0.195443658842015
0.309090909090909 0.19526105023544
0.315151515151515 0.195074942611101
0.321212121212121 0.194885342804885
0.327272727272727 0.194692257780952
0.333333333333333 0.194495694631474
0.339393939393939 0.19429566057638
0.345454545454545 0.194092162963088
0.351515151515152 0.193885209266235
0.357575757575758 0.193674807087403
0.363636363636364 0.193460964154838
0.36969696969697 0.193243688323171
0.375757575757576 0.193022987573122
0.381818181818182 0.192798870011216
0.387878787878788 0.192571343869476
0.393939393939394 0.192340417505128
0.4 0.192106099400289
0.406060606060606 0.191868398161658
0.412121212121212 0.191627322520201
0.418181818181818 0.191382881330826
0.424242424242424 0.191135083572062
0.43030303030303 0.190883938345727
0.436363636363636 0.190629454876595
0.442424242424242 0.190371642512055
0.448484848484848 0.190110510721771
0.454545454545455 0.189846069097331
0.460606060606061 0.189578327351896
0.466666666666667 0.189307295319843
0.472727272727273 0.189032982956405
0.478787878787879 0.188755400337304
0.484848484848485 0.188474557658381
0.490909090909091 0.18819046523522
0.496969696969697 0.187903133502775
0.503030303030303 0.187612573014978
0.509090909090909 0.187318794444359
0.515151515151515 0.187021808581649
0.521212121212121 0.186721626335386
0.527272727272727 0.186418258731515
0.533333333333333 0.186111716912981
0.539393939393939 0.18580201213932
0.545454545454545 0.185489155786246
0.551515151515152 0.185173159345234
0.557575757575758 0.184854034423097
0.563636363636364 0.184531792741558
0.56969696969697 0.184206446136825
0.575757575757576 0.183878006559147
0.581818181818182 0.183546486072386
0.587878787878788 0.183211896853564
0.593939393939394 0.182874251192422
0.6 0.182533561490968
};
\addlegendentry{$f_1$}
\addplot [very thick, color1]
table {%
0 0.1
0.00606060606060606 0.100606056895871
0.0121212121212121 0.101212091530767
0.0181818181818182 0.101818081644531
0.0242424242424242 0.102424004978641
0.0303030303030303 0.103029839277028
0.0363636363636364 0.103635562286893
0.0424242424242424 0.104241151759526
0.0484848484848485 0.10484658545112
0.0545454545454545 0.105451841123591
0.0606060606060606 0.106056896545394
0.0666666666666667 0.106661729492339
0.0727272727272727 0.107266317748409
0.0787878787878788 0.107870639106571
0.0848484848484848 0.1084746713696
0.0909090909090909 0.109078392350887
0.096969696969697 0.109681779875257
0.103030303030303 0.110284811779783
0.109090909090909 0.110887465914602
0.115151515151515 0.111489720143724
0.121212121212121 0.112091552345849
0.127272727272727 0.11269294041518
0.133333333333333 0.113293862262231
0.139393939393939 0.113894295814643
0.145454545454545 0.114494219017989
0.151515151515152 0.115093609836592
0.157575757575758 0.115692446254327
0.163636363636364 0.116290706275433
0.16969696969697 0.11688836792532
0.175757575757576 0.11748540925138
0.181818181818182 0.118081808323785
0.187878787878788 0.118677543236301
0.193939393939394 0.119272592107088
0.2 0.119866933079506
0.206060606060606 0.120460544322915
0.212121212121212 0.121053404033479
0.218181818181818 0.121645490434967
0.224242424242424 0.122236781779553
0.23030303030303 0.122827256348611
0.236363636363636 0.12341689245352
0.242424242424242 0.124005668436454
0.248484848484848 0.12459356267118
0.254545454545455 0.125180553563853
0.260606060606061 0.125766619553808
0.266666666666667 0.126351739114354
0.272727272727273 0.12693589075356
0.278787878787879 0.127519053015051
0.284848484848485 0.12810120447879
0.290909090909091 0.12868232376187
0.296969696969697 0.129262389519294
0.303030303030303 0.129841380444764
0.309090909090909 0.130419275271461
0.315151515151515 0.130996052772826
0.321212121212121 0.131571691763341
0.327272727272727 0.132146171099305
0.333333333333333 0.132719469679615
0.339393939393939 0.133291566446536
0.345454545454545 0.133862440386476
0.351515151515152 0.134432070530759
0.357575757575758 0.135000435956397
0.363636363636364 0.135567515786853
0.36969696969697 0.136133289192813
0.375757575757576 0.136697735392947
0.381818181818182 0.137260833654677
0.387878787878788 0.137822563294934
0.393939393939394 0.138382903680919
0.4 0.138941834230865
0.406060606060606 0.139499334414786
0.412121212121212 0.140055383755235
0.418181818181818 0.140609961828057
0.424242424242424 0.141163048263137
0.43030303030303 0.141714622745149
0.436363636363636 0.142264665014303
0.442424242424242 0.142813154867089
0.448484848484848 0.143360072157015
0.454545454545455 0.143905396795356
0.460606060606061 0.144449108751882
0.466666666666667 0.1449911880556
0.472727272727273 0.145531614795485
0.478787878787879 0.146070369121214
0.484848484848485 0.146607431243889
0.490909090909091 0.147142781436773
0.496969696969697 0.147676400036007
0.503030303030303 0.148208267441334
0.509090909090909 0.148738364116821
0.515151515151515 0.149266670591576
0.521212121212121 0.149793167460462
0.527272727272727 0.15031783538481
0.533333333333333 0.15084065509313
0.539393939393939 0.151361607381818
0.545454545454545 0.151880673115863
0.551515151515152 0.152397833229547
0.557575757575758 0.152913068727148
0.563636363636364 0.153426360683638
0.56969696969697 0.153937690245375
0.575757575757576 0.154447038630799
0.581818181818182 0.154954387131117
0.587878787878788 0.155459717110997
0.593939393939394 0.155963010009249
0.6 0.156464247339504
};
\addlegendentry{$f_2$}
\end{axis}

\end{tikzpicture}}
    \scalebox{0.6}{
\begin{tikzpicture}

\definecolor{color0}{rgb}{0.12156862745098,0.466666666666667,0.705882352941177}
\definecolor{color1}{rgb}{1,0.498039215686275,0.0549019607843137}

\begin{axis}[
legend cell align={left},
legend style={fill opacity=0.8, draw opacity=1, text opacity=1, at={(0.03,0.97)}, anchor=north west, draw=white!80.0!black},
tick align=outside,
tick pos=left,
x grid style={white!69.01960784313725!black},
xlabel={$u$},
title={Set 3},
xmin=-0.03, xmax=0.63,
xtick style={color=black},
y grid style={white!69.01960784313725!black},
ylabel={$f(u)$},
ymin=0.0892, ymax=0.3268,
ytick style={color=black}
]
\addplot [very thick, color0]
table {%
0 0.1
0.00606060606060606 0.100000222611793
0.0121212121212121 0.100001780894343
0.0181818181818182 0.100006010518407
0.0242424242424242 0.100014247154743
0.0303030303030303 0.100027826474107
0.0363636363636364 0.100048084147258
0.0424242424242424 0.100076355844951
0.0484848484848485 0.100113977237944
0.0545454545454545 0.100162283996995
0.0606060606060606 0.10022261179286
0.0666666666666667 0.100296296296296
0.0727272727272727 0.100384673178062
0.0787878787878788 0.100489078108913
0.0848484848484848 0.100610846759607
0.0909090909090909 0.100751314800902
0.096969696969697 0.100911817903553
0.103030303030303 0.10109369173832
0.109090909090909 0.101298271975958
0.115151515151515 0.101526894287225
0.121212121212121 0.101780894342878
0.127272727272727 0.102061607813674
0.133333333333333 0.10237037037037
0.139393939393939 0.102708517683724
0.145454545454545 0.103077385424493
0.151515151515152 0.103478309263433
0.157575757575758 0.103912624871303
0.163636363636364 0.104381667918858
0.16969696969697 0.104886774076857
0.175757575757576 0.105429279016056
0.181818181818182 0.106010518407213
0.187878787878788 0.106631827921084
0.193939393939394 0.107294543228428
0.2 0.108
0.206060606060606 0.108749533906559
0.212121212121212 0.109544480618861
0.218181818181818 0.110386175807663
0.224242424242424 0.111275955143724
0.23030303030303 0.112215154297799
0.236363636363636 0.113205108940646
0.242424242424242 0.114247154743023
0.248484848484848 0.115342627375685
0.254545454545455 0.116492862509391
0.260606060606061 0.117699195814898
0.266666666666667 0.118962962962963
0.272727272727273 0.120285499624343
0.278787878787879 0.121668141469794
0.284848484848485 0.123112224170075
0.290909090909091 0.124619083395943
0.296969696969697 0.126190054818154
0.303030303030303 0.127826474107466
0.309090909090909 0.129529676934636
0.315151515151515 0.13130099897042
0.321212121212121 0.133141775885578
0.327272727272727 0.135053343350864
0.333333333333333 0.137037037037037
0.339393939393939 0.139094192614854
0.345454545454545 0.141226145755071
0.351515151515152 0.143434232128447
0.357575757575758 0.145719787405738
0.363636363636364 0.148084147257701
0.36969696969697 0.150528647355094
0.375757575757576 0.153054623368673
0.381818181818182 0.155663410969196
0.387878787878788 0.15835634582742
0.393939393939394 0.161134763614102
0.4 0.164
0.406060606060606 0.16695339065587
0.412121212121212 0.16999627125247
0.418181818181818 0.173129977460556
0.424242424242424 0.176355844950886
0.43030303030303 0.179675209394218
0.436363636363636 0.183089406461307
0.442424242424242 0.186599771822912
0.448484848484848 0.19020764114979
0.454545454545455 0.193914350112697
0.460606060606061 0.197721234382391
0.466666666666667 0.20162962962963
0.472727272727273 0.205640871525169
0.478787878787879 0.209756295739767
0.484848484848485 0.21397723794418
0.490909090909091 0.218305033809166
0.496969696969697 0.222741019005482
0.503030303030303 0.227286529203885
0.509090909090909 0.231942900075131
0.515151515151515 0.23671146728998
0.521212121212121 0.241593566519186
0.527272727272727 0.246590533433509
0.533333333333333 0.251703703703704
0.539393939393939 0.256934413000529
0.545454545454545 0.262283996994741
0.551515151515152 0.267753791357097
0.557575757575758 0.273345131758355
0.563636363636364 0.279059353869271
0.56969696969697 0.284897793360603
0.575757575757576 0.290861785903108
0.581818181818182 0.296952667167543
0.587878787878788 0.303171772824665
0.593939393939394 0.309520438545232
0.6 0.316
};
\addlegendentry{$f_1$}
\addplot [very thick, color1]
table {%
0 0.2
0.00606060606060606 0.199996327040329
0.0121212121212121 0.199985309780014
0.0181818181818182 0.199966953073364
0.0242424242424242 0.199941265004956
0.0303030303030303 0.199908256880734
0.0363636363636364 0.199867943215583
0.0424242424242424 0.199820341717387
0.0484848484848485 0.199765473267617
0.0545454545454545 0.199703361898484
0.0606060606060606 0.199634034766697
0.0666666666666667 0.199557522123894
0.0727272727272727 0.199473857283788
0.0787878787878788 0.199383076586114
0.0848484848484848 0.199285219357427
0.0909090909090909 0.199180327868852
0.096969696969697 0.199068447290855
0.103030303030303 0.198949625645126
0.109090909090909 0.198823913753675
0.115151515151515 0.198691365185239
0.121212121212121 0.198552036199095
0.127272727272727 0.198405985686402
0.133333333333333 0.19825327510917
0.139393939393939 0.198093968436982
0.145454545454545 0.19792813208158
0.151515151515152 0.197755834829443
0.157575757575758 0.197577147772481
0.163636363636364 0.197392144236961
0.16969696969697 0.197200899710807
0.175757575757576 0.197003491769401
0.181818181818182 0.1968
0.187878787878788 0.196590505924927
0.193939393939394 0.196375092923643
0.2 0.196153846153846
0.206060606060606 0.195926852471724
0.212121212121212 0.195694200351494
0.218181818181818 0.195455979804355
0.224242424242424 0.195212282296985
0.23030303030303 0.194963200669713
0.236363636363636 0.194708829054477
0.242424242424242 0.194449262792715
0.248484848484848 0.194184598353283
0.254545454545455 0.193914933250543
0.260606060606061 0.193640365962716
0.266666666666667 0.193360995850622
0.272727272727273 0.193076923076923
0.278787878787879 0.192788248525953
0.284848484848485 0.192495073724264
0.290909090909091 0.192197500761963
0.296969696969697 0.191895632214946
0.303030303030303 0.191589571068124
0.309090909090909 0.19127942063971
0.315151515151515 0.190965284506666
0.321212121212121 0.190647266431378
0.327272727272727 0.190325470289639
0.333333333333333 0.19
0.339393939393939 0.189670959454563
0.345454545454545 0.18933845245127
0.351515151515152 0.189002582627742
0.357575757575758 0.18866345339673
0.363636363636364 0.188321167883212
0.36969696969697 0.187975828863181
0.375757575757576 0.187627538704175
0.381818181818182 0.187276399307559
0.387878787878788 0.186922512052616
0.393939393939394 0.186565977742448
0.4 0.186206896551724
0.406060606060606 0.185845367976288
0.412121212121212 0.18548149078464
0.418181818181818 0.1851153629713
0.424242424242424 0.184747081712062
0.43030303030303 0.184376743321143
0.436363636363636 0.184004443210219
0.442424242424242 0.183630275849358
0.448484848484848 0.183254334729825
0.454545454545455 0.182876712328767
0.460606060606061 0.182497500075755
0.466666666666667 0.182116788321168
0.472727272727273 0.181734666306404
0.478787878787879 0.181351222135899
0.484848484848485 0.180966542750929
0.490909090909091 0.180580713905168
0.496969696969697 0.180193820141978
0.503030303030303 0.179805944773407
0.509090909090909 0.179417169860856
0.515151515151515 0.179027576197388
0.521212121212121 0.178637243291644
0.527272727272727 0.178246249353337
0.533333333333333 0.177854671280277
0.539393939393939 0.177462584646902
0.545454545454545 0.177070063694268
0.551515151515152 0.176677181321467
0.557575757575758 0.176284009078428
0.563636363636364 0.17589061716006
0.56969696969697 0.175497074401708
0.575757575757576 0.175103448275862
0.581818181818182 0.174709804890096
0.587878787878788 0.174316208986188
0.593939393939394 0.173922723940373
0.6 0.173529411764706
};
\addlegendentry{$f_2$}
\end{axis}

\end{tikzpicture}}~
    \scalebox{0.6}{
\begin{tikzpicture}

\definecolor{color0}{rgb}{0.12156862745098,0.466666666666667,0.705882352941177}
\definecolor{color1}{rgb}{1,0.498039215686275,0.0549019607843137}

\begin{axis}[
legend cell align={left},
legend style={fill opacity=0.8, draw opacity=1, text opacity=1, at={(0.03,0.97)}, anchor=north west, draw=white!80.0!black},
tick align=outside,
tick pos=left,
x grid style={white!69.01960784313725!black},
xlabel={$u$},
title={Set 4},
xmin=-0.03, xmax=0.63,
xtick style={color=black},
y grid style={white!69.01960784313725!black},
ylabel={$f(u)$},
ymin=0.082, ymax=0.478,
ytick style={color=black}
]
\addplot [very thick, color0]
table {%
0 0.1
0.00606060606060606 0.100036730945822
0.0121212121212121 0.100146923783287
0.0181818181818182 0.100330578512397
0.0242424242424242 0.10058769513315
0.0303030303030303 0.100918273645546
0.0363636363636364 0.101322314049587
0.0424242424242424 0.101799816345271
0.0484848484848485 0.102350780532599
0.0545454545454545 0.10297520661157
0.0606060606060606 0.103673094582186
0.0666666666666667 0.104444444444444
0.0727272727272727 0.105289256198347
0.0787878787878788 0.106207529843893
0.0848484848484848 0.107199265381084
0.0909090909090909 0.108264462809917
0.096969696969697 0.109403122130395
0.103030303030303 0.110615243342516
0.109090909090909 0.111900826446281
0.115151515151515 0.11325987144169
0.121212121212121 0.114692378328742
0.127272727272727 0.116198347107438
0.133333333333333 0.117777777777778
0.139393939393939 0.119430670339761
0.145454545454545 0.121157024793388
0.151515151515152 0.122956841138659
0.157575757575758 0.124830119375574
0.163636363636364 0.126776859504132
0.16969696969697 0.128797061524334
0.175757575757576 0.13089072543618
0.181818181818182 0.133057851239669
0.187878787878788 0.135298438934803
0.193939393939394 0.137612488521579
0.2 0.14
0.206060606060606 0.142460973370064
0.212121212121212 0.144995408631772
0.218181818181818 0.147603305785124
0.224242424242424 0.150284664830119
0.23030303030303 0.153039485766758
0.236363636363636 0.155867768595041
0.242424242424242 0.158769513314968
0.248484848484848 0.161744719926538
0.254545454545455 0.164793388429752
0.260606060606061 0.16791551882461
0.266666666666667 0.171111111111111
0.272727272727273 0.174380165289256
0.278787878787879 0.177722681359045
0.284848484848485 0.181138659320478
0.290909090909091 0.184628099173554
0.296969696969697 0.188191000918274
0.303030303030303 0.191827364554637
0.309090909090909 0.195537190082645
0.315151515151515 0.199320477502296
0.321212121212121 0.20317722681359
0.327272727272727 0.207107438016529
0.333333333333333 0.211111111111111
0.339393939393939 0.215188246097337
0.345454545454545 0.219338842975207
0.351515151515152 0.22356290174472
0.357575757575758 0.227860422405877
0.363636363636364 0.232231404958678
0.36969696969697 0.236675849403122
0.375757575757576 0.24119375573921
0.381818181818182 0.245785123966942
0.387878787878788 0.250449954086318
0.393939393939394 0.255188246097337
0.4 0.26
0.406060606060606 0.264885215794307
0.412121212121212 0.269843893480257
0.418181818181818 0.274876033057851
0.424242424242424 0.279981634527089
0.43030303030303 0.285160697887971
0.436363636363636 0.290413223140496
0.442424242424242 0.295739210284665
0.448484848484848 0.301138659320478
0.454545454545455 0.306611570247934
0.460606060606061 0.312157943067034
0.466666666666667 0.317777777777778
0.472727272727273 0.323471074380165
0.478787878787879 0.329237832874197
0.484848484848485 0.335078053259871
0.490909090909091 0.34099173553719
0.496969696969697 0.346978879706152
0.503030303030303 0.353039485766758
0.509090909090909 0.359173553719008
0.515151515151515 0.365381083562902
0.521212121212121 0.371662075298439
0.527272727272727 0.37801652892562
0.533333333333333 0.384444444444444
0.539393939393939 0.390945821854913
0.545454545454545 0.397520661157025
0.551515151515152 0.404168962350781
0.557575757575758 0.41089072543618
0.563636363636364 0.417685950413223
0.56969696969697 0.42455463728191
0.575757575757576 0.431496786042241
0.581818181818182 0.438512396694215
0.587878787878788 0.445601469237833
0.593939393939394 0.452764003673095
0.6 0.46
};
\addlegendentry{$f_1$}
\addplot [very thick, color1]
table {%
0 0.1
0.00606060606060606 0.1
0.0121212121212121 0.1
0.0181818181818182 0.1
0.0242424242424242 0.1
0.0303030303030303 0.1
0.0363636363636364 0.1
0.0424242424242424 0.1
0.0484848484848485 0.1
0.0545454545454545 0.1
0.0606060606060606 0.1
0.0666666666666667 0.1
0.0727272727272727 0.1
0.0787878787878788 0.1
0.0848484848484848 0.1
0.0909090909090909 0.1
0.096969696969697 0.1
0.103030303030303 0.1
0.109090909090909 0.1
0.115151515151515 0.1
0.121212121212121 0.1
0.127272727272727 0.1
0.133333333333333 0.1
0.139393939393939 0.1
0.145454545454545 0.1
0.151515151515152 0.1
0.157575757575758 0.1
0.163636363636364 0.1
0.16969696969697 0.1
0.175757575757576 0.1
0.181818181818182 0.1
0.187878787878788 0.1
0.193939393939394 0.1
0.2 0.1
0.206060606060606 0.1
0.212121212121212 0.1
0.218181818181818 0.1
0.224242424242424 0.1
0.23030303030303 0.1
0.236363636363636 0.1
0.242424242424242 0.1
0.248484848484848 0.1
0.254545454545455 0.1
0.260606060606061 0.1
0.266666666666667 0.1
0.272727272727273 0.1
0.278787878787879 0.1
0.284848484848485 0.1
0.290909090909091 0.1
0.296969696969697 0.1
0.303030303030303 0.103030303030303
0.309090909090909 0.109090909090909
0.315151515151515 0.115151515151515
0.321212121212121 0.121212121212121
0.327272727272727 0.127272727272727
0.333333333333333 0.133333333333333
0.339393939393939 0.139393939393939
0.345454545454545 0.145454545454545
0.351515151515152 0.151515151515152
0.357575757575758 0.157575757575758
0.363636363636364 0.163636363636364
0.36969696969697 0.16969696969697
0.375757575757576 0.175757575757576
0.381818181818182 0.181818181818182
0.387878787878788 0.187878787878788
0.393939393939394 0.193939393939394
0.4 0.2
0.406060606060606 0.206060606060606
0.412121212121212 0.212121212121212
0.418181818181818 0.218181818181818
0.424242424242424 0.224242424242424
0.43030303030303 0.23030303030303
0.436363636363636 0.236363636363636
0.442424242424242 0.242424242424242
0.448484848484848 0.248484848484849
0.454545454545455 0.254545454545455
0.460606060606061 0.260606060606061
0.466666666666667 0.266666666666667
0.472727272727273 0.272727272727273
0.478787878787879 0.278787878787879
0.484848484848485 0.284848484848485
0.490909090909091 0.290909090909091
0.496969696969697 0.296969696969697
0.503030303030303 0.303030303030303
0.509090909090909 0.309090909090909
0.515151515151515 0.315151515151515
0.521212121212121 0.321212121212121
0.527272727272727 0.327272727272727
0.533333333333333 0.333333333333333
0.539393939393939 0.339393939393939
0.545454545454545 0.345454545454545
0.551515151515152 0.351515151515152
0.557575757575758 0.357575757575758
0.563636363636364 0.363636363636364
0.56969696969697 0.36969696969697
0.575757575757576 0.375757575757576
0.581818181818182 0.381818181818182
0.587878787878788 0.387878787878788
0.593939393939394 0.393939393939394
0.6 0.4
};
\addlegendentry{$f_2$}
\end{axis}

\end{tikzpicture}}
    \caption{Test function sets. Note in test function set 4, $f_2(u)$ does not have continuous derivatives at $u=0.3$. }
    \label{fig:test}
\end{figure}

We test four sets of $\mathbf{f}(u)$ (\Cref{fig:test})
\begin{enumerate}
    \item Test Function Set 1
    \begin{equation*}
        f_1(u) = 0.1 + u^{3.1},\qquad f_2(u) = 0.1 + u^{3.5}
    \end{equation*}
    \item Test Function Set 2
     \begin{equation*}
        f_1(u) = 0.1 + 0.1\cos(u),\qquad f_2(u) = 0.1 + 0.1\sin(u)
    \end{equation*}
    \item Test Function Set 3
    \begin{equation*}
        f_1(u) = 0.1 + u^3,\qquad f_2(u) = 0.1 + \frac{0.1}{1+u^2}
    \end{equation*}
    \item Test Function Set 4
    \begin{equation*}
        f_1(u) = 0.1 + u^2,\qquad f_2(u) = 0.1 + \max(0, u-0.3)
    \end{equation*}
\end{enumerate}
We use a deep neural network to approximate $\mathbf{f}$, which takes $u$ as input and outputs $f_1(u;\theta)$ and $f_2(u;\theta)$ ($\theta$ is the weights and biases)
\begin{equation*}
    \mathbf{f}_{\theta}(u) = 
    \begin{bmatrix}
        f_1(u; \theta) & 0 \\
        0 & f_2(u; \theta)
    \end{bmatrix}
\end{equation*}
The neural network is a $k$ hidden layer (we test multiple options: $k=1$, \dots, $5$) and has 20 neurons per layer and \texttt{tanh} activation functions since it is continuously differentiable and bounded so it does not create extreme values in the intermediate layers. 

The full solution on the grid is solved numerically and used as an observation, but the codes developed for this problem also work for sparse solutions. For PCL, we discretize \Cref{equ:poisson4} with the finite difference method on a  $31\times 31$ uniform grid. We demonstrate that even with this simple fully-connected architecture, the neural network can learn $\mathbf{f}(u)$ reasonably well. We compare PCL with PM; we vary the neural network architectures (different hidden layers) and penalty parameters $\lambda$ for PM, which we denote PM-$\lambda$. The error is reported by 
\begin{equation*}
    \sqrt{\sum_{i=1}^{N+1} (f_1(u_i) - f_1(u_i;\theta))^2 + \sum_{i=1}^{N+1} (f_2(u_i) - f_2(u_i;\theta))^2}
\end{equation*}
where $f_1(u;\theta)$ and $f_2(u;\theta)$ are the first and second outputs of the neural network with parameter $\theta$. $u_i=0.6\frac{i-1}{N}$ for $i=1,2,\ldots, N+1$ ($N=99$ in our case); we choose $0.6$ since $[0,0.6]$ is the region where the solution $u$ is. We run the experiments for sufficiently long time so that the parameters nearly converge to a local minimum. 

\Cref{tab:func1,tab:func2,tab:func3,tab:func4} show the error for the four test function sets. Additionally, we show the relative error for each $u$ in \Cref{fig:fuerr} for PCL and PM-0.01. We can see that PCL outperforms PM  in almost all cases, regardless of the neural network architectures. PM performs better than PCL only in a few cases, such as in the test function set 1 with $\lambda=0.01$ and $\lambda=1$ and the test function set 3 with $\lambda=0.01$; for other $\lambda$ and neural network architectures, PM fails to converge to the true functions. The sensitivity to the neural network hyperparameters and penalty parameters can be a difficulty when applying the penalty method. PM introduces more dependent variables for optimization, and, therefore, it is likely that the optimizer converges to a local minimum that is nonphysical. Note that the accuracy does not necessarily improve as we increase the number of hidden layers because larger neural networks are harder to optimize.

\begin{table}[htpb]
	\centering
	\resizebox{\textwidth}{!}{%
		\begin{tabular}{@{}llllll@{}}
			\toprule
			Method\textbackslash{}Hidden Layers & 1 & 2 & 3 & 4 & 5 \\ \midrule
			PCL & \textbf{$1.6\times 10^{-2}$} & \textbf{$1.1\times 10^{-2}$} & \textbf{$2.7\times 10^{-2}$} & $4.5\times 10^{-2}$ & \textbf{$3.5\times 10^{-2}$} \\
			PM-0.0 & 1.7 & $9.7\times 10^{-1}$ & $7.8\times 10^{-1}$ & \textbf{$1.9\times 10^{-3}$} & 2.8 \\
			PM-0.1 & 7.2 & 5.3 & 3.7 & 8.8 & 2.1 \\
			PM-1 & 1.8 & 5.2 & 2.9 & $4.2\times 10^{-3}$ & 1.7 \\
			PM-10 & 7.4 & 2.8 & 1.9 & 4.7 & 1.8 \\
			PM-100 & 2.1 & 2.6 & 3.0 & 3.9 & 2.6 \\ \bottomrule
		\end{tabular}%
	}
	\caption{Error for test function set 1.}
	\label{tab:func1}
\end{table}

\begin{table}[htpb]
	\centering
	\resizebox{\textwidth}{!}{%
		\begin{tabular}{@{}llllll@{}}
			\toprule
			Method\textbackslash{}Hidden Layers & 1 & 2 & 3 & 4 & 5 \\ \midrule
			PCL & \textbf{$4.1\times 10^{-5}$} & \textbf{$2.4\times 10^{-4}$} & \textbf{$6.8\times 10^{-4}$} & \textbf{$1.1\times 10^{-3}$} & \textbf{$3.8\times 10^{-3}$} \\
			PM-0.0 & $8.0\times 10^{-1}$ & 1.5 & $9.3\times 10^{-1}$ & $1.9\times 10^{-3}$ & 4.9 \\
			PM-0.1 & 3.8 & 1.9 & $1.1\times 10^1$ & 6.7 & 1.6 \\
			PM-1 & 2.3 & $2.8\times 10^{-3}$ & $2.4\times 10^1$ & 2.7 & 2.7 \\
			PM-10 & 3.2 & 2.7 & $1.6\times 10^{-2}$ & 2.6 & 2.9 \\
			PM-100 & 3.2 & 2.1 & 2.6 & 2.5 & 2.8 \\ \bottomrule
		\end{tabular}%
	}
	\caption{Error for test function set 2.}
	\label{tab:func2}
\end{table}

\begin{table}[htpb]
	\centering
	\resizebox{\textwidth}{!}{%
		\begin{tabular}{@{}llllll@{}}
			\toprule
			Method\textbackslash{}Hidden Layers & 1 & 2 & 3 & 4 & 5 \\ \midrule
			PCL & \textbf{$8.7\times 10^{-3}$} & $3.0\times 10^{-2}$ & $2.9\times 10^{-2}$ & $3.4\times 10^{-2}$ & $5.4\times 10^{-2}$ \\
			PM-0.0 & 8.0 & 5.6 & 4.5 & $1.1\times 10^{-2}$ & 1.0 \\
			PM-0.1 & 5.3 & $3.8\times 10^1$ & 8.0 & 6.4 & 7.2 \\
			PM-1 & $1.9\times 10^1$ & \textbf{$3.0\times 10^{-2}$} & 4.3 & $2.0\times 10^1$ & 2.3 \\
			PM-10 & 3.6 & 4.4 & $2.7\times 10^{-2}$ & $1.1\times 10^1$ & 4.6 \\
			PM-100 & 1.6 & 5.3 & 4.2 & $8.9\times 10^{-2}$ & 4.4 \\ \bottomrule
		\end{tabular}%
	}
	\caption{Error for test function set 3.}
	\label{tab:func3}
\end{table}

\begin{table}[htpb]
	\centering
	\resizebox{\textwidth}{!}{%
		\begin{tabular}{@{}llllll@{}}
			\toprule
			Method\textbackslash{}Hidden Layers & 1 & 2 & 3 & 4 & 5 \\ \midrule
			PCL & $3.1\times 10^{-1}$ & $8.4\times 10^{-1}$ & $2.3\times 10^{-1}$ & $1.2\times 10^{-1}$ & $1.4\times 10^{-1}$ \\
			PM-0.0 & $2.1\times 10^{-1}$ & 4.8 & $7.3\times 10^{-1}$ & 2.5 & 4.1 \\
			PM-0.1 & 3.7 & $2.0\times 10^1$ & 7.2 & 4.8 & 5.4 \\
			PM-1 & 4.3 & $1.7\times 10^1$ & 3.2 & 5.9 & 4.3 \\
			PM-10 & 4.4 & 3.3 & 7.1 & 2.3 & 9.7 \\
			PM-100 & 3.1 & $8.2\times 10^{-1}$ & 3.1 & 3.8 & 3.1 \\ \bottomrule
		\end{tabular}%
	}
	\caption{Error for test function set 4.}
	\label{tab:func4}
\end{table}

\begin{figure}[htpb]
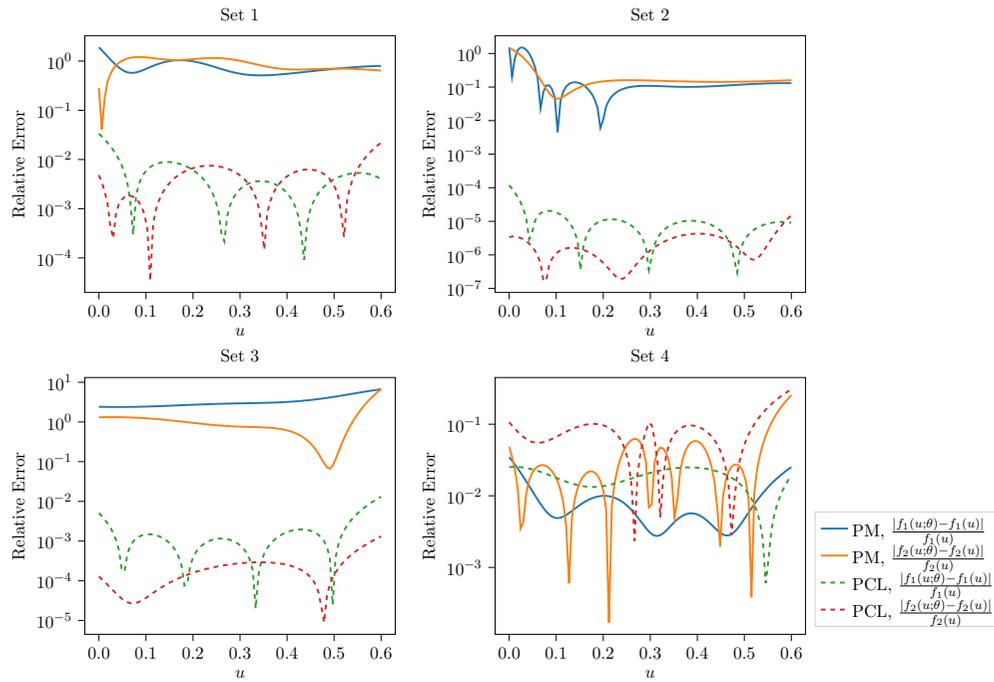

\centering
    \scalebox{0.6}{\input{figures/fuerr1}}~
    \scalebox{0.6}{\input{figures/fuerr2}}$\qquad\qquad\qquad\qquad$
    \scalebox{0.6}{\input{figures/fuerr3}}~
    \scalebox{0.6}{\input{figures/fuerr4}}
    \caption{The relative error for PCL and PM-0.01 for different test function sets.}
    \label{fig:fuerr}
\end{figure}

\subsection{Physics Based Machine Learning: Time-Dependent PDEs}

Finally, we consider the inverse problem of time-dependent PDEs. We consider the two-phase flow problem \cite{kleinstreuer2017two}, in which we derive the governing equations from the conservation of mass and momentum (Darcy's law) for each phase
\begin{equation}\label{eqn:two_phase_mass}
      \frac{\partial }{{\partial t}}(\phi {S_i}{\rho _i}) + \nabla  \cdot ({\rho _i}{\mathbf{v}_i}) = {\rho _i}{q_i},\quad 
      i = 1,2
\end{equation}
Here \(\phi\) denotes porosity, \(S_i\) denotes the saturation of the \(i\)-th phase, \(\rho_i\) denotes fluid density, \(\mathbf{v}_i\) denotes the volumetric velocity, and \(q_i\) stands for injection rates of the source. The saturation of two phases should add up to 1, that is,
\begin{equation} \label{eqn:two_phase_sat}
      S_{1} + S_{2} = 1
\end{equation}
Darcy's law has the following form
\begin{equation}\label{eqn:two_phase_darcy}
      {\mathbf{v}_i} = - \frac{{K{k_{ri}}}}{{{\tilde{\mu}_i}}}(\nabla {P_i} - g{\rho _i}\nabla Z), \quad
      i=1, 2
\end{equation}
where \(K\) is the permeability tensor and for simplicity we assume that \(K\) is a spatially varying scalar value (\Cref{fig:K}). \(\tilde{\mu}_i\) is the viscosity, \(P_i\) is the fluid pressure, \(g\) is the gravitational acceleration constant, and \(Z\) is the vector in the downward vertical direction. The relative permeability is written as \(k_{ri}\), a dimensionless measure of the effective permeability of that phase, and a function of $S_1$ (or $S_2$ equivalently since $S_1+S_2=1$). The relative permeability depends on the saturation $S_i$ and many empirical relations are used for approximations. For example, the Corey correlation \cite{brooks1966properties} assumes that the relative permeability is a power law in saturation. Another commonly used model is the LET-type correlation \cite{lomeland2005new}, which has the following form (recall that $S_2=1-S_1$ so that $k_{ri}$ can be viewed as a function of $S_1$ alone)
\begin{equation}\label{equ:kr12}
\begin{aligned}
	k_{r1}(S_1) =& \frac{k_{r1}^o S_1^{L_1}}{S_1^{L_1} + E_1 S_2^{T_1}}\\
	k_{r2}(S_1) =& \frac{ S_2^{L_2}}{S_2^{L_2} + E_2 S_1^{T_2}}
\end{aligned}
\end{equation}
Here $k_{r1}^o$, $L_i$, $E_i$, $T_i$ are degrees of freedoms that can be used to control the magnitude and shape of the measured relative permeability curves. In this problem, we assume the true model for relative permeability is \Cref{equ:kr12} with parameter $k_{r1}^o=0.6$, $L_1=L_2=1.8$, $E_1=E_2=2.1$, $T_1=T_2=2.3$. Finally, for numerical simulation of the system, we adopt the zero initial condition and no-flow boundary condition
\begin{equation}
    {\mathbf{v}_i } \cdot \mathbf{n}  = 0, i = 1, 2
\end{equation}

\begin{figure}[hbt]
\centering
  \includegraphics[width=0.6\textwidth]{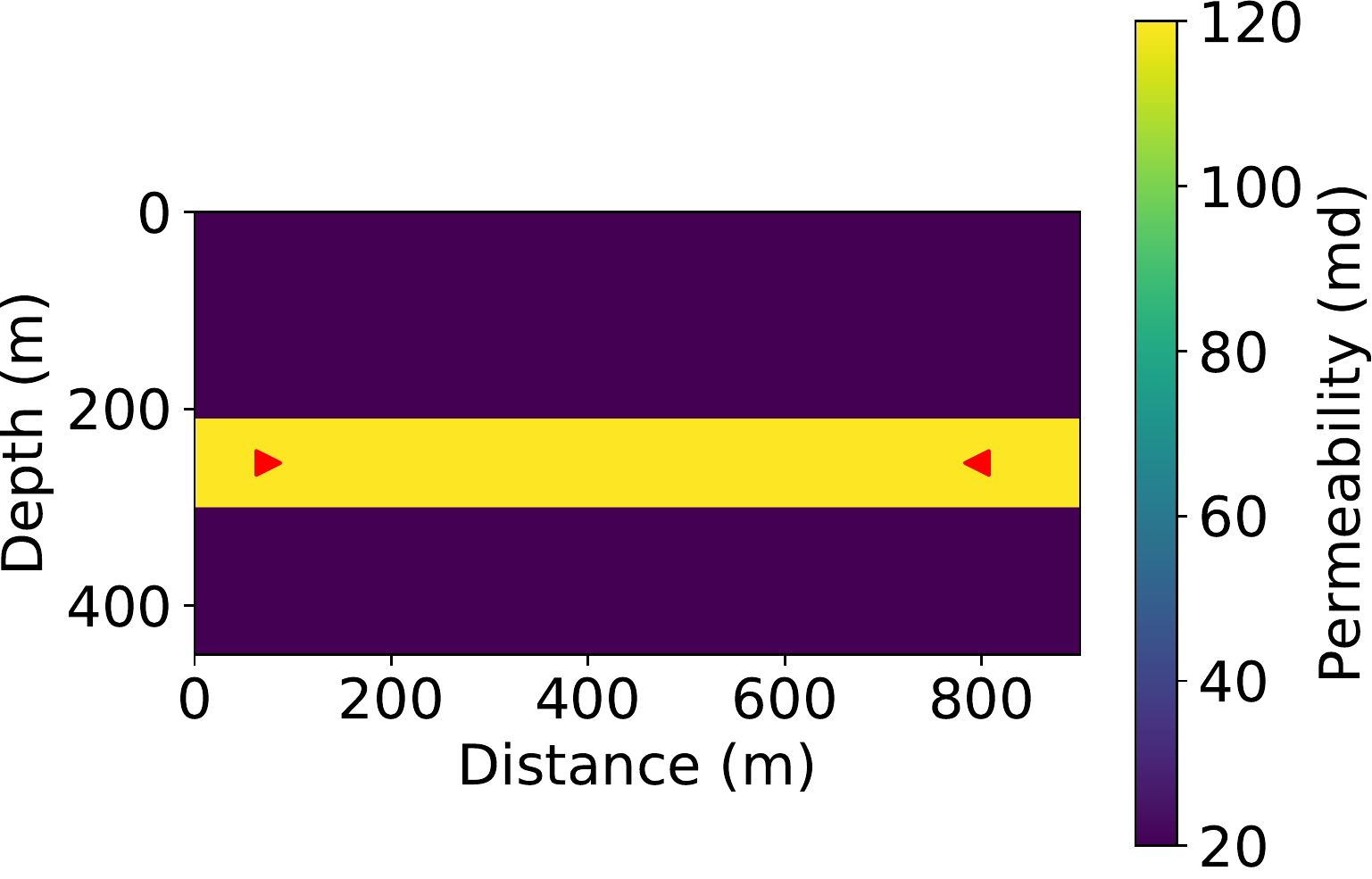}
  \caption{The scalar value of the permeability $K$ at each location. The triangular markers are point sources.}
  \label{fig:K}
\end{figure}

We adopt a nonlinear implicit scheme for solving the PDE system since it is unconditionally stable. We refer readers to \cite{li2019time} for the numerical scheme and implementation details. The nonlinear implicit step involves solving 
\begin{equation}\label{equ:major}
\phi (S_2^{n + 1} - S_2^n) - \nabla \cdot \left( {{m_{2}}(S_2^{n + 1})K\nabla \Psi _2^n} \right) \Delta t = 
\left(q_2^n + q_1^n \frac{m_2(S^{n+1}_2)}{m_1(S^{n+1}_2)}\right) 
\Delta t
\end{equation}
where $n$ stands for the time step $n$ and $\Delta t$ is the time step. $\Psi _2^n$ is the fluid potential and can be computed using $S_2^n$ (see \cite{li2019time} for details) and 
$$m_i(s) = \frac{k_{ri}(s)}{\tilde \mu_i}$$
We apply the Newton-Raphson algorithm to solve the nonlinear equation \Cref{equ:major}. This means we need to compute $m_i'(s)$. We implement an efficient custom operator for solving \Cref{equ:major} where the linear system is solved with algebraic multi-grid methods \cite{demidov2019amgcl,hu2011development}. 

\begin{figure}[hbt]
\centering
  \includegraphics[width=0.8\textwidth]{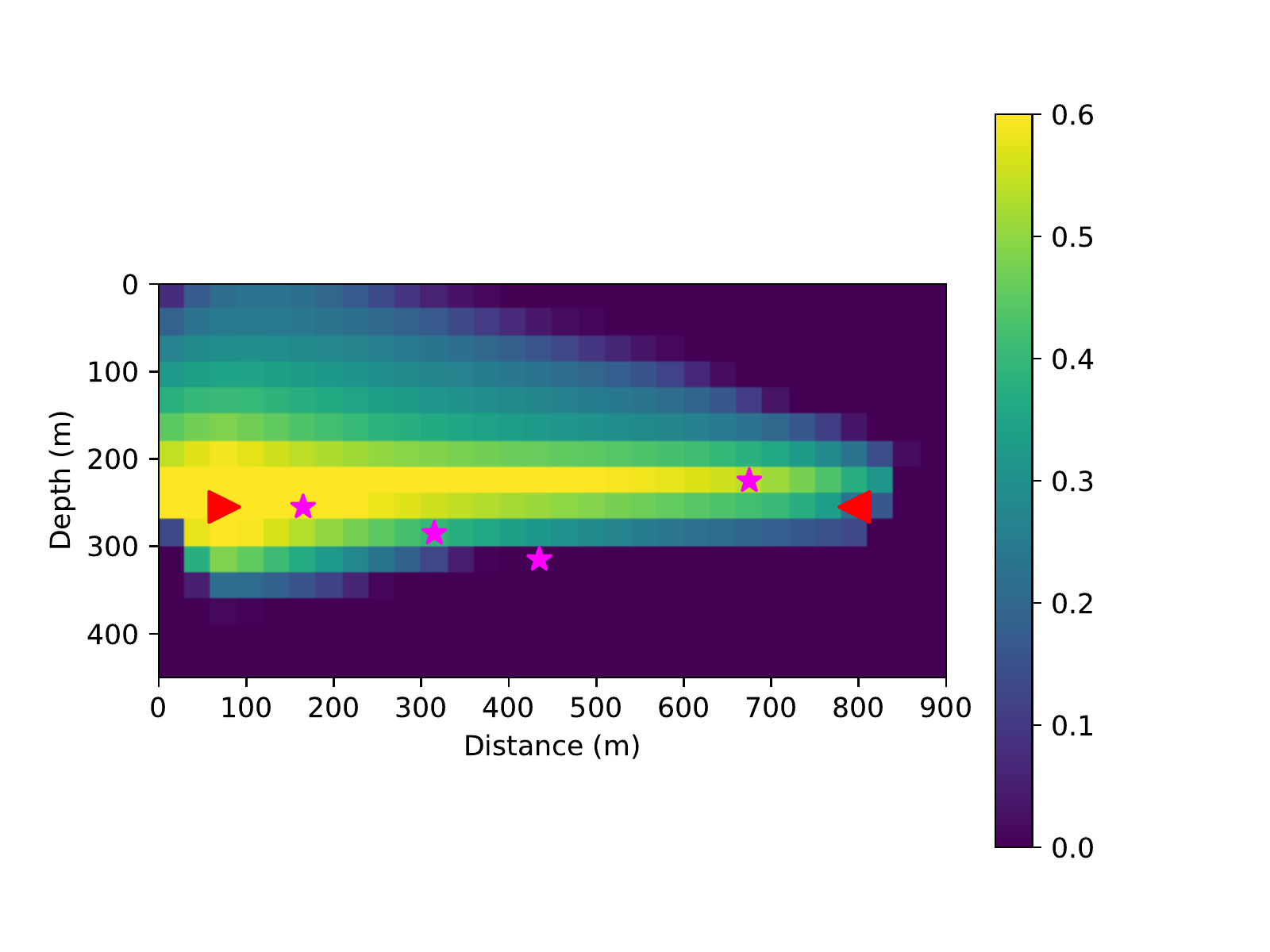}
  \caption{Source functions and locations where observations are collected. The triangular markers are point sources; to be more specific, the left source is an injection while the right source is an absorption (e.g., production wells). The four magenta markers are locations where observations are collected. The background scalar values show the saturation at the terminal time of the simulation.}
  \label{fig:sat}
\end{figure}

We assume that we can measure the saturation $S_1$, denoted as $\hat S_1^n(\bx_i)$ ($n$ denotes the time step), at multiple locations $\{\bx_i\}_{i\in \mathcal{I}_{\mathrm{obs}}}$ (the magenta markers in \Cref{fig:sat}). The saturation model \Cref{equ:kr12} is unknown to us. In this setting, we cannot measure or compute $K$ and $S_i$ directly at all locations, and, therefore, it is impossible to build the saturation model $K_{ri}$ by direct curve fitting. Instead, we use two neural networks to approximate the saturation models
\[
    f_1(S_1; \theta_1) \approx k_{r1}(S_1)\qquad f_2(S_1; \theta_2) \approx k_{r2}(S_1)
\]
where $\theta_1$ and $\theta_2$ are the weights and biases for two neural networks. The neural network is constrained to output values between $(0,1)$ by applying 
$$ x\mapsto \frac{\tanh(x)+1}{2} $$
to the outputs of the neural networks in the last layer. 
 
Corresponding to \Cref{equ:inverse}, we have ($S_i^n$ and $P_i^n$ are solution vectors)
\begin{align*}
	u &= [S_1^n\ S_2^n\ P_1^n\ P_2^n]_{n=1,2,\ldots}\\
	\theta &= [\theta_1\ \theta_2]\\
	L(u) &= \sum_{i\in \mathcal{I}_{\mathrm{obs}}} \sum_{n}(S_1^n(\bx_i) - \hat S_1^n(\bx_i))^2
\end{align*}
and $F(u, \theta)=0$ is given by the systems of \Cref{eqn:two_phase_mass,eqn:two_phase_sat,eqn:two_phase_darcy}. 
  
\begin{table}[htpb]
\centering
\begin{tabular}{@{}cc@{}}
\toprule
Parameter & Value \\ \midrule
$g$ & 9.8 \\
$h$ & 30 \\
Time Steps & 50 \\
$\Delta t$ & 20 \\
Domain & $[0,900]\times [0,450]$ \\
$\rho_1$ & 501.9 \\
$\rho_2$ & 1053 \\
$\tilde \mu_1$ & 0.1 \\
$\tilde \mu_2$ & 1 \\
Neural Network & Fully-connected 1-20-20-20-1 \\ \bottomrule
\end{tabular}
\caption{Parameters used in the simulation. }
\label{tab:parameters}
\end{table}

We list the parameters for carrying out the numerical simulation in \Cref{tab:parameters}. \Cref{tab:comparison_param} compares the number of independent variables for PCL and the penalty method. We can see that the penalty method leads to a huge number of unknown variables, and the associated optimization problem is difficult to solve. However, the independent variables in PCL is limited to weights and biases of the neural network. 

\begin{table}[hbtp]
\centering
  \begin{tabular}{c|cc}
  \toprule
   Method &PCL &  Penalty Method\\
    \midrule
   \# Parameters & 1802 & 24302 \\
   \bottomrule
  \end{tabular}
  \caption{Comparison of independent variables for penalty method and PCL.}
  \label{tab:comparison_param}
\end{table}

Therefore, we select PCL to solve the two-phase flow inverse modeling problem. As shown in \Cref{fig:rmsq}, PCL successfully learns the saturation model in the region of physical interest. The estimation is not accurate near the boundary since there are few data points in this region. Finally, we show the true and reconstructed saturation $S_1$ in \Cref{fig:satnn} using the estimated neural network saturation model. We see the reconstructed saturation is nearly the same as true saturation. 

\begin{figure}[htpb]
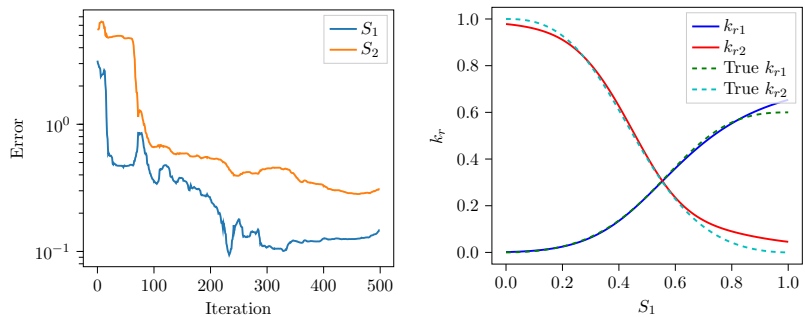

\centering
    \scalebox{0.6}{\input{figures/errwo}}~
    \scalebox{0.6}{\input{figures/krwo}}
    \caption{Left: root mean squared error of estimated and true $k_{ri}$. Right: visualization of the neural network approximation to $k_{ri}$; this plot is generated at the 491th iteration.}
    \label{fig:rmsq}
\end{figure}

\begin{figure}[hbtp]
\centering
    \includegraphics[width=1.0\textwidth]{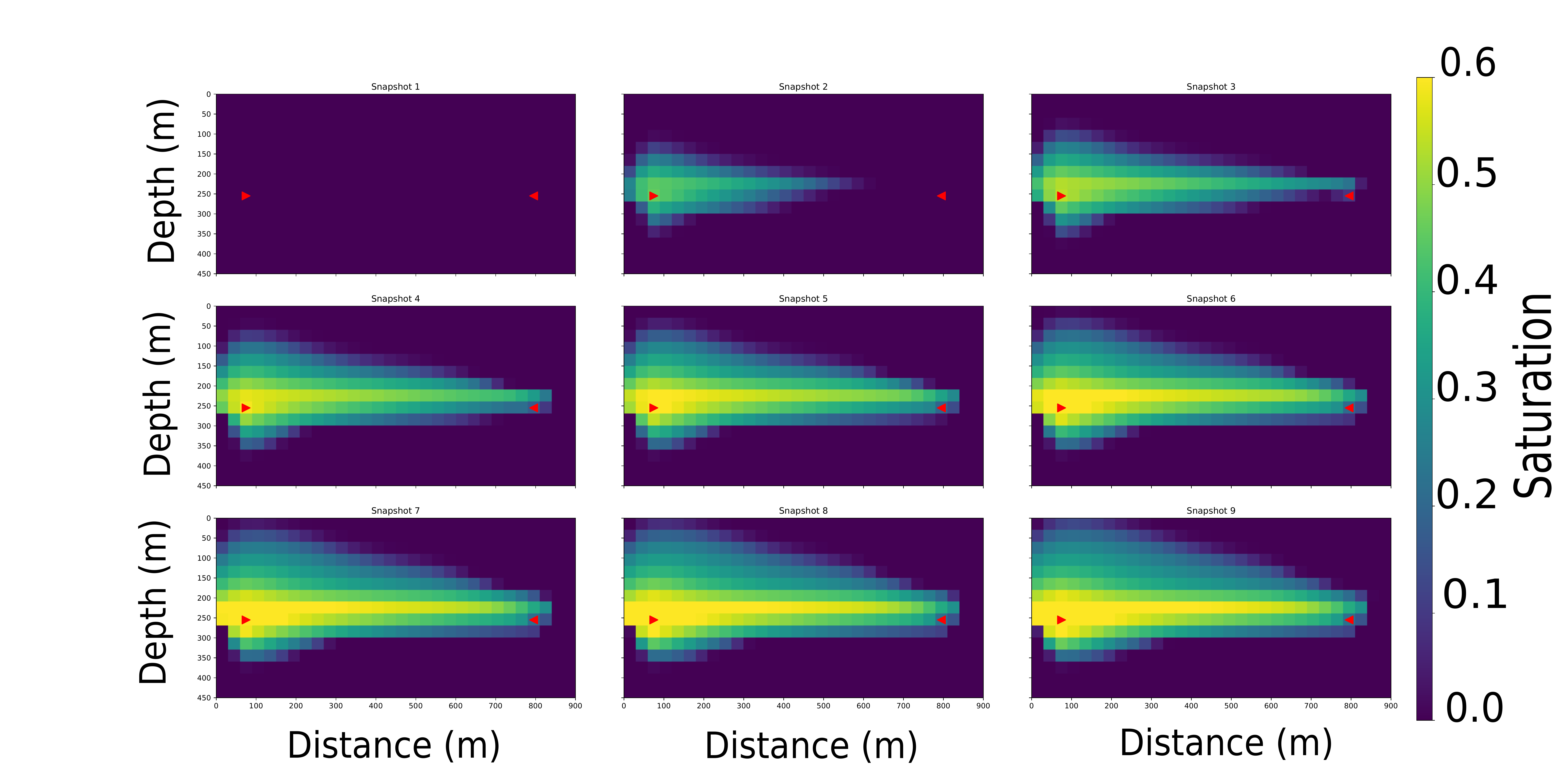}
    \includegraphics[width=1.0\textwidth]{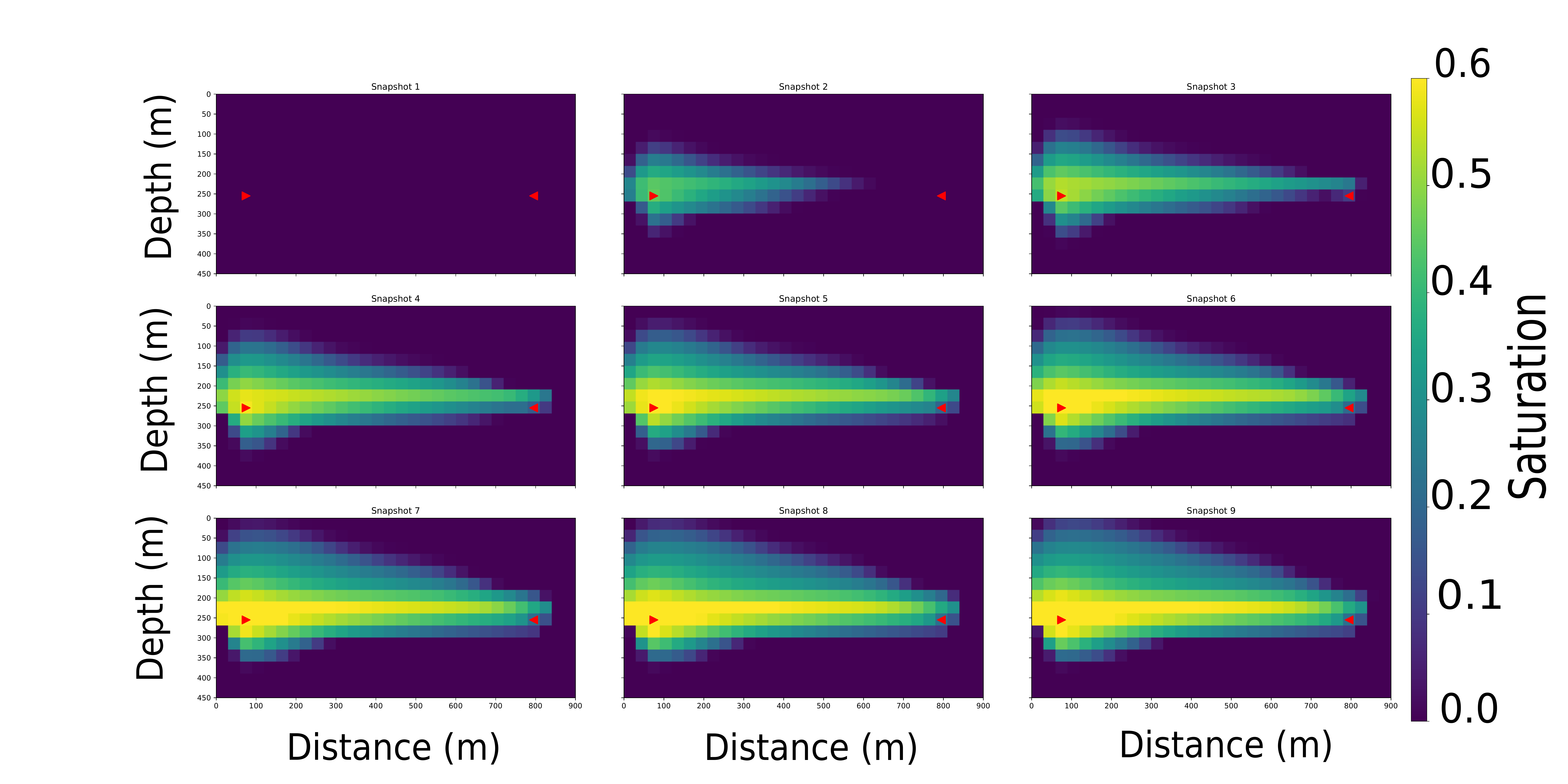}
    \caption{Saturation $S_1$ at Day 0, 5, 10, 15, 20, 25, 30, 35, 40 (from left to right, top to bottom). The top figure shows the true saturation computed with \Cref{equ:kr12} and the bottom figure shows the saturation computed with the neural network saturation model. The distance ranges from 0~m to 900~m and the depth ranges from 0~m to 450~m.}
    \label{fig:satnn}
\end{figure}

\section{Conclusion}\label{sect:conc}

We have compared two methods for data-driven modeling with sparse observations: the penalty method (PM) and the physics constrained learning method (PCL). The difference is that PM enforces the physical constraints by including a penalty term in the loss function, considering the solution vectors as independent variables, while PCL solves the system of PDEs directly. We proved that PM is more ill-conditioned than PCL in a model problem: the condition number of PM is the square of that of PCL asymptotically, which explains the slow convergence or divergence of PM during the optimization. The increase in the number of independent variables in PM can be very large for dynamical systems. 
Additionally, we proposed an approach for implementing PCL through automatic differentiation, and thus not only getting rid of the time-consuming and difficult process of deriving and implementing the gradients, but also leverages computational graph-based optimization for performance. 

The proposed PCL method shows superior performance over the PM method in terms of iterations to converge. We see almost a $10^4$ speed-up in the Helmholtz equation example. Additionally, when we approximate the unknown function with a neural network, PM is sensitive to the choice of penalty parameters and neural network architectures, and may converge to a bad local minimum. In contrast, PCL consistently converges to the true solution for different neural network architectures and test function sets. 

However, PCL suffers from some limitations compared to PM. Firstly, the memory and computation costs per iteration for PCL are greater than the penalty method; for nonlinear problems, we may need an expensive iterative method such as the Newton-Raphson algorithm for solving the nonlinear system.
Secondly, for time-dependent problems, PCL must solve the equation sequentially in time. This calculation can be computationally challenging for long time horizons. However, for PM long time horizon problems are also very challenging: it has a large number of independent variables. 
Thirdly, from the perspective of implementation, PM only requires first-order derivatives, while PCL requires extra Jacobian matrices. As a result, PCL requires more implementation work and is generally more expensive in computation and storage per iteration. 

The idea of enforcing physical constraints by solving the PDEs numerically is very crucial for learning parameters or training deep neural network surrogate models in stiff problems. To this end, physics constraint learning puts forth an effective and efficient approach that can be integrated into an automatic differentiation framework and interacts between deep learning and computational engineering. PCL allows for the potential to supplement the best of available scientific computing tools with deep learning techniques for inverse modeling. 
 
\newpage

\begin{appendix}
    
\section{NURBS Domain}\label{sect:nurbs}

The same geometry can be described in many ways using NURBS. We have implemented a module in \texttt{IGACS.jl} with many built-in meshes and utilities to generate NURBS mesh (e.g., from boundaries described by NURBS curves). See documentation or source codes for details. We present the NURBS data structure for the meshes we have used in this paper. The meshes are subjected to refinement or affine transformation for the numerical examples in \Cref{sect:num}. 

\

\begin{minipage}{\textwidth}
  \begin{minipage}[b]{0.49\textwidth}
  \centering
  \captionof{table}{Weights ($w$) and control points $(x,y)$ for the square mesh. The knot vectors are $u=(0,0,0,1,1,1)$, $v=(0,0,0,1,1,1)$ and the degrees are $p=q=2$.}
      \begin{tabular}{@{}lll@{}}
\toprule
$w$ & $x$  & $y$  \\ \midrule
1 & $-1$ & 1  \\
1 & $-1$ & 0  \\
1 & $-1$ & $-1$ \\
1 & 0  & 1  \\
1 & 0  & 0  \\
1 & 0  & $-1$ \\
1 & 1  & 1  \\
1 & 1  & 0  \\
1 & 1  & $-1$ \\ \bottomrule
\end{tabular}
  \end{minipage}~
  \begin{minipage}[b]{0.49\textwidth}
  \centering
\captionof{table}{Weights ($w$) and control points $(x,y)$ for the curved pipe mesh. The knot vectors are $u=(0,0,1,1)$, $v=(0,0,0,1,1,1)$ and the degrees are $p=1$,$q=2$.}
\vspace{0.6cm}
\begin{tabular}{@{}lll@{}}
\toprule
$w$                    & $x$ & $y$ \\ \midrule
1                    & 1 & 0 \\
1                    & 2 & 0 \\
$\frac{\sqrt{2}}{2}$ & 1 & 1 \\
$\frac{\sqrt{2}}{2}$ & 2 & 2 \\
1                    & 0 & 1 \\
1                    & 0 & 2 \\ \bottomrule
\end{tabular}
\end{minipage}
\end{minipage}

\begin{figure}
\centering
    \scalebox{0.6}{
\begin{tikzpicture}

\begin{axis}[
tick align=outside,
tick pos=left,
x grid style={white!69.01960784313725!black},
xmin=-1.10332661290323, xmax=1.10332661290323,
xtick style={color=black},
y grid style={white!69.01960784313725!black},
ymin=-1.10446428571429, ymax=1.10446428571429,
ytick style={color=black}
]
\addplot [only marks, draw=red, fill=red, colormap/viridis]
table{%
x                      y
-1 1
-1 0
-1 -1
0 1
0 0
0 -1
1 1
1 0
1 -1
};
\addplot [very thick, black, dashed]
table {%
-1 1
-0.894736842105263 1
-0.789473684210526 1
-0.68421052631579 1
-0.578947368421053 1
-0.473684210526316 1
-0.368421052631579 1
-0.263157894736842 1
-0.157894736842105 1
-0.0526315789473685 1
0.0526315789473684 1
0.157894736842105 1
0.263157894736842 1
0.368421052631579 1
0.473684210526316 1
0.578947368421053 1
0.684210526315789 1
0.789473684210526 1
0.894736842105263 1
0.99999998 1
};
\addplot [very thick, black, dashed]
table {%
-1 -0.99999998
-0.894736842105263 -0.99999998
-0.789473684210526 -0.99999998
-0.68421052631579 -0.99999998
-0.578947368421053 -0.99999998
-0.473684210526316 -0.99999998
-0.368421052631579 -0.99999998
-0.263157894736842 -0.99999998
-0.157894736842105 -0.99999998
-0.0526315789473685 -0.99999998
0.0526315789473684 -0.99999998
0.157894736842105 -0.99999998
0.263157894736842 -0.99999998
0.368421052631579 -0.99999998
0.473684210526316 -0.99999998
0.578947368421053 -0.99999998
0.684210526315789 -0.99999998
0.789473684210526 -0.99999998
0.894736842105263 -0.99999998
0.99999998 -0.99999998
};
\addplot [very thick, black, dashed]
table {%
-1 1
-1 0.894736842105263
-1 0.789473684210526
-1 0.68421052631579
-1 0.578947368421053
-1 0.473684210526316
-1 0.368421052631579
-1 0.263157894736842
-1 0.157894736842105
-1 0.0526315789473685
-1 -0.0526315789473684
-1 -0.157894736842105
-1 -0.263157894736842
-1 -0.368421052631579
-1 -0.473684210526316
-1 -0.578947368421053
-1 -0.684210526315789
-1 -0.789473684210526
-1 -0.894736842105263
-1 -0.99999998
};
\addplot [very thick, black, dashed]
table {%
0.99999998 1
0.99999998 0.894736842105263
0.99999998 0.789473684210526
0.99999998 0.684210526315789
0.99999998 0.578947368421053
0.99999998 0.473684210526316
0.99999998 0.368421052631579
0.99999998 0.263157894736842
0.99999998 0.157894736842105
0.99999998 0.0526315789473685
0.99999998 -0.0526315789473683
0.99999998 -0.157894736842105
0.99999998 -0.263157894736842
0.99999998 -0.368421052631579
0.99999998 -0.473684210526316
0.99999998 -0.578947368421053
0.99999998 -0.684210526315789
0.99999998 -0.789473684210526
0.99999998 -0.894736842105263
0.99999998 -0.99999998
};
\end{axis}

\end{tikzpicture}}~
    \scalebox{0.6}{
\begin{tikzpicture}

\begin{axis}[
tick align=outside,
tick pos=left,
x grid style={white!69.01960784313725!black},
xmin=-0.482181508967223, xmax=2.48218150896722,
xtick style={color=black},
y grid style={white!69.01960784313725!black},
ymin=-0.104464285714286, ymax=2.10446428571429,
ytick style={color=black}
]
\addplot [only marks, draw=red, fill=red, colormap/viridis]
table{%
x                      y
1 1.01465363575695e-17
2 2.02930727151391e-17
1 1
2 2
1.99673461754274e-16 1
3.99346923508548e-16 2
};
\addplot [very thick, black, dashed]
table {%
1 1.01465363575695e-17
0.997146573479169 0.0754898072507301
0.988272658157257 0.152699551861786
0.972963412819013 0.230959297963037
0.950898076293276 0.309504197874192
0.921873766915625 0.3874903326185
0.885826160804624 0.464017254888373
0.842844924697569 0.538156513396938
0.793181981282237 0.608984683361729
0.737251184692945 0.6756187465345
0.6756187465345 0.737251184692945
0.608984683361729 0.793181981282237
0.538156513396938 0.842844924697569
0.464017254888373 0.885826160804624
0.3874903326185 0.921873766915625
0.309504197874192 0.950898076293276
0.230959297963037 0.972963412819013
0.152699551861786 0.988272658157257
0.0754898072507303 0.997146573479169
1.41421359358866e-08 1
};
\addplot [very thick, black, dashed]
table {%
1.99999999 2.02930726136737e-17
1.99429313698687 0.150979613746562
1.97654530643179 0.305399102196577
1.94592681590839 0.461918593616481
1.90179614307757 0.619008392653343
1.84374752461251 0.774980661362097
1.77165231275099 0.928034505136572
1.68568984096669 1.07631302141231
1.58636395463265 1.21796936063361
1.47450236201338 1.35123748631281
1.35123748631281 1.47450236201338
1.21796936063361 1.58636395463265
1.07631302141231 1.68568984096669
0.928034505136572 1.77165231275099
0.774980661362097 1.84374752461251
0.619008392653343 1.90179614307757
0.461918593616481 1.94592681590839
0.305399102196577 1.97654530643179
0.150979613746563 1.99429313698687
2.82842717303518e-08 1.99999999
};
\addplot [very thick, black, dashed]
table {%
1 1.01465363575695e-17
1.05263157894737 1.06805645869153e-17
1.10526315789474 1.12145928162611e-17
1.15789473684211 1.17486210456068e-17
1.21052631578947 1.22826492749526e-17
1.26315789473684 1.28166775042983e-17
1.31578947368421 1.33507057336441e-17
1.36842105263158 1.38847339629899e-17
1.42105263157895 1.44187621923356e-17
1.47368421052632 1.49527904216814e-17
1.52631578947368 1.54868186510272e-17
1.57894736842105 1.60208468803729e-17
1.63157894736842 1.65548751097187e-17
1.68421052631579 1.70889033390645e-17
1.73684210526316 1.76229315684102e-17
1.78947368421053 1.8156959797756e-17
1.84210526315789 1.86909880271018e-17
1.89473684210526 1.92250162564475e-17
1.94736842105263 1.97590444857933e-17
1.99999999 2.02930726136737e-17
};
\addplot [very thick, black, dashed]
table {%
1.41421359358866e-08 1
1.48864588798806e-08 1.05263157894737
1.56307818238747e-08 1.10526315789474
1.63751047678687e-08 1.15789473684211
1.71194277118627e-08 1.21052631578947
1.78637506558567e-08 1.26315789473684
1.86080735998508e-08 1.31578947368421
1.93523965438448e-08 1.36842105263158
2.00967194878388e-08 1.42105263157895
2.08410424318329e-08 1.47368421052632
2.15853653758269e-08 1.52631578947368
2.23296883198209e-08 1.57894736842105
2.3074011263815e-08 1.63157894736842
2.3818334207809e-08 1.68421052631579
2.4562657151803e-08 1.73684210526316
2.53069800957971e-08 1.78947368421053
2.60513030397911e-08 1.84210526315789
2.67956259837851e-08 1.89473684210526
2.75399489277791e-08 1.94736842105263
2.82842717303518e-08 1.99999999
};
\end{axis}

\end{tikzpicture}}
\caption{Control points (red) and boundaries of the square and curved pipe mesh.}
\end{figure}
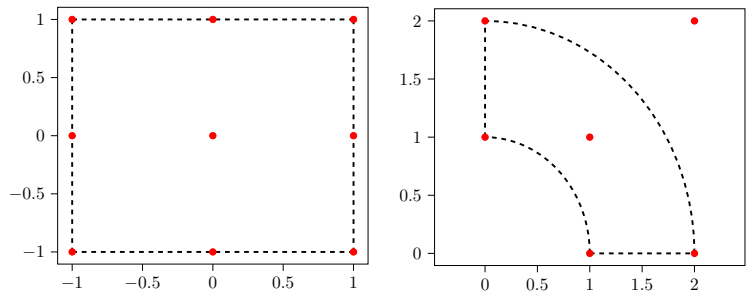

\section{Isogeometric Analysis}

\subsection{Overview}

Non-Uniform Rational B-Splines (NURBS) have been a standard tool for describing and modeling curves and surfaces in CAD programs. Isogeometric analysis uses NURBS as an analysis tool, proposed by Hughes \textit{et al} \cite{hughes2005isogeometric}. In this section, we present a short description of the isogeometric collocation method (IGA-C). For details, we refer the readers to \cite{auricchio2010isogeometric,cottrell2009isogeometric,hughes2010efficient,bazilevs2006isogeometric,nguyen2015isogeometric}. 

We emphasize that the isogeometric collocation method is not the only way to implement the physics constrained learning. We use IGA-C here since it has smooth basis functions and thus allows us to work with strong forms of the PDE directly. It has also been shown effective for applications in solid mechanics \cite{hughes2005isogeometric}, fluid dynamics \cite{nielsen2011discretizations}, etc. The physics constrained learning can also be applied to finite element analysis and other numerical discretization methods.

\subsection{B-splines and NURBS}

B-splines are piecewise polynomials curves, which are composed of linear combinations of B-spline basis functions. There are three components of $B$-splines:
\begin{enumerate}
    \item \textit{Control points}. Control points (denoted by $\bB_i$) are points in the plane. Control points affect the shape of B-spline curves but are not necessarily on the curve. 
    \item \textit{Knot vector}. A knot vector is a set of non-decreasing real numbers representing coordinates in the parametric space
    \begin{equation*}
        \{\xi_1=0, \ldots, \xi_{n+p+1}=1\}
    \end{equation*}
    where $p$ is the degree of the B-spline and $n$ is the number of basis functions, which is equal to the number of control points. A knot vector is uniform if the knots are uniformly-spaced and non-uniform otherwise. 
    \item \textit{B-spline basis functions}. B-spline basis functions are defined recursively by the knot vector. For $p=0$
    \begin{equation*}
        N_{i,0}(\xi) = \begin{cases}
            1 & \mbox{ if } \xi_i\leq \xi < \xi_{i+1}\\
            0 & \mbox{ otherwise }
        \end{cases}
    \end{equation*}
    For $p>1$, we have
    \begin{equation*}
        N_{i,p} = \begin{cases}
            \frac{\xi-\xi_i}{\xi_{i+p}-\xi_i}N_{i, p-1}(\xi) +\frac{\xi_{i+p+1}-\xi}{\xi_{i+p+1}-\xi_{i+1}}N_{i+1, p-1}(\xi)&\mbox{ if } \xi\in [\xi_i, \xi_{i+p+1})\\
            0 & \mbox{ otherwise }
        \end{cases}
    \end{equation*}
\end{enumerate}

The basis functions of a 2D B-spline region $\Omega$ can be constructed by tensor product of 1D B-spline basis functions. Let $\bB_{i,j}$, $i=1,2,\ldots, n+p+1$, $j=1,2,\ldots,m+q+1$ be the control points, where $n$, $m$ are the number of basis functions per dimension, and $p$, $q$ are degrees of the B-splines perspectively, the region can be represented as a map from the parametric space $[0,1]^2$ to the physical space $\Omega$
\begin{equation*}
    S(\xi, \eta) = \sum_{i=1}^n \sum_{j=1}^m N_{i,p}(\xi, \eta)M_{j,p}(\xi, \eta) \bB_{i,j}
\end{equation*}
here we use $N$ and $M$ to denote the basis functions for each dimension.

Analogously to B-splines, the basis functions for NURBS in a 2D domain are defined as
\begin{equation*}
    R_{i,j}^{p,q}(\xi, \eta) = \frac{N_{i,p}(\xi)M_{j,q}(\eta)w_{i,j}}{\sum_{i=1}^n\sum_{j=1}^m N_{i,p}(\xi)M_{j,q}(\eta)w_{i,j}}
\end{equation*}
where $w_{i,j}$ are called \textit{weights} of the NURBS surface. Hence the NURBS mapping is
\begin{equation*}
    R(\xi, \eta) = \sum_{i=1}^n \sum_{j=1}^m R_{i,j}^{p,q}(\xi, \eta) \bB_{i,j}
\end{equation*}

\subsection{Isogeometric Collocation Method}

In the isogeometric collocation method, the solution to the PDE is represented as a linear combination of NURBS basis functions
\begin{equation*}
    u_h(\bx) =  \sum_{i=1}^n \sum_{j=1}^m c_{i,j} R_{i,j}^{p,q}(\xi, \eta) \quad \bx = \sum_{i=1}^n \sum_{j=1}^m R_{i,j}^{p,q}(\xi, \eta) \bB_{i,j}
\end{equation*}
here $c_{i,j}\in \RR$ are the coefficients of the basis functions. The smoothness of $u_h(\bx)$ depends on the continuity of the basis function $R_{i,j}$, which, in turn, depends on the continuity of $\{N_{i,p}\}$, $\{M_{j,p}\}$. $N_{i,p}$ is smooth between knots in the knot vector and $C^{p-k}$ where $k$ is the multiplicity of $u_l$ at the knot $u_l$; the same is true for $M_{j,p}$. Therefore, we can make $u_h(x)$ arbitrarily smooth by selecting appropriate knot vectors and degrees. 

Since $u_h(x)$ can be made continuously differentiable up to any order, we can work directly with the strong form of PDEs. Consider a linear boundary-value problem ($\mathcal{P}$ and $\mathcal{B}$ are linear differential operators)
\begin{equation}\label{equ:system}
    \begin{aligned}
        \mathcal{P}u &=f &\mbox{ in } \Omega\\
        \mathcal{B}u &=g &\mbox{ on } \Omega
    \end{aligned}
\end{equation}
The isogeometric collocation method solves the problem by finding the coefficients $c_{i,j}$ that satisfies
\begin{equation}\label{equ:col}
    \begin{aligned}
        \sum_{i=1}^n \sum_{j=1}^m c_{i,j}\mathcal{P}\Big(R_{i,j}^{p,q}\Big)(\xi_l,\eta_l) &=f(R(\xi_l, \eta_l)), & l \in \mathcal{I}_P \\
        \sum_{i=1}^n \sum_{j=1}^m c_{i,j}\mathcal{B}\Big(R_{i,j}^{p,q}\Big)(\xi_l,\eta_l) &=g(R(\xi_l, \eta_l)), & l\in \mathcal{I}_B
    \end{aligned}
\end{equation}
here $\{(\xi_l,\eta_l)\}$ are tensor products of the Greville abscissae, and $\mathcal{I}_P$, $\mathcal{I}_B$ refer to the indices corresponding to inner and boundary points. $\mathcal{P}\Big(R_{i,j}^{p,q}\Big)$ and $\mathcal{B}\Big(R_{i,j}^{p,q}\Big)$ can be precomputed with efficient De Boor's algorithm. \Cref{equ:col} leads to a sparse linear system 
\begin{equation*}
    \mathbf{A}\mathbf{c} = \mathbf{f} \quad \mathbf{c}=\begin{bmatrix}
        c_{11}\\
        c_{12}\\
        \vdots\\
        c_{1n}\\
        c_{21}\\
        \vdots\\
        c_{mn}
    \end{bmatrix}\quad \mathbf{f}=\begin{bmatrix}
        f(R(\xi_1, \eta_1))\\
        f(R(\xi_1, \eta_2))\\
        \vdots\\
        f(R(\xi_1, \eta_n))\\
        f(R(\xi_2, \eta_1))\\
        \vdots\\
        f(R(\xi_m, \eta_n))
    \end{bmatrix}
\end{equation*}
where $\mathbf{A}$ is a sparse matrix. When \Cref{equ:system} is nonlinear, the Newton--Raphson algorithm is used to find the coefficients.

\section{Automatic Differentiation}

We give a short introduction to automatic differentiation. We do not attempt to cover the details in automatic differentiation since there is extensive literature (see \cite{baydin2018automatic} for a comprehensive review) on this topic; instead, we only discuss the relevant techniques associated with physics constrained learning. 

We describe the reverse mode automatic differentiation. Assume we are given inputs
\[ \{x_1, x_2, \ldots, x_n\} \]
and the algorithm produces a single output $x_N$, $N>n$. The gradients
\[ \frac{\partial x_N(x_1, x_2, \ldots, x_n)}{\partial x_i} \]
$i=1$, $2$, $\ldots$, $n$ are queried. 

The idea is that this algorithm can be decomposed into a sequence of functions $f_i$ ($i=n+1, n+2, \ldots, N$) that can be easily differentiated analytically, such as addition, multiplication, or basic functions like exponential, logarithm and trigonometric functions. Mathematically, we can formulate it as
\begin{equation*}
  \begin{aligned}
    x_{n+1} &= f_{n+1}(\bx_{\pi({n+1})})\\
    x_{n+2} &= f_{n+2}(\bx_{\pi({n+2})})\\
    \ldots\\
    x_{N} &= f_{N}(\bx_{\pi({N})})\\
\end{aligned}
\end{equation*}
where $\bx = \{x_i\}_{i=1}^N$ and $\pi(i)$ are the parents of $x_i$, s.t., $\pi(i) \in \{1,2,\ldots,i-1\}$.

The idea to compute $\partial x_N / \partial x_i$ is to start from $i = N$, and establish recurrences to calculate derivatives with respect to $x_i$ in terms of derivatives with respect to $x_j$, $j >i$. To define these recurrences rigorously, we need to define different functions that differ by the choice of independent variables.

The starting point is to define $x_i$ considering all previous $x_j$, $j < i$, as independent variables. Then:
\[ x_i(x_1, x_2, \ldots, x_{i-1}) = f_i(\bx_{\pi(i)}) \]
Next, we observe that $x_{i-1}$ is a function of previous $x_j$, $j < i-1$, and so on; so that we can recursively define $x_i$ in terms of fewer independent variables, say in terms of $x_1$, \dots, $x_k$, with $k < i-1$. This is done recursively using the following definition:
\begin{equation*}
    x_i(x_1, x_2, \ldots, x_j) = x_i(x_1, x_2, \ldots, x_j, f_{j+1}(\bx_{\pi(j+1)})), \quad n < j+1 < i
\end{equation*}
Observe that the function of the left-hand side has $j$ arguments, while the function on the right has $j+1$ arguments. This equation is used to ``reduce'' the number of arguments in $x_i$.

With these definitions, we can define recurrences for our partial derivatives which form the basis of the back-propagation algorithm. The partial derivatives for
\[ x_N(x_1, x_2, \ldots, x_{N-1}) \]
are readily available since we can differentiate
\[ f_N(\bx_{\pi(N)}) \]
directly. The problem is therefore to calculate partial derivatives for functions of the type $x_N(x_1, x_2, \ldots, x_i)$ with $i<N-1$. This is done using the following recurrence:
\[
    \frac{\partial x_N(x_1, x_2, \ldots, x_{i})}{\partial x_i} = \sum_{j\,:\,i\in \pi(j)}
    \frac{\partial x_N(x_1, x_2, \ldots, x_j)}{\partial x_j}
    \frac{\partial x_j(x_1, x_2, \ldots, x_{j-1})}{\partial x_i}
\]
with $n < i< N-1$. Since $i \in \pi(j)$, we have $i < j$. So we are defining derivatives with respect to $x_i$ in terms of derivatives with respect to $x_j$ with $j > i$. The last term
\[ \frac{\partial x_j(x_1, x_2, \ldots, x_{j-1})}{\partial x_k} \]
is readily available since:
\[ x_j(x_1, x_2, \ldots, x_{j-1}) = f_j(\bx_{\pi(j)}) \]
The computational cost of this recurrence is proportional to the number of edges in the computational graph (excluding the nodes $1$ through $n$), assuming that the cost of differentiating $f_k$ is $O(1)$. The last step is defining
\[
    \frac{\partial x_N(x_1, x_2, \ldots, x_n)}{\partial x_i} = \sum_{j\,:\,i\in \pi(j)}
    \frac{\partial x_N(x_1, x_2, \ldots, x_j)}{\partial x_j}
    \frac{\partial x_j(x_1, x_2, \ldots, x_{j-1})}{\partial x_i}
\]
with $1 \le i \le n$. Since $n < j$, the first term
\[ \frac{\partial x_N(x_1, x_2, \ldots, x_j)}{\partial x_j} \]
has already been computed in earlier steps of the algorithm. The computational cost is equal to the number of edges connected to one of the nodes in $\{1, \dots, n\}$.

We can see that the complexity of the back-propagation is bounded by that of the forward step, up to a constant factor. Reverse mode differentiation is very useful in the penalty method, where the loss function is a scalar, and no other constraints are present. 

\section{Automatic Jacobian Calculation}\label{sect:ace}

When the constraint $F_h(\theta, u_h) = 0$ is nonlinear, we need extra effort to compute the Jacobian 
\[ \frac{{\partial {F_h}}}{{\partial {u_h}}}(\theta , u_h)\Bigg|_{u_h = {G_h}(\theta )} \]
for both the forward modeling (e.g., the Jacobian is used in Newton-Raphson iteration for solving nonlinear systems) and computing the gradient \Cref{equ:s2}. \revise{This should not be confused with reverse mode automatic differentiation: the technique described in this section is only used for computing sparse Jacobians; when computing the gradient \Cref{equ:s2}, on top of the Jacobians, the reverse mode automatic differentiation is applied to extract gradients (\ref{sect:trick}).} 
For efficiency, it is desired that we can exploit the sparsity of the Jacobian term.  

In this section, we describe a method for automatically computing the sparse Jacobian matrix. This method is not restricted to the numerical methods we use in our numerical examples (finite difference and isogeometric analysis) and can be extended to most commonly used numerical schemes such as finite element analysis and finite volume methods.

The technique used here is similar to \textit{forward} mode automatic differentiation. The difference is that we pass the \textit{sparse} Jacobian matrices in the computational graph instead of gradient vectors. The locality of differential operators is the key to preserving the sparsity of those Jacobian matrices. The major three steps of our method for computing Jacobians automatically are 
\begin{enumerate}
    \item Building the computational graph; 
    \item For each operator in the computational graph, implementing the differentiation rule with respect to its input;
    \item Pushing the sparse Jacobians through the computational graph with chain rules in the same order as the operators of the forward modeling. 
\end{enumerate}

We consider a specific example to illustrate our algorithm. Assume that the numerical solution is represented by a linear combination of basis functions such as isogeometric analysis, $u_{\mathbf{c}}$, where $\mathbf{c}$ is the coefficients. Let the physical constraint and its discretization be (we omit parameters $\theta$ here since it is irrelevant) 
\[F(u) = \nabla \cdot((1+u^2) \nabla u) \qquad F_h(u_\mathbf{c}) = \nabla \cdot( (1+u_\mathbf{c}^2)\nabla u_\mathbf{c}  ) \]
We want to compute the Jacobian matrix 
\begin{equation}
	J(\mathbf{c}) = \frac{\partial F_h(u_\mathbf{c})}{\partial \mathbf{c}}
\end{equation}

Now we decompose the operator into several individual operators in the computational graph (\Cref{fig:graph})
\begin{align*}
	F_0(\mathbf{c}) =&\; u_\mathbf{c} = M\mathbf{c}\\
	F_1(\mathbf{c}) =&\; \nabla F_0(\mathbf{c}) = [M_xF_0(\mathbf{c})\ M_yF_0(\mathbf{c})] \\
	F_2(\mathbf{c}) =&\; F_0(\mathbf{c})^2 \\
	F_3(\mathbf{c}) =&\; 1 + F_2(\mathbf{c}) \\
	F_4(\mathbf{c}) =&\; F_3(\mathbf{c}) F_1(\mathbf{c}) \\
	F_5(\mathbf{c}) =&\; \nabla \cdot F_4(\mathbf{c})
\end{align*}
here $M$ is the coefficient matrix for numerical representation of the solution $u$, $M_x$ and $M_y$ are coefficient matrices for the numerical representation of $\frac{\partial}{\partial x}$ and $\frac{\partial}{\partial y}$ associated with the selected basis functions (they can be obtained analytically by differentiating the basis functions) and all the operations denote component-wise operations, for example,
\begin{align*}
	[F_0(\mathbf{c})^2]_i =&\; F_0(\mathbf{c})_i^2\\
	[\nabla F_0(\mathbf{c})]_{ik} =&\; \nabla [\big(F_0(\mathbf{c})_i\big)]_k\\
	[F_3(\mathbf{c}) F_1(\mathbf{c})]_{ik} =&\; F_3(\mathbf{c})_iF_1(\mathbf{c})_{ik}\quad k=1,2\\
	[\nabla \cdot F_4(\mathbf{c})]_i =&\; \nabla\cdot [F_4(\mathbf{c})_{i1} \ F_4(\mathbf{c})_{i2}]
\end{align*}

\begin{figure}[hbt]
\centering
  \includegraphics[width=0.8\textwidth]{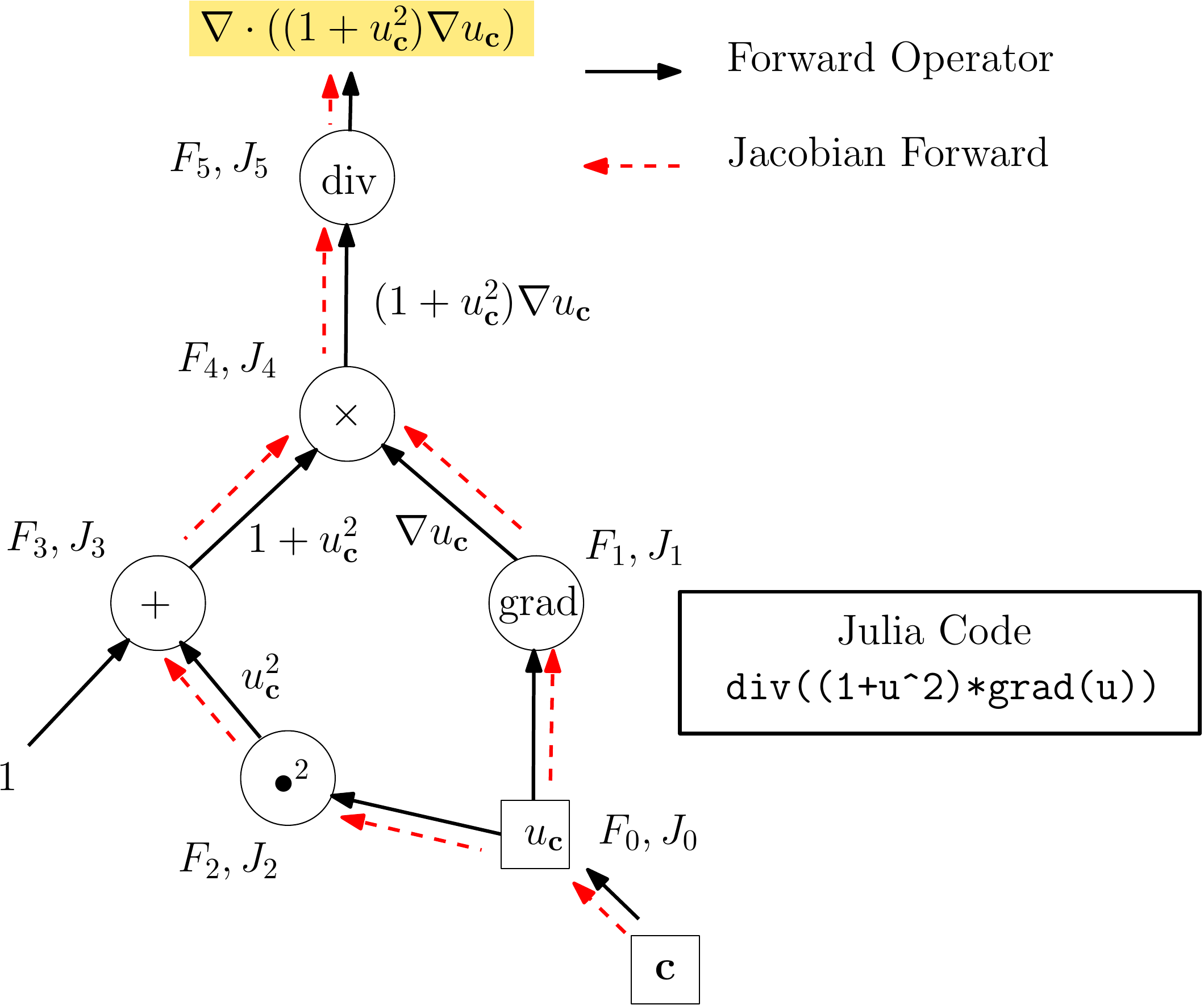}
  \caption{Jacobian forward propagation in the computational graph. The Jacobian computation is automatic: the users only need to specify the forward operations.}
  \label{fig:graph}
\end{figure}

The corresponding differentiation rule for each operator is 
\begin{align*}
	{J_0(\mathbf{c})} =&\; \frac{\partial}{\partial \mathbf{c}}F_0(\mathbf{c}) = M\\
	{J_1(\mathbf{c})} =&\; \frac{\partial}{\partial \mathbf{c}} F_1(\mathbf{c}) = [M_x{J_0(\mathbf{c})} \ M_y{J_0(\mathbf{c})}] \\
	{J_2(\mathbf{c})} = &\; \frac{\partial}{\partial \mathbf{c}} F_2(\mathbf{c}) =2\mathrm{diag}(F_0(\mathbf{c})) {J_0(\mathbf{c})} \\
	{J_3(\mathbf{c})} =&\; \frac{\partial}{\partial \mathbf{c}} F_3(\mathbf{c}) ={J_2(\mathbf{c})} \\
	{J_4(\mathbf{c})} = &\; \frac{\partial}{\partial \mathbf{c}} F_4(\mathbf{c}) = 2\mathrm{diag}(F_3(\mathbf{c})) {J_1(\mathbf{c})} \\
	& + [\mathrm{diag}(M_xF_0(\mathbf{c})){J_3(\mathbf{c})} \ \mathrm{diag}(M_yF_0(\mathbf{c})){J_3(\mathbf{c})}]\\
     {J_5(\mathbf{c})} =&\; \frac{\partial}{\partial \mathbf{c}} F_5(\mathbf{c}) = {J_4(\mathbf{c})}\begin{bmatrix}
		M_x\\
		M_y
	\end{bmatrix}
\end{align*}
Here $\mathrm{diag}(\mathbf{v})$ denotes the diagonal matrix whose diagonal entries are $\mathbf{v}$. By chaining together all the differentiation rules, $J_5(\mathbf{c}) = J(\mathbf{c})$ is the desired Jacobian. We can see the dependency of $J_k(\mathbf{c})$, $k=1$, \dots, $5$ is the same as the forward modeling, which also explains why the dashed red arrows (Jacobian Forward) in \Cref{fig:graph} parallel the forward computation. The order of the computation is different from reverse mode automatic differentiation in \Cref{equ:s2}. Reverse mode automatic differentiation is most efficient when the final output is a scalar value. $F(u, \theta)$ is a multiple input multiple output operator and therefore computing $\frac{\partial F(u, \theta)}{\partial u}$ using reverse mode automatic differentiation is inefficient. Nevertheless, $F(u, \theta)$ is composed of differential operators and our representation of the solution $u$ consists of local basis functions. This locality enables us to obtain an efficient automatic Jacobian calculation algorithm that preserves the sparsity.

\section{Automatic Jacobian Calculation Implementation in\\\texttt{IGACS.jl}}\label{sect:detail}

It is useful to have an understanding of how we implement automatic Jacobian calculation in \texttt{IGACS.jl}. The implementation can also be extended to other basis functions or weak formulations such as finite difference method, finite element method, and finite volume method. It is also possible to extend to global basis functions by back-propagating special structured matrices instead of sparse Jacobian matrices. 

The automatic Jacobian calculation implementation is based on\\\texttt{ADCME.jl}~\cite{adcme}, which leverages \texttt{TensorFlow} \cite{abadi2016tensorflow_proc,abadi2016tensorflow_art} for graph optimization and numerical acceleration (e.g., XLA, CUDA). The computational graph parsing and bookkeeping are done in \texttt{Julia}~\cite{DBLP:journals/corr/abs-1209-5145}, which does not affect runtime performance. \texttt{ADCME.jl} also implements a sparse linear algebra custom operator library to augment the built-in \texttt{tf.sparse} in \texttt{TensorFlow}. In sum, we have made our best effort to attain high performance of \texttt{IGACS.jl} by leveraging \texttt{TensorFlow} and customizing performance critical computations. 

One important concern for designing \texttt{IGACS.jl} is the accessibility to users. We hide the details of gradients back-propagation or Jacobian for\-ward-propagation by abstraction via the structure \texttt{Coefficient}, which holds the data of the NURBS representation and tracing information in the computational graph. For example, to compute $J(c) = \nabla \cdot ((1+c^2)\nabla c)$, users can simply write
\[
\texttt{J} = \texttt{div}((1+\texttt{c}^2)*\texttt{grad}(\texttt{c}))    
\]
This will build the computational graph automatically. For computing Jacobian, users only need to write
\[
\texttt{deriv}(\texttt{J}, \texttt{c})
\]
This will parse the dependency of the computational graph and return a sparse matrix (\texttt{SparseTensor} struct in \texttt{ADCME}) at the Greville abscissae (other collocation points can be specified by passing an augment). 

As an example, let the loss function be
\begin{equation*}
    L(\mu) =\int_{\Omega} \nabla \cdot ((1+\mu c^2)\nabla c)d\bx,\quad \mu\in\RR
\end{equation*}
To compute the gradient $\frac{\partial L}{\partial \mu}$, we need to write
\begin{align*}
&\texttt{J} = \texttt{div}((1+\mu * \texttt{c}^2)*\texttt{grad}(\texttt{c}))    \\
&\texttt{g} = \texttt{gradients}(\texttt{J}, \mu)
\end{align*}

\section{Computing the Gradient with AD}\label{sect:trick}

Having introduced the technical tools AD and automatic Jacobian calculation, we can describe the numerical procedure for computing \Cref{equ:s2}
\begin{equation}\label{equ:s3}
    \frac{{\partial {{\tilde L}_h}(\theta )}}{{\partial \theta }} =  - \frac{{\partial {L_h}({u_h})}}{{\partial {u_h}}}{\left( {\frac{{\partial {F_h}}}{{\partial {u_h}}}\Bigg|_{u_h = {G_h}(\theta )}} \right)^{ - 1}}\frac{{\partial {F_h}}}{{\partial \theta }}(\theta ,{G_h}(\theta ))
\end{equation}
For efficiency, we do \textit{not} compute 
\[ \frac{{\partial {L_h}({u_h})}}{{\partial {u_h}}},\qquad{\left( {\frac{{\partial {F_h}}}{{\partial {u_h}}}\Bigg|_{u_h = {G_h}(\theta )}} \right)^{ - 1}},\qquad\frac{{\partial {F_h}}}{{\partial \theta }}(\theta ,{G_h}(\theta )) \]
separately. Instead, we take advantage of reverse mode automatic differentiation and sparse linear algebra. 

First, we compute $\frac{{\partial {L_h}({u_h})}}{{\partial {u_h}}}$. Since $L_h:\RR^n\rightarrow \RR$, we can apply the reverse mode automatic differentiation to compute the gradients.

Next, computing
\[ 
    x = \frac{{\partial {L_h}({u_h})}}{{\partial {u_h}}}{\Big( {\frac{{\partial {F_h}}}{{\partial {u_h}}}\Bigg|_{u_h = {G_h}(\theta )}} \Big)^{ - 1}} 
\]
is equivalent to solving a linear system
\begin{equation}\label{equ:sparse}
    x \; 
    \frac{{\partial {F_h}}}{{\partial {u_h}}}\Bigg|_{u_h = {G_h}(\theta )} 
    = \frac{{\partial {L_h}({u_h})}}{{\partial {u_h}}}
\end{equation}
The term
\[ 
    {\frac{{\partial {F_h}}}{{\partial {u_h}}}\Bigg|_{u_h = {G_h}(\theta )}}
\] 
is computed with automatic Jacobian calculation and is sparse. Therefore, we can solve the linear system \Cref{equ:sparse} with a sparse linear solver. 

Consequently, we have according to \Cref{equ:s3} (\revise{in principle, $x$ depends on $\theta$; but in the equation below, we treat $x$ as independent of $\theta$}) 
\begin{equation}\label{equ:s5}
    \frac{{\partial {{\tilde L}_h}(\theta )}}{{\partial \theta }}  = -x\frac{{\partial {F_h}}}{{\partial \theta }}(\theta , u_h)\Big|_{u_h = G_h(\theta)} = \frac{\partial \Big(-x F_h\big(\theta, u_h\big)\Big)}{\partial \theta}\Bigg|_{u_h = G_h(\theta)}
\end{equation}
Note that 
\begin{equation*}
    \theta \mapsto  -x^T F_h(\theta, u_h)
\end{equation*}
is a mapping from $\RR^d$ to $\RR$, we can again apply the reverse mode automatic differentiation to compute the gradients \Cref{equ:s2}. 

The following code snippet shows a possible implementation of the discussion above
\begin{verbatim}
    l  = L(u)
    r  = F(theta, u)
    g  = gradients(l, u)
    x  = dF'\g
    x  = independent(x)
    dL = -gradients(sum(r*x), theta)
\end{verbatim}
here \texttt{dF} is the sparse Jacobian computed with automatic Jacobian calculation. \revise{Here {\tt independent} is the programmatic way of treating $x$ as independent of $\theta$.} 

\begin{figure}[hbt]
    \centering
    \includegraphics[width=0.8\textwidth]{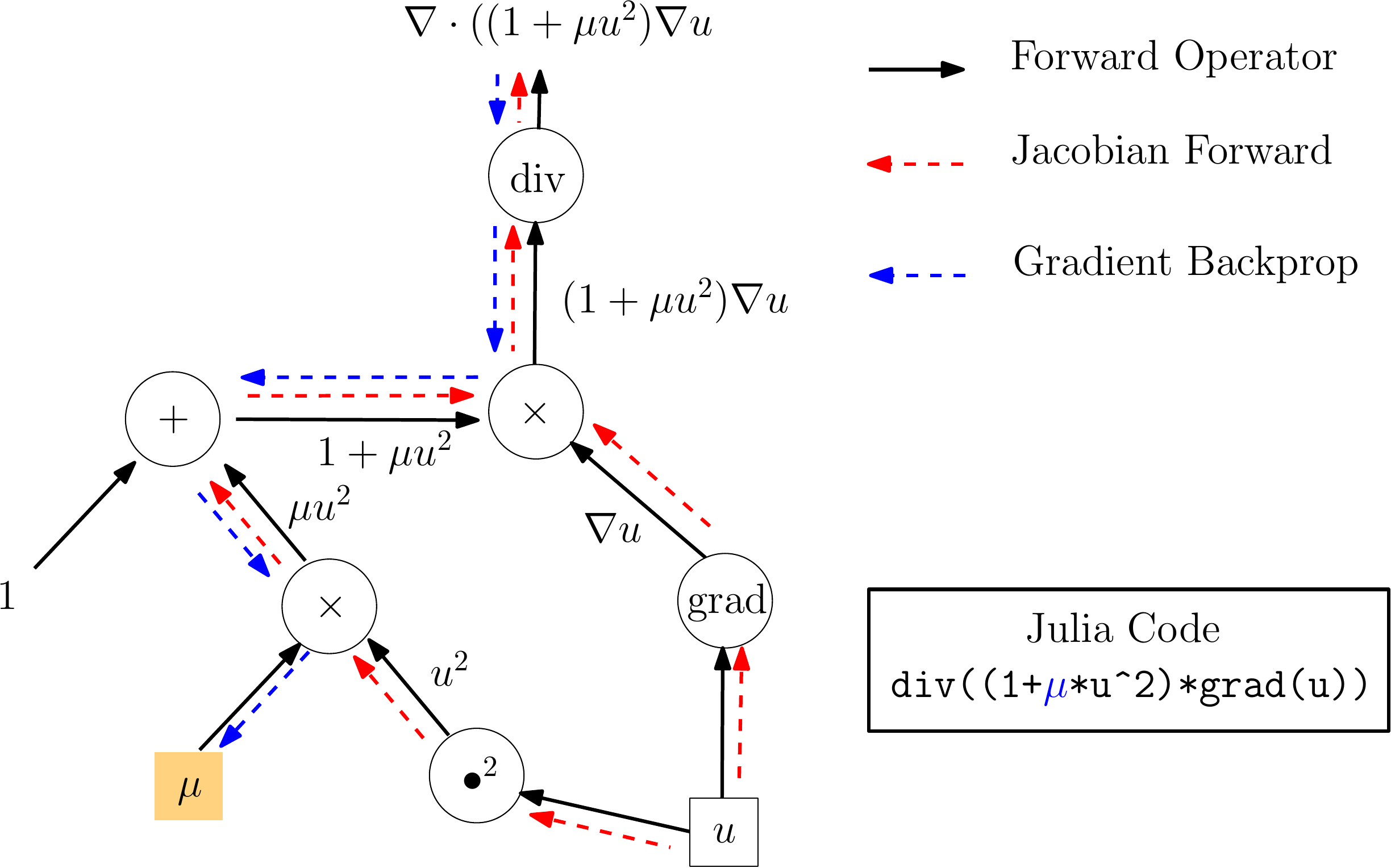}
    \caption{Forward-Backward pattern in constraint enforcement method. In each iteration, a forward simulation is performed first (black and red arrows). A gradient back-propagation (blue arrows) follows and updates the unknown parameter $\mu$.}
    \label{fig:forward-backward}
\end{figure}

In the inverse modeling optimization process, for each iteration, we need a forward simulation where we populate each edge with intermediate data and Jacobian matrices. Then a gradient back-propagation follows and updates the unknown parameters (\Cref{fig:forward-backward}). This forward-backward pattern also exists in the penalty method, but no Jacobian information is required.

\end{appendix}

\newpage

\bibliographystyle{unsrt}
\bibliography{IMSO}

\end{document}